\documentclass[12pt,oneside,notitlepage]{book} %insert 12pt before oneside for larger type
\usepackage{amssymb,latexsym,amsmath,graphics,setspace}
\usepackage{epsfig,epsf,amsthm,amsfonts}

\newtheorem{theorem}{Theorem}
\newtheorem{conjecture}[theorem]{Conjecture}
\newtheorem{corollary}[theorem]{Corollary}
\newtheorem{definition}[theorem]{Definition}

\newtheorem{lemma}[theorem]{Lemma}

\newtheorem{proposition}[theorem]{Proposition}

\newtheorem{remark}[theorem]{Remark}

\newcommand{\one}{{{\rm 1\mkern-1.5mu}\!{\rm I}}}
\numberwithin{equation}{chapter}   %equations numbered by chapter
\begin{document}
%start with frontmatter
\frontmatter

%titlepage
\pagestyle{empty}
\begin{center}
\rule{165pt}{0pt} \\
\vspace{1.5cm}
\LARGE{Large Deviations for Random Walk in a Random Environment}\\
\vspace{1.5cm}
\normalsize{by} \\
\vspace{1.5cm}
\large{Atilla Y\i lmaz} \\
\vspace{2.5cm}
\normalsize{A dissertation submitted in partial fulfillment} \\
\normalsize{of the requirements for the degree of} \\
\normalsize{Doctor of Philosophy}  \\
\normalsize{Department of Mathematics} \\
\normalsize{New York University} \\
\normalsize{September 2008} \\
\vspace{1.0cm}
\end{center}
\begin{flushright}
\vspace{.6cm}
\rule{145pt}{.6pt} \\
\normalsize{S.\ R.\ S.\ Varadhan}\\
\vspace{.8cm}
\end{flushright}
\pagebreak

% 2 copyright page
%\pagestyle{empty}             %dont number this page
%\setcounter{page}{0}     %dont start counting pages yet
%\vspace*{5cm}
%\begin{center}
%\copyright \quad Atilla Y\i lmaz\\
%All Rights Reserved, 2008
%\end{center}
%\pagebreak

% 3 blankpage
\pagestyle{empty}           %dont number this page
\qquad \pagebreak

% 3 acknowledgements page
\setcounter{page}{3}
\pagestyle{plain}
\addcontentsline{toc}{section}{Acknowledgements}
\begin{center}
\huge{\textbf{Acknowledgements}}\\
\end{center}
\vspace{1cm}
\par
\normalsize{My professors at Bo\u{g}azi\c{c}i University, Istanbul, provided me with a solid undergraduate education. I especially thank A.\ Eden, A.\ Feyzio\u{g}lu and K.\ \"Oz\c{c}ald\i ran for their constant support and valuable advice. In particular, it was A.\ Eden who strongly encouraged me to come to the Courant Institute for my graduate studies.

I am indebted to T.\ Arnon, G.\ Ben Arous, D.\ Cai, S.\ G\"unt\"urk, H.\ McKean, T.\ Suidan, N.\ Zygouras and many other members of the Courant community for their help during the five years I have been here.

Working with S.\ R.\ S.\ Varadhan has been an absolutely amazing experience. He has warmly welcomed me whenever I've shown up at his office door, patiently listened to my often too long presentations, answered my questions, clarified my understanding of many fundamental concepts, taught me numerous techniques, and generously suggested new ideas. I am very fortunate to be his student. He has been and will always be a role model for me both as a person and as a mathematician.

Finally, I thank F.\ Rassoul-Agha for his detailed and constructive comments on the three papers that contain my results in this dissertation, and O.\ Zeitouni for his hospitality during my short visit to the Weizmann Institute.}
\pagebreak

% 4 abstract page
\setcounter{page}{4}
\pagestyle{plain}
\addcontentsline{toc}{section}{Abstract}
\begin{center}
\huge{\textbf{Abstract}}\\
\end{center}
\vspace{1cm}
\par
\normalsize{In this work, we study the large deviation properties of random walk in a random environment on $\mathbb{Z}^d$ with $d\geq1$.

We start with the quenched case, take the point of view of the particle, and prove the large deviation principle (LDP) for the pair empirical measure of the environment Markov chain. By an appropriate contraction, we deduce the quenched LDP for the mean velocity of the particle and obtain a variational formula for the corresponding rate function $I_q$. We propose an Ansatz for the minimizer of this formula. This Ansatz is easily verified when $d=1$.

In his 2003 paper, Varadhan proves the averaged LDP for the mean velocity and gives a variational formula for the corresponding rate function $I_a$. Under the non-nestling assumption (resp.\ Kalikow's condition), we show that $I_a$ is strictly convex and analytic on a non-empty open set $\mathcal{A}$, and that the true velocity $\xi_o$ is an element (resp.\ in the closure) of $\mathcal{A}$. We then identify the minimizer of Varadhan's variational formula at any $\xi\in\mathcal{A}$.

For walks in high dimension, we believe that $I_a$ and $I_q$ agree on a set with non-empty interior. We prove this for space-time walks when the dimension is at least $3+1$. In the latter case, we show that the cheapest way to condition the asymptotic mean velocity of the particle to be equal to any $\xi$ close to $\xi_o$ is to tilt the transition kernel of the environment Markov chain via a Doob $h$-transform.}
\pagebreak

% 5 table of contents
\tableofcontents

%beginning of mainmatter
\mainmatter
%\include{intro} % if your intro is called intro.tex, for example
%\part{First part}% if you want separate parts
%\include{swe} % if your first body chapter is named swe.tex, etc...

\pagestyle{plain}
\addcontentsline{toc}{section}{Introduction}
\begin{center}
\huge{\textbf{Introduction}}\\
\end{center}
\vspace{1cm}
\par
\normalsize{The random motion of a particle on $\mathbb{Z}^d$ can be modelled by a discrete time Markov chain. Write $\pi(x,x+z)$ for the transition probability from $x$ to $x+z$ for each $x,z\in\mathbb{Z}^d$, and refer to $\omega_x:=(\pi(x,x+z))_{z\in\mathbb{Z}^d}$ as the ``environment" at $x$. If the environment $\omega:=(\omega_x)_{x\in\mathbb{Z}^d}$ is sampled from a probability space $(\Omega,\mathcal{B},\mathbb{P})$, then the particle is said to perform ``random walk in a random environment" (RWRE). Here, $\mathcal{B}$ is the Borel $\sigma$-algebra corresponding to the product topology.

For each $z\in\mathbb{Z}^d$, define the shift $T_z$ on $\Omega$ by $\left(T_z\omega\right)_x=\omega_{x+z}$, and assume that $\mathbb{P}$ is stationary and ergodic under $\left(T_z\right)_{z\in\mathbb{Z}^d}$. Further assume that the step sizes are bounded by a constant $B$, i.e., for any $z=(z_1,\ldots,z_d)\in\mathbb{Z}^d$, $\pi(0,z)=0$ $\mathbb{P}$-a.s.\ unless $0<|z_1|+\cdots+|z_d|\leq B$. Denote the set of allowed steps of the walk by \[\mathcal{R} := \{(z_1,\ldots,z_d)\in\mathbb{Z}^d:\;0<|z_1|+\cdots+|z_d|\leq B\}.\] The walk is said to be nearest-neighbor when $B=1$, and the set of allowed steps is then \[U:=\{(z_1,\ldots,z_d)\in\mathbb{Z}^d:\;|z_1|+\cdots+|z_d|=1\}.\]

For any $x\in\mathbb{Z}^d$ and $\omega\in\Omega$, the Markov chain with transition probabilities given by $\omega$ induces a probability measure $P_x^\omega$ on the space of paths starting at $x$. Statements about $P_x^\omega$ that hold for $\mathbb{P}$-a.e.\ $\omega$ are referred to as ``quenched". Statements about the semi-direct product $P_x:=\mathbb{P}\times P_x^\omega$ are referred to as ``averaged". Expectations under $\mathbb{P}, P_x^\omega$ and $P_x$ are denoted by $\mathbb{E}, E_x^\omega$ and $E_x$, respectively.

Because of the extra layer of randomness in the model, the standard questions of recurrence vs.\ transience, the law of large numbers (LLN), the central limit theorem (CLT) and the large deviation principle (LDP) --- which have well known answers for classical random walk --- become subtle. However, it is possible by taking the ``point of view of the particle" to treat the two layers of randomness as one: If we denote the random path of the particle by $X:=(X_n)_{n\geq0}$, then $(T_{X_n}\omega)_{n\geq0}$ is a Markov chain (referred to as ``the environment Markov chain") on $\Omega$ with transition kernel $\overline{\pi}$ given by \[\overline{\pi}(\omega,\omega'):=\sum_{z:\,T_z\omega=\omega'}\pi(0,z).\] This is a standard approach in the study of random media. See for example \cite{DeMasi}, \cite{KV}, \cite{Kozlov}, \cite{Olla} or \cite{PV}.

Instead of viewing the environment Markov chain as an auxiliary construction, one can introduce it first and then deduce the particle dynamics from it:
\begin{definition}\label{ortamkeli}
A function $\hat{\pi}(\cdot,\cdot):\Omega\times\mathcal{R}\to\mathbb{R}^+$ is said to be an ``environment kernel" if $\hat{\pi}(\cdot,z)$ is $\mathcal{B}$-measurable for each $z\in\mathcal{R}$ and if $\sum_{z\in\mathcal{R}}\hat{\pi}(\cdot,z)=1,\ \mathbb{P}$-a.s. It can be viewed as a transition kernel on $\Omega$ by the following identification: \[\hat{\pi}(\omega,\omega'):=\sum_{z:\,T_z\omega=\omega'}\hat{\pi}(\omega,z).\]

Given $x\in\mathbb{Z}^d$, $\omega\in\Omega$ and any environment kernel $\hat{\pi}$, the probability measure $P_x^{\hat{\pi},\omega}$ on the space of particle paths $(X_n)_{n\geq0}$ starting at $x$ is defined by setting $P_x^{\hat{\pi},\omega}\left(X_o=x\right)=1$ and $P_x^{\hat{\pi},\omega}\left(X_{n+1}=y+z\left|X_n=y\right.\right)=\hat{\pi}(T_y\omega,z)$ for all $n\geq0$, $y\in\mathbb{Z}^d$ and $z\in\mathcal{R}$. Expectations under $P_x^{\hat{\pi},\omega}$ and $P_x^{\hat{\pi}}:=\mathbb{P}\times P_x^{\hat{\pi},\omega}$ are denoted by $E_x^{\hat{\pi},\omega}$ and $E_x^{\hat{\pi}}$, respectively.
\end{definition}

See \cite{Sznitman} or \cite{Zeitouni} for a survey of results on RWRE. We study the large deviation properties of this model. Our results are taken from \cite{YilmazSpaceTime}, \cite{YilmazQuenched} and \cite{YilmazAveraged}.

Recall that a sequence $\left(Q_n\right)_{n\geq1}$ of probability measures on a topological space satisfies the LDP with rate function $I$ if:\\ $I$ is non-negative,  lower semicontinuous, and for any measurable set $G$, $$-\inf_{x\in G^o}I(x)\leq\liminf_{n\to\infty}\frac{1}{n}\log Q_n(G)\leq\limsup_{n\to\infty}\frac{1}{n}\log Q_n(G)\leq-\inf_{x\in\bar{G}}I(x).$$ Here, $G^o$ denotes the interior of $G$, and $\bar{G}$ its closure. See \cite{DZ} for general background and definitions regarding large deviations.}

\chapter{Statement of results}

\section{Quenched large deviations}\label{egridogru}

\subsection{Previous results}

In the case of nearest-neighbor RWRE on $\mathbb{Z}$, Greven and den Hollander \cite{GdH} assume that $\mathbb{P}$ is a product measure, and prove

\begin{theorem}[Quenched LDP]\label{qLDPgeneric}
For $\mathbb{P}$-a.e.\ $\omega$, $\left(P_o^\omega\left(\frac{X_n}{n}\in\cdot\,\right)\right)_{n\geq1}$ satisfies the LDP with a deterministic and convex rate function $I_q$.
\end{theorem}

\noindent They provide a formula for $I_q$ and show that its graph typically has flat pieces. Their proof makes use of an auxiliary branching process formed by the excursions of the walk. By a completely different technique, Comets, Gantert and Zeitouni \cite{CGZ} extend the results in \cite{GdH} to stationary and ergodic environments. Their argument involves first proving a quenched LDP for the passage times of the walk by an application of the G\"artner-Ellis theorem, and then inverting this to get the desired LDP for the mean velocity.

For $d\geq2$, the first result on quenched large deviations is given by Zerner \cite{Zerner}. He uses a subadditivity argument for certain passage times to prove Theorem \ref{qLDPgeneric} in the case of ``nestling" walks in product environments.
\begin{definition}\label{defnestling}
The nestling property is said to hold if the convex hull of the support of the law of $\sum_{z\in\mathcal{R}}\pi(0,z)z$ contains the origin. Otherwise, the walk is referred to as non-nestling.
\end{definition}
\noindent By a more direct use of the subadditive ergodic theorem, Varadhan \cite{Raghu} drops the nestling assumption and generalizes Theorem \ref{qLDPgeneric} to stationary and ergodic environments. The drawback of these approaches is that they don't lead to any formula for the rate function. 

In his Ph.D.\ thesis, Rosenbluth \cite{jeffrey} takes the point of view of the particle and gives an alternative proof of Theorem \ref{qLDPgeneric} in the case of stationary and ergodic environments. He provides a variational formula for the rate function $I_q$. Our results concerning quenched large deviations build on his approach.

\subsection{Our results}

For any measurable space $(Y,\mathcal{F})$, write $M_1(Y,\mathcal{F})$ (or simply $M_1(Y)$ whenever no confusion occurs) to denote the space of probability measures on $(Y,\mathcal{F})$. Consider random walk $X=(X_n)_{n\geq0}$ on $\mathbb{Z}^d$ in a stationary and ergodic random environment, and focus on \[\nu_{n,X} := \frac{1}{n}\sum_{k=0}^{n-1}\one_{T_{X_k}\omega,X_{k+1}-X_k}\] which is a random element of $M_1(\Omega\times\mathcal{R})$. The map $(\omega,z)\mapsto(\omega,T_z\omega)$ allows us to imbed $M_1(\Omega\times\mathcal{R})$ into $M_1(\Omega\times\Omega)$, and we therefore refer to $\nu_{n,X}$ as ``the pair empirical measure of the environment Markov chain".

Given any $\mu\in M_1(\Omega\times\mathcal{R})$, introduce the probability measures $(\mu)^1$ and $(\mu)^2$ on $\Omega$ by setting \[\mathrm{d}(\mu)^1(\omega):=\sum_{z\in\mathcal{R}}\mathrm{d}\mu(\omega,z)\qquad\mbox{and}\qquad\mathrm{d}(\mu)^2(\omega):=\sum_{z\in\mathcal{R}}\mathrm{d}\mu(T_{-z}\omega,z).\] In words, $(\mu)^1$ and $(\mu)^2$ are the marginals of $\mu$ when $\mu$ is seen as an element of $M_1(\Omega\times\Omega)$. With this notation, define
\begin{align*}
&M_{1,s}^{\ll}(\Omega\times\mathcal{R})\\&:= \left\{\mu\in M_1(\Omega\times\mathcal{R}): (\mu)^1=(\mu)^2\ll\mathbb{P},\ \frac{\mathrm{d}\mu(\omega,z)}{\mathrm{d}(\mu)^1(\omega)}>0\ \mathbb{P}\mbox{-a.s.\ for each }z\in U\right\}.
\end{align*}

Our main result is the following theorem whose proof constitutes Section \ref{pairLDPsection}.
\begin{theorem}\label{level2LDP} If there exists $\alpha>0$ such that
\begin{equation}
\int|\log \pi(0,z)|^{d+\alpha}\,\mathrm{d}\mathbb{P}<\infty\label{kimimvarki}
\end{equation} for each $z\in\mathcal{R}$, then
$\mathbb{P}$-a.s.\ $(P_o^\omega(\nu_{n,X}\in\cdot\,))_{n\geq1}$ satisfies the LDP. The rate function $\mathfrak{I}_q^{**}$ is the double Fenchel-Legendre transform of $\mathfrak{I}_q:M_1(\Omega\times\mathcal{R})\to\mathbb{R}^+$ given by
\begin{equation}
\mathfrak{I}_q(\mu)=\left\{
\begin{array}{ll}
\int_{\Omega}\sum_{z\in\mathcal{R}} \mathrm{d}\mu(\omega,z)\log\frac{\mathrm{d}\mu(\omega,z)}{\mathrm{d}(\mu)^1(\omega)\pi(0,z)}& \mbox{if }\mu\in M_{1,s}^{\ll}(\Omega\times\mathcal{R}),\\ \infty & \mbox{otherwise.}
\end{array}\right.\label{level2ratetilde}
\end{equation}
\end{theorem}
\begin{remark}
$\mathfrak{I}_q$ is convex but may not be lower semicontinuous, therefore $\mathfrak{I}_q^{**}$ is not a-priori equal to $\mathfrak{I}_q$.
\end{remark}

We start Section \ref{birboyutadonus} by deducing the quenched LDP for the mean velocity of the particle by an application of the contraction principle. For any $\xi\in\mathbb{R}^d$, define
\begin{align}
A_\xi&:=\{\mu\in M_1(\Omega\times\mathcal{R}):\xi_{\mu}=\xi\}\quad\text{where}\label{Axi}\\
\xi_{\mu}&:=\int\sum_{z\in\mathcal{R}}\mathrm{d}\mu(\omega,z)z\quad\text{for any }\mu\in M_1(\Omega\times\mathcal{R}).\label{ximu}
\end{align} The corollary below follows immediately from Theorem \ref{level2LDP} and reproduces the central result of \cite{jeffrey}. It is the most general version of Theorem \ref{qLDPgeneric} in the RWRE literature.

\begin{corollary}\label{level1LDP}
Under the assumption that there exists $\alpha>0$ such that (\ref{kimimvarki}) holds for each $z\in\mathcal{R}$, $(P_o^\omega(\frac{X_n}{n}\in\cdot\,))_{n\geq1}$ satisfies the LDP for $\mathbb{P}$-a.e.\ $\omega$. The rate function $I_q$ is given by
\begin{eqnarray}
I_q(\xi)&=&\inf_{\mu\in A_\xi} \mathfrak{I}_q^{**}(\mu)\label{level1rate}\\
&=&\inf_{\mu\in A_\xi} \mathfrak{I}_q(\mu)\label{level1ratetilde}
\end{eqnarray} where $\mathfrak{I}_q$ and $A_\xi$ are defined in (\ref{level2ratetilde}) and (\ref{Axi}), respectively. $I_q$ is convex.
\end{corollary}

One would like to get a more explicit expression for the rate function $I_q$. This is not an easy task in general. $M_1(\Omega\times\mathcal{R})$ is compact (when equipped with the weak topology), $A_\xi$ is closed and $\mathfrak{I}_q^{**}$ is lower semicontinuous, therefore the infimum in (\ref{level1rate}) is attained. However, due to the possible lack of lower semicontinuity of $\mathfrak{I}_q$, the infimum in (\ref{level1ratetilde}) may not be attained. Below, we propose an Ansatz and show that whenever an element of $A_\xi$ fits this Ansatz, it is the unique minimizer of (\ref{level1ratetilde}). Let us start by defining a class of functions.
\begin{definition} \label{K}
A measurable function $F:\Omega\times\mathcal{R}\rightarrow\mathbb{R}$ is said to be in class $\mathcal{K}$ if it satisfies the following conditions:
\begin{description}
\item[Moment.] For each $z\in\mathcal{R}$, $F(\cdot,z)\in\bigcup_{\alpha>0}L^{d+\alpha}(\mathbb{P})$.
\item[Mean zero.] For each $z\in\mathcal{R}$, $\mathbb{E}\left[F(\cdot,z)\right]=0$.
\item[Closed loop.] For $\mathbb{P}$-a.e.\ $\omega$, and any $(x_{k})_{k=0}^n$ with $x_0=x_n$ and $x_{k+1}-x_k\in\mathcal{R}$, \[\sum_{k=0}^{n-1}F(T_{x_k}\omega,x_{k+1}-x_k)=0.\]
\end{description}
\end{definition}
\noindent The following lemma provides the aforementioned Ansatz for the unique minimizer of (\ref{level1ratetilde}). Its proof concludes Section $\ref{birboyutadonus}$.

\begin{lemma}\label{lagrange}
For any $\xi\in\mathbb{R}^d$, if there exists $\mu_\xi\in A_\xi\cap M_{1,s}^{\ll}(\Omega\times\mathcal{R})$ such that \[\mathrm{d}\mu_\xi(\omega,z)=\mathrm{d}(\mu_\xi)^1(\omega) \pi(0,z)\mathrm{e}^{\langle\theta,z\rangle+F(\omega,z)+ r}\] for some $\theta\in\mathbb{R}^d$, $F\in\mathcal{K}$ and $r\in\mathbb{R}$, then $\mu_\xi$ is the unique minimizer of (\ref{level1ratetilde}).
\end{lemma}

In Section \ref{vandiseksin}, we verify the above Ansatz in the case of nearest-neighbor RWRE on $\mathbb{Z}$.
\begin{theorem}\label{findthemin}
Assume that $d=1$, the walk is nearest-neighbor, and
\begin{equation}
\int|\log\pi(0,\pm 1)|^{1+\alpha}\mathrm{d}\mathbb{P}<\infty\label{pleasant}
\end{equation}for some $\alpha>0$. Then, there exist $\xi_c,\xi_c'\in\mathbb{R}$ with $-1<\xi_c'\leq0\leq\xi_c<1$ such that there is a $\mu_\xi\in M_1(\Omega\times U)$ that fits the Ansatz given in Lemma \ref{lagrange} whenever $\xi\in (-1,\xi_c')\cup(\xi_c,1)$.
\end{theorem}
\begin{remark}
In the proof of Theorem \ref{findthemin}, we construct the unique minimizer $\mu_\xi$. Plugging it in (\ref{level2ratetilde}) gives an explicit expression for (\ref{level1ratetilde}) when $\xi\in (-1,\xi_c')\cup(\xi_c,1)$. Our formula agrees with the one provided in \cite{CGZ}.
\end{remark}
\begin{remark}
Theorem \ref{findthemin} generalizes to the case where the steps are bounded but not necessarily nearest-neighbor. The idea of the proof is the same. We chose to focus on nearest-neighbor walks in order to keep the arguments short.
\end{remark}

In general, whenever one takes the point of view of a particle performing RWRE, the main tool for proving limit theorems is
\begin{lemma}[Kozlov \cite{Kozlov}]\label{Kozlov}
If an environment kernel $\hat{\pi}$ satisfies $\hat{\pi}(\cdot,z)>0$ $\mathbb{P}$-a.s.\ for each $z\in U$, and if there exists a $\hat{\pi}$-invariant probability measure $\mathbb{Q}\ll\mathbb{P}$, then the following hold:
\begin{itemize}
\item[(a)] The measures $\mathbb{P}$ and $\mathbb{Q}$ are in fact mutually absolutely continuous.
\item[(b)] The environment Markov chain with transition kernel $\hat{\pi}$ and initial distribution $\mathbb{Q}$ is stationary and ergodic.
\item[(c)] $\mathbb{Q}$ is the unique $\hat{\pi}$-invariant probability measure on $\Omega$ that is absolutely continuous relative to $\mathbb{P}$.
\item[(d)] The following LLN is satisfied: \[P_o^{\hat{\pi}}\left(\lim_{n\rightarrow\infty}\frac{X_n}{n}=\int\sum_{z\in\mathcal{R}}\hat{\pi}(\omega,z)z\;\mathrm{d}\mathbb{Q}\right) = 1.\]
\end{itemize}
\end{lemma}
Let us for every $y\in\mathbb{Z}$ define the passage times \begin{equation}t_y:=\inf\{k\geq0:X_k\geq y\}\quad\mbox{and}\quad t_y':=\inf\{k\geq0:X_k\leq y\}.\label{nihatgomer}\end{equation} When $d=1$, if the walk is ballistic (i.e., if $E_o^{\hat{\pi}}[t_1]$ or $E_o^{\hat{\pi}}[t_{-1}']$ is finite) and nearest-neighbor, \cite{alili} shows the existence of a $\hat{\pi}$-invariant probability measure $\mathbb{Q}\ll\mathbb{P}$ and provides a formula for its density. We use this in our proof of Theorem \ref{findthemin}. The last result of Section \ref{vandiseksin} constructs the invariant measure in the case of ballistic RWRE with bounded steps on $\mathbb{Z}$.
\begin{theorem}\label{density}
In the case of RWRE with bounded steps on $\mathbb{Z}$, if the environment kernel $\hat{\pi}$ satisfies $\hat{\pi}(\cdot,1)>0$ $\mathbb{P}$-a.s.\ and if $E_o^{\hat{\pi}}[t_1]<\infty$, then the following hold:
\begin{itemize}
\item[(a)] $\phi(\omega):=\lim_{x\rightarrow-\infty} E_x^{\hat{\pi},\omega}\left[\sum_{k=0}^\infty\one_{X_k=0}\right]>0$ exists for $\mathbb{P}$-a.e.\ $\omega$.\label{directlimit}
\item[(b)] $\phi\in L^1(\mathbb{P})$.
\item[(c)] The measure $\mathbb{Q}$ defined by $\mathrm{d}\mathbb{Q}(\omega)=\left({1}/{\left\|\phi\right\|_{L^1(\mathbb{P})}}\right)\phi(\omega)\mathrm{d}\mathbb{P}(\omega)$ is $\hat{\pi}$-invariant.
\end{itemize}
\end{theorem}
\begin{remark}
If $E_o^{\hat{\pi}}[t_{-1}']<\infty$, then take $x\to\infty$ instead of $x\to-\infty$ in (a).
\end{remark}
\begin{remark}
Br\'emont \cite{Bremont} also shows the existence of a $\hat{\pi}$-invariant probability measure $\mathbb{Q}\ll\mathbb{P}$ in the case of ballistic RWRE with bounded steps on $\mathbb{Z}$. However, his argument is not elementary, assumes a stronger ellipticity condition, and does not provide a formula for the density. Rassoul-Agha \cite{Firas} takes an approach similar to ours, but resorts to Ces\`aro means and weak limits instead of showing the almost sure convergence in part (a) of Theorem \ref{density}, and assumes that Kalikow's condition (see (A3) in Section \ref{franzek}) holds. For the related model of ``random walk on a strip", Roitershtein \cite{roiter} shows the existence of the ergodic invariant measure. It is easy to see that the natural analog of our formula works in that setting.
\end{remark}

\section{Averaged large deviations}\label{franzek}

\subsection{Previous results}

In their aforementioned paper concerning nearest-neighbor RWRE on $\mathbb{Z}$, Comets et al.\ \cite{CGZ} prove also the following

\begin{theorem}[Averaged LDP]\label{aLDPgeneric}
$\left(P_o\left(\frac{X_n}{n}\in\cdot\,\right)\right)_{n\geq1}$ satisfies the LDP with a convex rate function $I_a$.
\end{theorem}

\noindent They establish this result for a class of environments including the i.i.d.\ case, and obtain the following variational formula for $I_a$:
\begin{equation}\label{montreal}
I_a(\xi)=\inf_{\mathbb{Q}}\left\{I_q^\mathbb{Q}(\xi) + |\xi|h\left(\mathbb{Q}\left|\mathbb{P}\right.\right)\right\}.
\end{equation}Here, the infimum is over all stationary and ergodic probability measures on $\Omega$, $I_q^\mathbb{Q}(\cdot)$ denotes the rate function for the quenched LDP when the environment measure is $\mathbb{Q}$, and $h\left(\cdot\left|\cdot\right.\right)$ is specific relative entropy. Similar to the quenched picture, the graph of $I_a$ is shown to typically have flat pieces. Note that the regularity properties of $I_a$ are not studied in \cite{CGZ}.

Varadhan \cite{Raghu} considers RWRE with bounded steps on $\mathbb{Z}^d$, assumes that $\mathbb{P}$ is a product measure, and proves Theorem \ref{aLDPgeneric} for any $d\geq1$. He gives yet another variational formula for $I_a$. Below, we focus on the nearest-neighbor case and introduce some notation in order to write down this formula.

An infinite path $\left(x_i\right)_{i\leq0}$ with nearest-neighbor steps $x_{i+1}-x_i$ is said to be in $W_\infty^{\mathrm{tr}}$ if $x_o=0$ and $\lim_{i\to-\infty}|x_i|=\infty$. For any $w\in W_\infty^{\mathrm{tr}}$, let $n_o$ be the number of times $w$ visits the origin, excluding the last visit. By the transience assumption, $n_o$ is finite. For any $z\in U$, let $n_{o,z}$ be the number of times $w$ jumps to $z$ after a visit to the origin. Clearly, $\sum_{z\in U}n_{o,z}=n_o$. If the averaged walk starts from time $-\infty$ and its path $\left(X_i\right)_{i\leq0}$ up to the present is conditioned to be equal to $w$, then the probability of the next step being equal to $z$ is
\begin{equation}\label{ozgurevren}
q(w,z):=\frac{\mathbb{E}\left[\pi(0,z)\prod_{z'\in U}\pi(0,z')^{n_{o,z'}}\right]}{\mathbb{E}\left[\prod_{z'\in U}\pi(0,z')^{n_{o,z'}}\right]}
\end{equation} by Bayes' rule. The probability measure that the averaged walk induces on $\left(X_n\right)_{n\geq0}$ conditioned on $\{\left(X_i\right)_{i\leq0}=w\}$ is denoted by $Q^w$. As usual, $E^w$ stands for expectation under $Q^w$.

Consider the map $T^*:W_\infty^{\mathrm{tr}}\to W_\infty^{\mathrm{tr}}$ that takes $\left(x_i\right)_{i\leq0}$ to $\left(x_i-x_{-1}\right)_{i\leq-1}$. Let $\mathcal{I}$ be the set of probability measures on $W_\infty^{\mathrm{tr}}$ that are invariant under $T^*$, and $\mathcal{E}$ be the set of extremal points of $\mathcal{I}$. Each $\mu\in\mathcal{I}$ (resp. $\mu\in\mathcal{E})$ corresponds to a transient process with stationary (resp. stationary and ergodic) increments and induces a probability measure $Q_\mu$ on particle paths $\left(X_i\right)_{-\infty<i<\infty}$. The associated ``mean drift" is $m(\mu) := \int\left(x_o-x_{-1}\right)\mathrm{d}\mu=Q_\mu(X_1-X_o)$. Define $$Q_\mu^w(\cdot):=Q_\mu(\,\cdot\,\left|\left(X_i\right)_{i\leq0}=w\right.)\quad\text{and}\quad q_\mu(w,z):=Q_\mu^w(X_1=z)$$ for any $w\in W_\infty^{\mathrm{tr}}$ and $z\in U$. Expectations under $Q_\mu$ and $Q_\mu^w$ are denoted by $E_\mu$ and $E_\mu^w$, respectively.

With this notation,
\begin{equation}\label{divaneasik}
I_a(\xi)=\inf_{\substack{\mu\in\mathcal{E}\\m(\mu)=\xi}}\mathfrak{I}_a(\mu)
\end{equation} for every $\xi\neq0$, where
\begin{equation}\label{hayirsh}
\mathfrak{I}_a(\mu):=\int_{W_\infty^{\mathrm{tr}}}\left[\sum_{z\in U}q_\mu(w,z)\log\frac{q_\mu(w,z)}{q(w,z)}\right]\,\mathrm{d}\mu(w).
\end{equation}

Aside from showing that $I_a$ is convex, Varadhan analyzes the set $$\mathcal{N}:=\left\{\xi\in\mathbb{R}^d:I_a(\xi)=0\right\}$$ where the rate function $I_a$ vanishes. For non-nestling walks (recall Definition \ref{defnestling}), $\mathcal{N}$ consists of a single point $\xi_o$ which is the LLN velocity. In the case of nestling walks, $\mathcal{N}$ is a line segment through the origin that can extend in one or both directions.

Rassoul-Agha \cite{FirasLDP} generalizes Varadhan's result to a class of mixing environments, and also to some other models of random walk on $\mathbb{Z}^d$.

\subsection{Our results}

We make the following assumptions:
\begin{enumerate}
\item [(A1)] $\mathbb{P}$ is a product measure and the walk is nearest-neighbor.
\item [(A2)] There exists a constant $c_1>0$ such that $\mathbb{P}\left(\pi(0,z)\geq c_1\right)=1$ for each $z\in U$. This is known as ``uniform ellipticity".
\item [(A3)] Kalikow's condition relative to a unit vector $\hat{u}\in\mathbb{R}^d$ is satisfied. Namely, $$\inf_{\mathcal{G}}\inf_{x\in\mathcal{G}}\frac{\mathbb{E}\left[E_o^\omega\left[\sum_{k=0}^{t_\mathcal{G}}\one_{X_k=x}\right]\sum_{z\in U}\pi(x,x+z)\langle z,\hat{u}\rangle\right]}{E_o\left[\sum_{k=0}^{t_\mathcal{G}}\one_{X_k=x}\right]}>0.$$ Here, the first infimum is over all connected strict subsets of $\mathbb{Z}^d$ that contain the origin, and $t_\mathcal{G}$ is the first time the walk exits $\mathcal{G}$.
\end{enumerate}
\begin{remark}
Assumption (A3) is first formulated in \cite{Kalikow}. It is the weakest known condition that implies transience. However, it is not easy to verify since it involves both the walk and the environment. In the case of non-nestling walks, $\hat{u}$ can be chosen such that $$\mathbb{P}\left(\sum_{z\in U}\pi(0,z)\langle z,\hat{u}\rangle\geq c_2\right)=1$$ for some constant $c_2>0$, and (A3) is clearly satisfied.
\end{remark}

Our approach is based on a renewal structure which is first introduced in \cite{SznitmanZerner}. Here is a brief description: Take the unit vector $\hat{u}\in\mathbb{R}^d$ appearing in (A3). Let $$D:=\inf\left\{k\geq0:\langle X_k,\hat{u}\rangle<\langle X_o,\hat{u}\rangle\right\}.$$ Recursively define a sequence $\left(\tau_m\right)_{m\geq1}$ of random times, which will be referred to as ``regeneration times", by
\begin{align*}
\tau_1&:=\inf\left\{j>0:\langle X_i,\hat{u}\rangle<\langle X_j,\hat{u}\rangle\leq\langle X_k,\hat{u}\rangle\mbox{ for all }i,k\mbox{ with }i<j<k\right\},\\
\tau_{m+1}&:=\inf\left\{j>\tau_m:\langle X_i,\hat{u}\rangle<\langle X_j,\hat{u}\rangle\leq\langle X_k,\hat{u}\rangle\mbox{ for all }i,k\mbox{ with }i<j<k\right\}.
\end{align*}
Denote the steps $X_i-X_{i-1}$ of the walk by $Z_i$. Then, $\left(Z_{\tau_m+1},\ldots,Z_{\tau_{m+1}}\right)_{m\geq1}$ is an i.i.d.\ sequence under $P_o$, and $$P_o\left(\left(Z_{\tau_1+1},\ldots,Z_{\tau_2}\right)\in\cdot\,\right)=P_o\left(\left.\left(Z_1,\ldots,Z_{\tau_1}\right)\in\cdot\,\right|\,D=\infty\right).$$ Sznitman and Zerner \cite{SznitmanZerner} use these facts to show that the LLN holds with limiting velocity 
\begin{equation}\label{huseysh}
\xi_o=\frac{E_o\left[\left.X_{\tau_1}\right|D=\infty\right]}{E_o\left[\left.\tau_1\right|D=\infty\right]}\neq0.
\end{equation}

Since (A1) and (A2) are sufficient for the validity of Theorem \ref{aLDPgeneric},
\begin{equation}\label{kazimkoyuncu}
\Lambda_a(\theta):=\lim_{n\to\infty}\frac{1}{n}\log E_o\left[\mathrm{e}^{\langle\theta,X_n\rangle}\right]=\sup_{\xi\in\mathbb{R}^d}\left\{\langle\theta,\xi\rangle - I_a(\xi)\right\}
\end{equation} by Varadhan's lemma (see \cite{DZ}). We start Section \ref{aLDPregularitysection} by obtaining a series of intermediate results including
\begin{lemma}\label{berkeleyolursa}
$\Lambda_a$ is strictly convex and analytic on a non-empty open set $\mathcal{C}$.
\begin{itemize}
\item [(a)] If the walk is non-nestling, $\mathcal{C}=\left\{\theta\in\mathbb{R}^d:|\theta|<c_3\right\}$ for some $c_3>0$.
\item [(b)] If the walk is nestling, $\mathcal{C}=\left\{\theta\in\mathbb{R}^d:|\theta|<c_4\,, \Lambda_a(\theta)>0\right\}$ for some $c_4>0$.
\end{itemize}
\end{lemma}
\noindent We then use the convex duality in (\ref{kazimkoyuncu}) to establish
\begin{theorem}\label{qual}
$I_a$ is strictly convex and analytic on the non-empty open set $$\mathcal{A}:=\{\nabla\Lambda_a(\theta):\theta\in\mathcal{C}\}.$$
\begin{enumerate}
\item [(a)] If the walk is non-nestling, then $\xi_o\in\mathcal{A}$.
\item [(b)] If the walk is nestling, then $\xi_o\in\partial\mathcal{A}$. For $d\geq2$, $\partial\mathcal{A}$ is smooth at $\xi_o$. The unit vector $\eta_o$ normal to $\partial\mathcal{A}$ (and pointing in $\mathcal{A}$) at $\xi_o$ satisfies $\langle\eta_o,\xi_o\rangle>0$. 
\end{enumerate}
\end{theorem}

In Section \ref{aLDPminimizersection}, we identify the unique minimizer in (\ref{divaneasik}) for $\xi\in\mathcal{A}$. The natural interpretation is that this minimizer gives the distribution of the RWRE path under $P_o$ when the particle is conditioned to escape to infinity with mean velocity $\xi$. Here is our candidate:

\begin{definition}\label{definemuyuanan}
For every $\xi\in\mathcal{A}$, define a measure $\bar{\mu}_\xi^\infty$ on $U^\mathbb{N}$ in the following way: There exists a unique $\theta\in\mathcal{C}$ satisfying $\xi=\nabla\Lambda_a(\theta)$. For every $K\in\mathbb{N}$, take any bounded function $f:U^{\mathbb{N}}\rightarrow\mathbb{R}$ such that $f((z_i)_{i\geq1})$ is independent of $(z_i)_{i>K}$.
\begin{equation}
\int\!\! f\mathrm{d}\bar{\mu}_\xi^\infty:=\frac{E_o\left[\left.\sum_{j=0}^{\tau_1 -1}f((Z_{j+i})_{i\geq1})\ \mathrm{e}^{\langle\theta,X_{\tau_{K}}\rangle - \Lambda_a(\theta)\tau_{K}}\,\right|\,D=\infty\right]}{E_o\left[\left.\tau_1\ \mathrm{e}^{\langle\theta,X_{\tau_1}\rangle - \Lambda_a(\theta)\tau_1}\,\right|\,D=\infty\right]}\label{muyucananan}.
\end{equation}
\end{definition}
%Note that $\bar{\mu}_\xi^\infty$ induces a measure $\mu_\xi^\infty$ on $\left(\mathbb{Z}^d\right)^{\mathbb{N}}$ via the map $$(z_1,z_2,z_3,\ldots)\mapsto(z_1,z_1+z_2,z_1+z_2+z_3,\ldots).$$ After showing that $\mu_\xi^\infty\in\mathcal{I}$ and $m\left(\mu_\xi^\infty\right)=\xi$, we finally prove
\begin{theorem}\label{sevval}
For every $\xi\in\mathcal{A}$, the measure $\mu_\xi^\infty$ on $\left(\mathbb{Z}^d\right)^{\mathbb{N}}$ induced by $\bar{\mu}_\xi^\infty$  via the map $(z_1,z_2,\ldots)\mapsto(z_1,z_1+z_2,\ldots)$ is the unique minimizer of (\ref{divaneasik}).
\end{theorem}

\section{Quenched vs.\ averaged}\label{nilgunisik}

\subsection{Previous results and a conjecture}

Consider nearest-neighbor RWRE on $\mathbb{Z}^d$. Assume that the environment is i.i.d.\ and uniformly elliptic. Then, the quenched and averaged LDPs hold with rate functions $I_q$ and $I_a$, respectively. Clearly,
\begin{equation}\label{geckalmakkk}
\mathcal{D}:=\left\{(\xi_1,\ldots,\xi_d)\in\mathbb{R}^d:|\xi_1|+\cdots+|\xi_d|\leq1\right\}=\left\{\xi\in\mathbb{R}^d:I_q(\xi)<\infty\right\}.
\end{equation} 
For any $\xi\in\mathcal{D}$, it follows from Jensen's inequality that $I_a(\xi)\leq I_q(\xi)$.

Take any $\xi=(\xi_1,\ldots,\xi_d)\in\mathbb{R}^d$ with $|\xi_1|+\cdots+|\xi_d|=1$, and assume WLOG that $\xi_j\geq0$ for all $j=1,\ldots,d$. Denote the canonical basis of $\mathbb{Z}^d$ by $(e_1,\ldots,e_d)$. The paths constituting the event $\left\{\frac{X_n}{n}=\xi\right\}$ do not visit the same point more than once, and it is not hard to see that $$I_a(\xi)=\sum_{j=1}^d\xi_j\log\frac{\xi_j}{\mathbb{E}\left[\pi(0,e_j)\right]}\qquad\mbox{and}\qquad I_q(\xi)=\mathbb{E}\left[\sum_{j=1}^d\xi_j\log\frac{\xi_j}{\pi(0,e_j)}\right].$$ Again by Jensen's inequality, $I_a(\xi)<I_q(\xi)$ as long as the environment is not deterministic. Since the rate functions are convex and thus continuous on $\mathcal{D}$, we conclude that $I_a(\cdot)<I_q(\cdot)$ on the boundary and at some interior points of $\mathcal{D}$.

In the case of nearest-neighbor RWRE on $\mathbb{Z}$, recall that (\ref{montreal}) connects the rate functions $I_a$ and $I_q$. When $\mathbb{P}$ is a product measure, Comets et al.\ \cite{CGZ} use this formula to show that $I_a(\xi)=I_q(\xi)$ if and only if $\xi=0$ or $I_a(\xi)=0$.

When $d\geq2$, Varadhan \cite{Raghu} proves that the statements $I_a(0)=I_q(0)$ and $\left\{\xi:I_a(\xi)=0\right\}=\left\{\xi:I_q(\xi)=0\right\}$ continue to hold. It is not known whether these are the only points where the two rate functions are equal. Here is our
\begin{conjecture}\label{conjecture}
For walks in high dimension, $I_a$ and $I_q$ agree on a set with non-empty interior.
\end{conjecture}

\subsection{Our results in the space-time case}

In the definition of RWRE, the environment $\omega$ is sampled from $(\Omega,\mathcal{B}, \mathbb{P})$ and kept fixed throughout the walk. In other words, if the particle visits a point multiple times, it sees the same environment there at every visit. Thus, the walk under the averaged measure $P_o$ has a long-term memory which makes the model hard to analyze.

In Chapter \ref{spacetimechapter}, we consider a simpler model referred to as ``space-time RWRE" where we assume that the transition probabilities at distinct points are i.i.d.\ and are freshly sampled at each time step. To explicitly indicate the time dependence, write $\omega_{n,x}:=\left(\pi_{n,n+1}(x,x+z)\right)_{z\in\mathbb{Z}^d}$ for the environment at $x$ at time $n$. The environment is i.i.d.\ in space as well as in time, i.e., $\omega:=\left(\omega_{n,x}\right)_{n\in\mathbb{Z},x\in\mathbb{Z}^d}$ is an i.i.d.\ collection.

Apart from $\mathcal{B}$, define the ``past" and ``future" $\sigma$-algebras $\mathcal{B}_n^-$ and $\mathcal{B}_n^+$ on $\Omega$ which for every $n\in\mathbb{Z}$ are respectively generated by $\left(\omega_{m,x}:\ x\in\mathbb{Z}^d,\ m\leq n\right)$ and $\left(\omega_{m,x}:\ x\in\mathbb{Z}^d,\ m\geq n\right)$.

Note that if $\left(X_n\right)_{n\geq0}$ denotes the space-time RWRE path on $\mathbb{Z}^d$, then $\left(n,X_n\right)_{n\geq0}$ can be viewed as the trajectory of a particle performing RWRE on $\mathbb{Z}^{d+1}$ such that the first component of the position of the particle at time $n$ is always equal to $n$. With this picture in mind, the quenched and averaged measures on paths starting at $x$ at time $k$ are denoted by $P_{k,x}^{\omega}$ and $P_{k,x}$, respectively. Similarly, write $E_{k,x}^{\omega}$ and $E_{k,x}$ for the corresponding expectations.

To keep the arguments short, assume that the walk $\left(X_n\right)_{n\geq0}$ is nearest-neighbor. Plus, impose a uniform ellipticity condition which now means there exists a constant $c_1>0$ such that $\mathbb{P}(\pi_{0,1}(0,z)\geq c_1)=1$ for each $z\in U$.

Define the space-time shifts $\left(T_{m,y}\right)_{m\in\mathbb{Z}, y\in\mathbb{Z}^d}$ on $\Omega$ by $\left(T_{m,y}\omega\right)_{n,x}=\omega_{n+m,x+y}$. With this notation, the transition kernel $\overline{\pi}$ of the environment Markov chain $\left(T_{n,X_n}\omega\right)_{n\geq 0}$ satisfies $\overline{\pi}(\omega,T_{1,z}\omega)=\pi_{0,1}(0,z)$ for every $\omega\in\Omega$ and $z\in U$.

The marginal of $P_{o,o}$ on paths is classical random walk with transition vector $\left(q(z)\right)_{z\in U}$ given by $q(z)=\mathbb{E}[\pi_{0,1}(0,z)]$ for every $z\in U$. Therefore, the LLN for the mean velocity is valid, and the limiting velocity vector $\xi_o$ is $\sum_{z\in U}q(z)z$. The averaged LDP for the mean velocity is simply Cram\'{e}r's theorem (see \cite{DZ}) and the rate function $I_c$ is the convex conjugate of the logarithmic moment generating function $\Lambda_c:\mathbb{R}^d\to\mathbb{R}$ given by 
\begin{equation}
\Lambda_c(\theta)=\log\left(\sum_{z\in U}q(z)\mathrm{e}^{\langle\theta,z\rangle}\right).\label{fi}
\end{equation}

Even though we can think of $(n,X_n)_{n\geq0}$ as RWRE on $\mathbb{Z}^{d+1}$, the results of \cite{jeffrey} and \cite{Raghu} on quenched large deviations are not directly applicable since our environment is not elliptic in the ``time" direction. However, one expects that modifications of these arguments should work. Instead of taking this route, we develop an alternative technique in Section \ref{Qsection} and prove Conjecture \ref{conjecture} in the space-time case:

\begin{theorem}\label{AequalsQ}
If $d\geq3$, then there exists $\eta>0$ such that the quenched LDP for the mean velocity holds in the $\eta$-neighborhood of $\xi_o$, and the rate function is identically equal to the rate function $I_c$ of the averaged LDP in this neighborhood.
\end{theorem}
\begin{remark}
This theorem is similar in flavor to the results in \cite{Flury}, \cite{Song}, and \cite{Nikos} on the related model of random walk with a random potential.
\end{remark}

%By a generalization of a technique first given by Kozlov \cite{Kozlov}, Rassoul-Agha \cite{Firas} shows that the environment Markov chain has a unique invariant measure $\mu_{\xi_o}^1$ (the notation will become clear later) that is absolutely continuous relative to $\mathbb{P}$ on every $\mathcal{B}_n^+$, and that the empirical measure $$\nu_{n,X}^1:=\frac{1}{n}\sum_{k=0}^{n-1}\one_{T_{X_k}\omega}$$ of the environment Markov chain converges to $\mu_{\xi_o}^1$ with $P_{o,o}$-probability $1$.

Having established the equality of the rate functions in a neighborhood of the true velocity $\xi_o$, we move on to another large deviation property of space-time RWRE. Note that the random measures \[\bar{\nu}_{n,X}^\infty := \frac{1}{n}\sum_{j=0}^{n-1}\one_{T_{j,X_j}\omega,\left(Z_{j+i}\right)_{i\geq1}}\qquad\mbox{and}\qquad\nu_{n,X}^\infty:= \frac{1}{n}\sum_{j=0}^{n-1}\one_{\left(T_{j+i,X_{j+i}}\omega\right)_{i\geq0}}\] can be naturally identified. (Here, $Z_i=X_i-X_{i-1}$ are the steps of the walk.) Therefore, $\bar{\nu}_{n,X}^\infty$ is referred to as ``the empirical process of the environment Markov chain". Recall (\ref{geckalmakkk}). Given $\xi\in \mathcal{D}^o$, consider the event defined by the particle having mean velocity $\xi$ after a large time $n$. If $\xi\neq\xi_o$, this is a rare event and the exponential rate of decay in $n$ of its $P_{o,o}$-probability is given by $I_c(\xi)>0$. Conditioned on this event, we show that $\nu_{n,X}^\infty$ under $P_{o,o}$ converges to a stationary process uniquely determined by $\xi$. In order to rigorously formulate this result, we first give a

\begin{definition}\label{definemu}
For every $\xi\in\mathcal{D}^o$, define a measure $\bar{\mu}_\xi^\infty$ on $\Omega\times U^\mathbb{N}$ in the following way: There exists a unique $\theta\in\mathbb{R}^d$ satisfying $\xi=\nabla\Lambda_c(\theta)$. For every $N,M$ and $K\in\mathbb{N}$, take any bounded function $f:\Omega\times U^{\mathbb{N}}\rightarrow\mathbb{R}$ such that $f(\cdot,(z_i)_{i\geq1})$ is independent of $(z_i)_{i>K}$ and $\mathcal{B}_{-N}^+\cap\mathcal{B}_M^-$-measurable for each $(z_i)_{i\geq1}$.
\begin{equation}\label{mucan}
\int\!\! f\mathrm{d}\bar{\mu}_\xi^\infty:=E_{o,o}\left[\mathrm{e}^{\langle\theta,X_{N+M+K+1}\rangle - (N+M+K+1)\Lambda_c(\theta)}f(T_{N,X_N}\omega,(Z_{N+i})_{i\geq1})\right].
\end{equation}
\end{definition}
\begin{remark}\label{fehimbey}
Recall the terminology introduced in Section \ref{franzek}. The walk $(n,X_n)_{n\geq0}$ on $\mathbb{Z}^{d+1}$ is clearly non-nestling in the ``time" direction, and the regeneration times satisfy $\tau_m=m$. Therefore, Definition \ref{definemu} is nothing but the space-time version of Definition \ref{definemuyuanan}, except that the test functions here depend also on the environment.
\end{remark}

%Let us now define the events that we use in the statement of the main theorem of Section \ref{Asection}:

%\begin{definition}\label{events}
%For every $\xi\in\mathcal{D}^o$, $N,M$ and $K\in\mathbb{N}$, $f$ as in Definition \ref{definemu}, $\epsilon>0$ and $n\in\mathbb{N}$, define the event \begin{equation}
%|\langle f,\bar{\nu}_{n,X}^\infty - \bar{\mu}_\xi^\infty\rangle|>\epsilon:=\left\{\left|\frac{1}{n}\sum_{j=0}^{n-1}f(T_{j,X_j}\omega,(Z_{j+i})_{i\geq1})-\int\!\! f\mathrm{d}\bar{\mu}_\xi^\infty\right|>\epsilon\right\}.\label{A}
%\end{equation}
%Given $\delta>0$, define the event 
%\begin{equation}
%D_{\xi,n}^\delta:=\left\{\left|\frac{X_n}{n}-\xi\right|\leq\delta\right\}\label{D}.
%\end{equation}
%\end{definition}

We start Section \ref{Asection} by showing that $\bar{\mu}_\xi^\infty$ is well defined, and that it naturally induces a stationary process $\mu_\xi^\infty$ with values in $\Omega$. The theorem below says that $\nu_{n,X}^\infty$ converges to $\mu_\xi^\infty$ under $P_{o,o}$ when the particle is conditioned to have mean velocity $\xi$. It is the first main result of Section \ref{Asection}.

\begin{theorem}\label{averagedconditioning}
For every $\xi\in\mathcal{D}^o$, $N,M,K\in\mathbb{N}$, $f$ as in Definition \ref{definemu}, and $\epsilon>0$, \[\limsup_{\delta\to0}\limsup_{n\rightarrow\infty}\frac{1}{n}\log P_{o,o}\left(\ \left|\int\!\! f\mathrm{d}\bar{\nu}_{n,X}^\infty-\int\!\! f\mathrm{d}\bar{\mu}_\xi^\infty\right|>\epsilon\ \left|\ |\frac{X_n}{n}-\xi|\leq\delta\right.\right)<0.\]
\end{theorem}
\begin{remark}
Note that Theorem \ref{sevval} for RWRE and Theorem \ref{averagedconditioning} for space-time RWRE have very similar interpretations. In fact, as we will see in Section \ref{aLDPminimizersection}, the proof of Theorem \ref{sevval} relies on the RWRE analog of Theorem \ref{averagedconditioning}.
\end{remark}

One can ask what $\nu_{n,X}^\infty$ converges to under $P_{o,o}^\omega$ when the particle is conditioned to have mean velocity $\xi$. Whenever the quenched LDP for the mean velocity holds in a neighborhood of $\xi$ with rate $I_c(\xi)$ at $\xi$ --- in particular when $d\geq3$ and $|\xi-\xi_o|<\eta$ --- the answer is again $\mu_\xi^\infty$, as one expects.

\begin{theorem}\label{strong}
Assume that the quenched LDP for the mean velocity holds in a neighborhood of $\xi$ with rate $I_c(\xi)$ at $\xi$. Then, for $\mathbb{P}$-a.e.\ $\omega$, and every $N,M,K\in\mathbb{N}$, $f$ as in Definition \ref{definemu}, and $\epsilon>0$, \[\limsup_{\delta\to0}\limsup_{n\rightarrow\infty}\frac{1}{n}\log P_{o,o}^\omega\left(\ \left|\int\!\! f\mathrm{d}\bar{\nu}_{n,X}^\infty-\int\!\! f\mathrm{d}\bar{\mu}_\xi^\infty\right|>\epsilon\ \left|\ |\frac{X_n}{n}-\xi|\leq\delta\right.\right)<0.\]
\end{theorem}

The formula for $\bar{\mu}_\xi^\infty$ given in Definition \ref{definemu} is not very explicit. We conclude Section \ref{Asection} by showing that $\mu_\xi^\infty$ actually has a simple and elegant structure for $d\geq3$ and $|\xi-\xi_o|<\eta$.

\begin{theorem}\label{doob}
For $d\geq3$ and $|\xi-\xi_o|<\eta$ with $\eta$ as in Theorem \ref{AequalsQ}, let $\theta\in\mathbb{R}^d$ be the unique solution of $\xi=\nabla\Lambda_c(\theta)$. There exists a $\mathcal{B}_o^+$-measurable function $u^\theta>0$ that satisfies $\int\! u^\theta\mathrm{d}\mathbb{P} = 1$ and $\mathbb{P}$-a.s. \[u^\theta(\omega)=\sum_{z\in U}\overline{\pi}(\omega,T_{1,z}\omega)\mathrm{e}^{\langle\theta,z\rangle-\Lambda_c(\theta)}u^\theta(T_{1,z}\omega).\] Define a new kernel \[\overline{\pi}^\theta(\omega,T_{1,z}\omega) := \overline{\pi}(\omega,T_{1,z}\omega)\frac{u^\theta(T_{1,z}\omega)}{u^\theta(\omega)}\mathrm{e}^{\langle\theta,z\rangle-\Lambda_c(\theta)}\] on $\Omega$ via Doob $h$-transform. $\mu_\xi^\infty$ is the unique stationary Markov process with transition kernel $\overline{\pi}^\theta$ and whose marginal is absolutely continuous relative to $\mathbb{P}$ on every $\mathcal{B}_n^+$.
\end{theorem}

In other words, when the particle is conditioned to have mean velocity $\xi$, the environment Markov chain chooses to switch from its original kernel $\overline{\pi}$ to a new kernel $\overline{\pi}^\theta$. The most economical tilt is given by a Doob $h$-transform.

\chapter{Quenched large deviations for RWRE}\label{quenchedchapter}

\section{LDP for the pair empirical measure}\label{pairLDPsection}

As mentioned in Section \ref{egridogru}, Rosenbluth \cite{jeffrey} takes the point of view of a particle performing RWRE and proves the quenched LDP for the mean velocity. In this section, we generalize his argument and prove Theorem \ref{level2LDP}.

The strategy is to first show the existence of the logarithmic moment generating function $\Lambda_q: C_b(\Omega\times\mathcal{R})\rightarrow\mathbb{R}$ given by
\begin{align}
\Lambda_q(f)=&\lim_{n\rightarrow\infty}\frac{1}{n}\log E_o^{\omega}\left[\mathrm{e}^{n\langle f,\nu_{n,X}\rangle}\right]\nonumber\\
=&\lim_{n\rightarrow\infty}\frac{1}{n}\log E_o^{\omega}\left[\exp\left(\sum_{k=0}^{n-1} f(T_{X_k}\omega,X_{k+1}-X_k)\right)\right]\label{askerdatca}
\end{align} where $C_b$ denotes the space of bounded continuous functions.
\pagebreak
\begin{theorem}\label{LMGF}
Assume there exists $\alpha>0$ such that (\ref{kimimvarki}) holds for each $z\in\mathcal{R}$. Then, the following hold:
\begin{description}
\item[Lower bound.] For $\mathbb{P}$-a.e.\ $\omega$,
\begin{align*}
&\liminf_{n\rightarrow\infty}\frac{1}{n}\log E_o^{\omega}\left[\exp\left(\sum_{k=0}^{n-1} f(T_{X_k}\omega,X_{k+1}-X_k)\right)\right]\\
&\geq\sup_{\mu\in M_{1,s}^{\ll}(\Omega\times\mathcal{R})}\int\sum_{z\in\mathcal{R}}\mathrm{d}\mu(\omega,z)\left(f(\omega,z)-\log \frac{\mathrm{d}\mu(\omega,z)}{\mathrm{d}(\mu)^1(\omega)\pi(0,z)}\right)=:\Gamma(f).
\end{align*}
\item[Upper bound.] For $\mathbb{P}$-a.e.\ $\omega$,
\begin{align*}
&\limsup_{n\rightarrow\infty}\frac{1}{n}\log E_o^{\omega}\left[\exp\left(\sum_{k=0}^{n-1} f(T_{X_k}\omega,X_{k+1}-X_k)\right)\right]\\
&\leq\inf_{F\in\mathcal{K}}\mathrm{ess}\sup_{\omega} \log\sum_{z\in\mathcal{R}}\pi(0,z)\mathrm{e}^{f(\omega,z)+F(\omega,z)}=:\Lambda_q(f).
\end{align*}
\item[Equivalence of the bounds.] For every $\epsilon>0$, there exists $F_\epsilon\in\mathcal{K}$ such that
\[\mathrm{ess}\sup_{\omega} \log\sum_{z\in\mathcal{R}}\pi(0,z)\mathrm{e}^{f(\omega,z)+F_\epsilon(\omega,z)}\leq\Gamma(f)+\epsilon.\]
\end{description}Thus, $\Lambda_q(f)\leq\Gamma(f)$. In other words, the limit in (\ref{askerdatca}) exists.
\end{theorem}
Subsection \ref{LMGFsubsection} is devoted to the proof of Theorem \ref{LMGF}. After that, proving Theorem \ref{level2LDP} is easy: the LDP lower bound follows from a standard change of measure argument and the LDP upper bound is obtained by an application of the G\"{a}rtner-Ellis theorem. These arguments are given in Subsection \ref{LDPproof}.

\subsection{Logarithmic moment generating function}\label{LMGFsubsection}

\subsubsection{Lower bound}\label{pourum}

This is a standard change of measure argument. 
For any environment kernel $\hat{\pi}$ as in Definition \ref{ortamkeli},
\begin{eqnarray*}
&&\ \ E_o^{\omega}\left[\exp\left(\sum_{k=0}^{n-1} f(T_{X_k}\omega,X_{k+1}-X_k)\right)\right]\\
&=&E_o^{\hat{\pi},\omega}\left[\exp\left(\sum_{k=0}^{n-1} f(T_{X_k}\omega,X_{k+1}-X_k)\right)\,\frac{\mathrm{d}P_o^\omega}{\mathrm{d}P_o^{\hat{\pi},\omega}}\right]\\
&=&E_o^{\hat{\pi},\omega}\left[\exp\left(\sum_{k=0}^{n-1}f(T_{X_k}\omega,X_{k+1}-X_k)-\log\frac{\hat{\pi}(T_{X_k}\omega,X_{k+1}-X_k)}{\pi(X_k,X_{k+1})}\right)\right].
\end{eqnarray*} If $\hat{\pi}(\cdot,z)>0$ $\mathbb{P}$-a.s.\ for each $z\in U$, and if there exists $\phi\in L^1(\mathbb{P})$ such that $\phi\,\mathrm{d}\mathbb{P}$ is an invariant probability measure for the environment kernel $\hat{\pi}$ (i.e., if $\phi(\omega)=\sum_{z\in\mathcal{R}}\phi(T_{-z}\omega)\hat{\pi}(T_{-z}\omega,z)$ for $\mathbb{P}$-a.e.\ $\omega$), then it follows from Lemma \ref{Kozlov} that $\phi\,\mathrm{d}\mathbb{P}$ is in fact an ergodic invariant measure for $\hat{\pi}$. By Jensen's inequality,
\begin{align}
&\liminf_{n\rightarrow\infty}\frac{1}{n}\log E_o^{\omega}\left[\exp\left(\sum_{k=0}^{n-1} f(T_{X_k}\omega,X_{k+1}-X_k)\right)\right]\nonumber\\\geq&\liminf_{n\rightarrow\infty}E_o^{\hat{\pi},\omega}\left[\frac{1}{n}\sum_{k=0}^{n-1}f(T_{X_k}\omega,X_{k+1}-X_k)-\log\frac{\hat{\pi}(T_{X_k}\omega,X_{k+1}-X_k)}{\pi(X_k,X_{k+1})}\right]\nonumber\\
=&\int\sum_{z\in\mathcal{R}}\hat{\pi}(\omega,z)\left(f(\omega,z)-\log \frac{\hat{\pi}(\omega,z)}{\pi(0,z)}\right)\phi(\omega)\mathrm{d}\mathbb{P}=:H_f(\hat{\pi},\phi).\label{hakkariye}\\
\text{Therefore,}\quad&\liminf_{n\rightarrow\infty}\frac{1}{n}\log E_o^{\omega}\left[\exp\left(\sum_{k=0}^{n-1} f(T_{X_k}\omega,X_{k+1}-X_k)\right)\right]\nonumber\\
\geq&\sup_{(\hat{\pi},\phi)}\int\sum_{z\in\mathcal{R}}\hat{\pi}(\omega,z)\left(f(\omega,z)-\log \frac{\hat{\pi}(\omega,z)}{\pi(0,z)}\right)\phi(\omega)\mathrm{d}\mathbb{P}\label{stef}
\end{align}
where the supremum is taken over the set of $(\hat{\pi},\phi)$ pairs where $\hat{\pi}(\cdot,z)>0$ $\mathbb{P}$-a.s.\ for each $z\in U$ and $\phi\,\mathrm{d}\mathbb{P}$ is a $\hat{\pi}$-invariant probability measure. Notice that there is a one-to-one correspondence between this set and $M_{1,s}^{\ll}(\Omega\times\mathcal{R})$. Hence, (\ref{stef}) is the desired lower bound.

Before proceeding with the upper bound, let us put (\ref{stef}) in a form that will turn out to be more convenient for showing the equivalence of the bounds. We start by giving the following
\begin{lemma}\label{camilla}
For every $f\in C_b(\Omega\times\mathcal{R})$, $H_f$ (defined in (\ref{hakkariye})) has the following concavity property: For each $t\in(0,1)$ and any two pairs $(\hat{\pi}_1,\phi_1)$ and $(\hat{\pi}_2,\phi_2)$ where $\phi_i\,\mathrm{d}\mathbb{P}$ is a $\hat{\pi}_i$-invariant probability measure (for $i=1,2$), define \[\gamma:=\frac{t\phi_1}{t\phi_1+(1-t)\phi_2},\ \ \phi_3:=t\phi_1+(1-t)\phi_2\ \ \mbox{and}\ \ \hat{\pi}_3:=\gamma\hat{\pi}_1+(1-\gamma)\hat{\pi}_2.\] Then, $\phi_3\,\mathrm{d}\mathbb{P}$ is $\hat{\pi}_3$-invariant and 
\begin{equation}
H_f(\hat{\pi}_3,\phi_3)\geq tH_f(\hat{\pi}_1,\phi_1)+(1-t)H_f(\hat{\pi}_2,\phi_2).\label{yozgat}
\end{equation}
\end{lemma}

\begin{proof}
For $t\in(0,1)$, use the definitions and the assumptions in the statement of the lemma to observe that $\mathbb{P}$-a.s.
\begin{align*}
&\sum_{z\in\mathcal{R}}\phi_3(T_{-z}\omega)\hat{\pi}_3(T_{-z}\omega,z)\\=&\sum_{z\in\mathcal{R}}\phi_3(T_{-z}\omega)\gamma(T_{-z}\omega)\hat{\pi}_1(T_{-z}\omega,z)+\sum_{z\in\mathcal{R}}\phi_3(T_{-z}\omega)(1-\gamma(T_{-z}\omega))\hat{\pi}_2(T_{-z}\omega,z)\\=&\ t\sum_{z\in\mathcal{R}}\phi_1(T_{-z}\omega)\hat{\pi}_1(T_{-z}\omega,z)+(1-t)\sum_{z\in\mathcal{R}}\phi_2(T_{-z}\omega)\hat{\pi}_2(T_{-z}\omega,z)\\=&\ t\phi_1(\omega)+(1-t)\phi_2(\omega)=\phi_3(\omega)
\end{align*} which proves that $\phi_3\,\mathrm{d}\mathbb{P}$ is $\hat{\pi}_3$-invariant. Finally,
\begin{align*}
H_f(\hat{\pi}_3,\phi_3)&=\int\sum_{z\in\mathcal{R}}\hat{\pi}_3(\omega,z)\left(f(\omega,z)-\log \frac{\hat{\pi}_3(\omega,z)}{\pi(0,z)}\right)\phi_3(\omega)\mathrm{d}\mathbb{P}\\&\geq\int\gamma(\omega)\sum_{z\in\mathcal{R}}\hat{\pi}_1(\omega,z)\left(f(\omega,z)-\log \frac{\hat{\pi}_1(\omega,z)}{\pi(0,z)}\right)\phi_3(\omega)\mathrm{d}\mathbb{P}\\&\ \ \ +\int(1-\gamma(\omega))\sum_{z\in\mathcal{R}}\hat{\pi}_2(\omega,z)\left(f(\omega,z)-\log \frac{\hat{\pi}_2(\omega,z)}{\pi(0,z)}\right)\phi_3(\omega)\mathrm{d}\mathbb{P}\\&=\ t\int\sum_{z\in\mathcal{R}}\hat{\pi}_1(\omega,z)\left(f(\omega,z)-\log \frac{\hat{\pi}_1(\omega,z)}{\pi(0,z)}\right)\phi_1(\omega)\mathrm{d}\mathbb{P}\\&\ \ \ +(1-t)\int\sum_{z\in\mathcal{R}}\hat{\pi}_2(\omega,z)\left(f(\omega,z)-\log \frac{\hat{\pi}_2(\omega,z)}{\pi(0,z)}\right)\phi_2(\omega)\mathrm{d}\mathbb{P}\\&=\ tH_f(\hat{\pi}_1,\phi_1)+(1-t)H_f(\hat{\pi}_2,\phi_2)
\end{align*} where the second line is obtained by applying Jensen's inequality to the integrand.
\end{proof}
Let us go back to the argument and define $(\hat{\pi}_1,\phi_1)$ by $\hat{\pi}_1(\omega,z):={1}/{(2d)}$ for each $z\in U$ and $\phi_1(\omega):=1$, $\mathbb{P}$-a.s. By an easy computation, $H_f(\hat{\pi}_1,\phi_1)>-\infty$. Take any pair $(\hat{\pi}_2,\phi_2)$ such that $\phi_2\,\mathrm{d}\mathbb{P}$ is $\hat{\pi}_2$-invariant and $H_f(\hat{\pi}_2,\phi_2)>-\infty$. For any $t\in(0,1)$, define $(\hat{\pi}_3,\phi_3)$ as in Lemma \ref{camilla} and see that $\hat{\pi}_3(\omega,z)>0$ $\mathbb{P}$-a.s.\ for each $z\in U$. Recalling (\ref{yozgat}), note that $H_f(\hat{\pi}_3,\phi_3)\geq (1-t)H_f(\hat{\pi}_2,\phi_2) + O(t)$. Since $t$ can be arbitrarily small, the value of (\ref{stef}) does not change if the supremum there is taken over the set of all $(\hat{\pi},\phi)$ pairs where $\phi\,\mathrm{d}\mathbb{P}$ is a $\hat{\pi}$-invariant probability measure, dropping the positivity condition on $\hat{\pi}$. Finally, decouple $\hat{\pi}$ and $\phi$, and express the lower bound $\Gamma(f)$ as
\begin{equation}
\sup_{\phi}\sup_{\hat{\pi}}\inf_{h}\int\sum_{z\in\mathcal{R}}\hat{\pi}(\omega,z)\left(f(\omega,z)-\log \frac{\hat{\pi}(\omega,z)}{\pi(0,z)}+h(\omega)-h(T_z\omega)\right)\phi\,\mathrm{d}\mathbb{P}\label{putinh}
\end{equation} where the suprema are over all probability densities and all environment kernels, and the infimum is over all bounded measurable functions. This is due to the observation that if $\phi\,\mathrm{d}\mathbb{P}$ is not $\hat{\pi}$-invariant, then there exists a bounded measurable function $h:\Omega\to\mathbb{R}$ satisfying \[\int\sum_{z\in\mathcal{R}}\hat{\pi}(\omega,z)\left(h(\omega)-h(T_z\omega)\right)\phi(\omega)\mathrm{d}\mathbb{P}\neq0,\] and the infimum in (\ref{putinh}) is $-\infty$ since $h$ can be multiplied by any scalar.

\subsubsection{Upper bound}

Let us fix $f\in C_b(\Omega\times\mathcal{R})$. For any $F\in\mathcal{K}$, set \[K(F):=\mbox{ess}\sup_\omega \log\sum_{z\in\mathcal{R}}\pi(0,z)\mathrm{e}^{f(\omega,z)+F(\omega,z)}.\] Then, for every $n\geq1$ and $\mathbb{P}$-a.e.\ $\omega$,
\begin{align*}
&E_o^{\omega}\left[\left.\mathrm{e}^{f(T_{X_{n-1}}\omega,X_n-X_{n-1})+F(T_{X_{n-1}}\omega,X_n-X_{n-1})}\right|X_{n-1}\right]\\
&\qquad\qquad=\sum_{z\in\mathcal{R}}\pi(X_{n-1},X_{n-1}+z)\mathrm{e}^{f(T_{X_{n-1}}\omega,z)+F(T_{X_{n-1}}\omega,z)}\leq\mathrm{e}^{K(F)}.
\end{align*}
It is easy to see by induction that
\[E_o^{\omega}\left[\exp\left(\sum_{k=0}^{n-1}f(T_{X_k}\omega,X_{k+1}-X_k)+F(T_{X_k}\omega,X_{k+1}-X_k)\right)\right]\leq\mathrm{e}^{nK(F)}.\] For any $\epsilon>0$, applying Lemma \ref{GRR} (stated below) gives \[E_o^{\omega}\left[\exp\left(-c_\epsilon-n\epsilon+\sum_{k=0}^{n-1}f(T_{X_k}\omega,X_{k+1}-X_k)\right)\right]\leq\mathrm{e}^{nK(F)}\] where $c_\epsilon = c_\epsilon(\omega)$ is some constant. Arranging the terms, \[\frac{1}{n}\log E_o^{\omega}\left[\exp\left(\sum_{k=0}^{n-1} f(T_{X_k}\omega,X_{k+1}-X_k)\right)\right]\leq K(F)+\epsilon+\frac{c_\epsilon}{n}.\] The desired upper bound is obtained by letting $n\to\infty,\ \epsilon\to 0$ and taking infimum over $F\in\mathcal{K}$.

\begin{lemma}\label{GRR}
For every $F\in\mathcal{K}$, $\epsilon>0$ and $\mathbb{P}$-a.e.\ $\omega$, $\exists\, c_\epsilon= c_\epsilon(\omega)\geq 0$ such that for any $n\geq1$ and any sequence $(x_{k})_{k=0}^n$ with $x_o=0$ and $x_{k+1}-x_k\in\mathcal{R}$, \[\left|\sum_{k=0}^{n-1}F(T_{x_k}\omega,x_{k+1}-x_k)\right| \leq c_\epsilon+n\epsilon.\]
\end{lemma}

\begin{remark}
See Chapter 2 of \cite{jeffrey} for the proof.
\end{remark}

\subsubsection{Equivalence of the bounds}

Consider a sequence $\left(\mathcal{E}_k\right)_{k\geq1}$ of finite $\sigma$-algebras such that
$\mathcal{B}=\sigma\left(\bigcup_k \mathcal{E}_k\right)$ and $\mathcal{E}_{k}\subset T_z\mathcal{E}_{k+1}$ for all $z\in\mathcal{R}$ and $k\geq1$. Then, recall (\ref{putinh}) and see that $\Gamma(f)$ can be bounded below by
\begin{align}
&\sup_{\phi}\sup_{\hat{\pi}}\inf_{h}\int\sum_{z\in\mathcal{R}}\hat{\pi}(\omega,z)\left(f(\omega,z)-\log \frac{\hat{\pi}(\omega,z)}{\pi(0,z)}+h(\omega)-h(T_z\omega)\right)\phi\,\mathrm{d}\mathbb{P}\label{thingone}\\=&\sup_{\phi}\inf_{h}\sup_{\hat{\pi}}\int\sum_{z\in\mathcal{R}}\hat{\pi}(\omega,z)\left(f(\omega,z)-\log \frac{\hat{\pi}(\omega,z)}{\pi(0,z)}+h(\omega)-h(T_z\omega)\right)\phi\,\mathrm{d}\mathbb{P}\label{thingtwo}\\=&\sup_{\phi}\inf_{h}\sup_{\hat{\pi}}\int\sum_{z\in\mathcal{R}}\left[v(\omega,z)-\log\hat{\pi}(\omega,z)\right]\hat{\pi}(\omega,z)\phi\,\mathrm{d}\mathbb{P}\label{thingthree}\\=&\sup_{\phi}\inf_{h}\int\sup_{\hat{\pi}(\omega,\cdot)}\left(\sum_{z\in\mathcal{R}}[v(\omega,z)-\log\hat{\pi}(\omega,z)]\hat{\pi}(\omega,z)\right)\phi\,\mathrm{d}\mathbb{P}\label{thingfour}\\=&\sup_{\phi}\inf_{h}\int\left(\log\sum_{z\in\mathcal{R}}\mathrm{e}^{v(\omega,z)}\right)\phi\,\mathrm{d}\mathbb{P}\label{thingfive}\displaybreak\\=&\inf_{h}\sup_{\phi}\int\left(\log\sum_{z\in\mathcal{R}}\mathrm{e}^{v(\omega,z)}\right)\phi\,\mathrm{d}\mathbb{P}\label{thingsix}\\=&\inf_{h}\mathrm{ess}\sup_{\omega}\log\sum_{z\in\mathcal{R}}\mathrm{e}^{v(\omega,z)}.\label{thingseven}
\end{align}
Let us explain: In (\ref{thingone}), the first supremum is taken over $\mathcal{E}_k$-measurable probability densities, the second supremum is over $\mathcal{E}_k$-measurable environment kernels, and the infimum is over bounded $\mathcal{B}$-measurable functions. For each $\phi$, the second supremum in (\ref{thingone}) is over a compact set, the integral is concave and continuous in $\hat{\pi}$ and affine (hence convex) in $h$. Apply the minimax theorem of Ky Fan \cite{KyFan} to obtain (\ref{thingtwo}). Evaluate the integral in (\ref{thingtwo}) in two steps by first taking a conditional expectation with respect to $\mathcal{E}_k$. This gives (\ref{thingthree}) where \[v(\omega,z):=\mathbb{E}\left[\log \pi(0,z) + f(\omega,z) + h(\omega) - h(T_z\omega)\left|\mathcal{E}_k\right.\right].\] The integrand in (\ref{thingthree}) is a local function of $\hat{\pi}(\omega,\cdot)$, therefore one can take the supremum inside the integral and obtain (\ref{thingfour}). Apply the method of Lagrange multipliers and see that the supremum in (\ref{thingfour}) is attained at \[\hat{\pi}(\omega,z)=\frac{\mathrm{e}^{v(\omega,z)}}{\sum_{z'\in\mathcal{R}}\mathrm{e}^{v(\omega,z')}}.\] Plugging this back in (\ref{thingfour}) gives (\ref{thingfive}). The integral in (\ref{thingfive}) is convex in $h$, and affine (hence concave) and continuous in $\phi$. Plus, the supremum is taken over a compact set. Apply once again the minimax theorem of Ky Fan \cite{KyFan} and arrive at (\ref{thingsix}) which is clearly equal to (\ref{thingseven}).

Let us proceed with the proof: (\ref{thingseven}) implies that $\forall\epsilon >0$ and $k\geq1$, there exists a bounded $\mathcal{B}$-measurable function $h_{k,\epsilon}$ that $\mathbb{P}$-a.s. satisfies\pagebreak
\begin{equation}
\log\sum_{z\in\mathcal{R}}\exp \mathbb{E}\left[\log \pi(0,z) + f(\omega,z) + h_{k,\epsilon}(\omega) - h_{k,\epsilon}(T_z\omega)\left|\mathcal{E}_k\right.\right]\leq\Gamma(f) + \epsilon.\label{coklugot}
\end{equation}
For each $z\in\mathcal{R}$,
\begin{equation}
\mathbb{E}\left[h_{k,\epsilon}(\omega) - h_{k,\epsilon}(T_z\omega)\left|\mathcal{E}_k\right.\right]\leq\mathbb{E}\left[-\log \pi(0,z)\left|\mathcal{E}_k\right.\right]+\|f\|_{\infty}+\Gamma(f)+\epsilon\label{yarinbitersekral}.
\end{equation} Define $F_{k,\epsilon}:\Omega\times\mathcal{R}\to\mathbb{R}$ by $F_{k,\epsilon}(\omega,z):=\mathbb{E}\left[h_{k,\epsilon}(\omega) - h_{k,\epsilon}(T_z\omega)\left|\mathcal{E}_{k-1}\right.\right]$. Then,
\begin{equation}
F_{k,\epsilon}(\omega,z)\leq\mathbb{E}\left[-\log \pi(0,z)\left|\mathcal{E}_{k-1}\right.\right]+\|f\|_{\infty}+\Gamma(f)+\epsilon\label{kirmizigul}
\end{equation} holds $\mathbb{P}$-a.s. Also note that
\begin{align*}
-\mathbb{E}\left[h_{k,\epsilon}(\omega) - h_{k,\epsilon}(T_z\omega)\left|T_{-z}\mathcal{E}_k\right.\right]&=-\mathbb{E}\left[h_{k,\epsilon}(T_{-z}\omega) - h_{k,\epsilon}(\omega)\left|\mathcal{E}_k\right.\right](T_z\cdot)\\
&=\mathbb{E}\left[h_{k,\epsilon}(\omega) - h_{k,\epsilon}(T_{-z}\omega)\left|\mathcal{E}_k\right.\right](T_z\cdot)\\
&\leq\mathbb{E}\left[-\log \pi(0,-z)\left|\mathcal{E}_k\right.\right](T_z\cdot)+\|f\|_{\infty}+\Gamma(f)+\epsilon\\
&=\mathbb{E}\left[-\log \pi(z,0)\left|T_{-z}\mathcal{E}_k\right.\right]+\|f\|_{\infty}+\Gamma(f)+\epsilon
\end{align*} where the inequality follows from (\ref{yarinbitersekral}). Since $\mathcal{E}_{k-1}\subset T_{-z}\mathcal{E}_k$, taking conditional expectation with respect to $\mathcal{E}_{k-1}$ gives \[-F_{k,\epsilon}(\omega,z)\leq\mathbb{E}\left[-\log \pi(z,0)\left|\mathcal{E}_{k-1}\right.\right]+\|f\|_{\infty}+\Gamma(f)+\epsilon.\] Recall (\ref{kirmizigul}) and deduce that \[\left|F_{k,\epsilon}(\omega,z)\right|\leq\mathbb{E}\left[-\log \pi(0,z)\left|\mathcal{E}_{k-1}\right.\right]+\mathbb{E}\left[-\log \pi(z,0)\left|\mathcal{E}_{k-1}\right.\right]+\|f\|_{\infty}+\Gamma(f)+\epsilon.\] This implies by (\ref{kimimvarki}) that $\left(F_{k,\epsilon}(\cdot,z)\right)_{k\geq1}$ is uniformly bounded in $L^{d+\alpha}(\mathbb{P})$ for $z\in\mathcal{R}$. Passing to a subsequence if necessary, $F_{k,\epsilon}(\cdot,z)$ converges weakly to a limit $F_{\epsilon}(\cdot,z)\in L^{d+\alpha}(\mathbb{P})$.

For $j\geq1$, and any sequence $(x_{i})_{i=0}^n$ in $\mathbb{Z}^d$ such that $x_{i+1}-x_i\in\mathcal{R}$ and $x_0=x_n$, 
\begin{align}
&\mathbb{E}\left(\left.\sum_{i=0}^{n-1}F_{\epsilon}(T_{x_i}\omega,x_{i+1}-x_i)\right|\mathcal{E}_j\right)\nonumber\\
=&\sum_{i=0}^{n-1}\mathbb{E}\left(\left.\lim_{k\to\infty}F_{k,\epsilon}(T_{x_i}\omega,x_{i+1}-x_i)\right|\mathcal{E}_j\right)\nonumber\\
=&\sum_{i=0}^{n-1}\lim_{k\to\infty}\mathbb{E}\left(\left.F_{k,\epsilon}(T_{x_i}\omega,x_{i+1}-x_i)\right|\mathcal{E}_j\right)\nonumber\\
=&\sum_{i=0}^{n-1}\lim_{k\to\infty}\mathbb{E}\left(\left.\mathbb{E}\left[h_{k,\epsilon}(\omega) - h_{k,\epsilon}(T_{x_{i+1}-x_i}\omega)\left|\mathcal{E}_{k-1}\right.\right](T_{x_i}\omega)\right|\mathcal{E}_j\right)\nonumber\\
=&\sum_{i=0}^{n-1}\lim_{k\to\infty}\mathbb{E}\left(\left.\mathbb{E}\left[h_{k,\epsilon}(T_{x_i}\omega) - h_{k,\epsilon}(T_{x_{i+1}}\omega)\left|T_{-x_i}\mathcal{E}_{k-1}\right.\right]\right|\mathcal{E}_j\right)\nonumber\\
=&\sum_{i=0}^{n-1}\lim_{k\to\infty}\mathbb{E}\left(\left.h_{k,\epsilon}(T_{x_i}\omega) - h_{k,\epsilon}(T_{x_{i+1}}\omega)\right|\mathcal{E}_j\right)\label{dursunmus}\\
=&\lim_{k\to\infty}\mathbb{E}\left(\left.\sum_{i=0}^{n-1}\left(h_{k,\epsilon}(T_{x_i}\omega) - h_{k,\epsilon}(T_{x_{i+1}}\omega)\right)\right|\mathcal{E}_j\right)=0\nonumber
\end{align} holds $\mathbb{P}$-a.s., where (\ref{dursunmus}) follows from the fact that $\mathcal{E}_j\subset T_{-x_i}\mathcal{E}_{k-1}$ for $k$ large. Therefore, $\sum_{i=0}^{n-1}F_{\epsilon}(T_{x_i}\omega,x_{i+1}-x_i)=0$ for $\mathbb{P}$-a.e.\ $\omega$, and $F_{\epsilon}:\Omega\times\mathcal{R}\to\mathbb{R}$ satisfies the closed loop condition in Definition \ref{K}. We already know that it satisfies the moment condition, and it is also clearly mean zero. Hence, $F_{\epsilon}\in\mathcal{K}$.

Since $\mathbb{E}\left[\log \pi(0,z) + f(\omega,z)\left|\mathcal{E}_{k-1}\right.\right]$ is an $L^{d+\alpha}(\mathbb{P})$-bounded martingale, it converges in $L^{d+\alpha}(\mathbb{P})$ to $\log \pi(0,z) + f(\cdot,z)$. Therefore, \[\mathcal{L}_{k,\epsilon}(\cdot,z):=\mathbb{E}\left[\log \pi(0,z) + f(\omega,z)\left|\mathcal{E}_{k-1}\right.\right]+F_{k,\epsilon}(\cdot,z)\] converges weakly in $L^{d+\alpha}(\mathbb{P})$ to $\log \pi(0,z) + f(\cdot,z)+F_{\epsilon}(\cdot,z)$. By Mazur's theorem (see \cite{Rudin}), we can find $\mathcal{L}_{k,\epsilon}':\Omega\times\mathcal{R}\to\mathbb{R}$ for $k\geq1$ such that $\mathcal{L}_{k,\epsilon}'(\cdot,z)$ converges strongly in $L^{d+\alpha}(\mathbb{P})$ to $\log \pi(0,z) + f(\cdot,z)+F_{\epsilon}(\cdot,z)$ for each $z\in\mathcal{R}$ and $\mathcal{L}_{k,\epsilon}'$ is a convex combination of $\{\mathcal{L}_{1,\epsilon},\mathcal{L}_{2,\epsilon},\ldots,\mathcal{L}_{k,\epsilon}\}$. Passing to a further subsequence, $\mathcal{L}_{k,\epsilon}'(\cdot,z)$ converges $\mathbb{P}$-a.s.\ to $\log \pi(0,z) + f(\cdot,z)+F_{\epsilon}(\cdot,z)$. Take conditional expectation of both sides of (\ref{coklugot}) with respect to $\mathcal{E}_{k-1}$ and use Jensen's inequality to write \[\log\sum_{z\in\mathcal{R}}\exp\left(\mathbb{E}\left[\log \pi(0,z) + f(\omega,z)\left|\mathcal{E}_{k-1}\right.\right]+F_{k,\epsilon}(\cdot,z)\right)\leq\Gamma(f) + \epsilon.\] Again by Jensen's inequality, $\log\sum_{z\in\mathcal{R}}\exp\left(\mathcal{L}_{k,\epsilon}'(\cdot,z)\right)\leq\Gamma(f) + \epsilon$. Taking $k\to\infty$ gives \[\log\sum_{z\in\mathcal{R}}\pi(0,z)\mathrm{e}^{f(\omega,z)+F_{\epsilon}(\omega,z)}\leq\Gamma(f) + \epsilon\] for $\mathbb{P}$-a.e.\ $\omega$. Theorem \ref{LMGF} is proved.

\subsection{Large deviation principle}\label{LDPproof}

Putting together (\ref{level2ratetilde}) and Theorem \ref{LMGF},
\begin{align*}
\Lambda_q(f)&=\sup_{\mu\in M_{1,s}^{\ll}(\Omega\times\mathcal{R})}\int\sum_{z\in\mathcal{R}}\mathrm{d}\mu(\omega,z)\left(f(\omega,z)-\log \frac{\mathrm{d}\mu(\omega,z)}{\mathrm{d}(\mu)^1(\omega)\pi(0,z)}\right)\\
&=\sup_{\mu\in M_{1,s}^{\ll}(\Omega\times\mathcal{R})}\left\{\left\langle f,\mu\right\rangle - \mathfrak{I}_q(\mu)\right\}\\
&=\sup_{\mu\in M_1(\Omega\times\mathcal{R})}\left\{\left\langle f,\mu\right\rangle - \mathfrak{I}_q(\mu)\right\}\\
&=\ \mathfrak{I}_q^*(f),
\end{align*} the Fenchel-Legendre transform of $\mathfrak{I}_q$. Therefore, $\mathfrak{I}_q^{**}=\Lambda_q^*$.

Since $M_1(\Omega\times\mathcal{R})$ is compact, it directly follows from the G\"{a}rtner-Ellis theorem (see \cite{DZ}) that for any closed subset $C$ of $M_1(\Omega\times\mathcal{R})$ and $\mathbb{P}$-a.e.\ $\omega$,
\begin{displaymath}
\limsup_{n\rightarrow\infty}\frac{1}{n}\log P_o^{\omega}(\nu_{n,X}\in C)\leq-\inf_{\mu\in C}\Lambda_q^*(\mu)=-\inf_{\mu\in C}\mathfrak{I}_q^{**}(\mu).
\end{displaymath}

To conclude the proof of Theorem \ref{level2LDP}, one needs to obtain the LDP lower bound. Note that for any open subset $G$ of $M_1(\Omega\times\mathcal{R})$, $\inf_{\nu\in G}\mathfrak{I}_q^{**}(\nu)=\inf_{\nu\in G}\mathfrak{I}_q(\nu)$. (See \cite{Rockafellar}, page 104.) Therefore, it suffices to show that for any $\mu\in M_{1,s}^{\ll}(\Omega\times\mathcal{R})$, any open set $O$ containing $\mu$, and $\mathbb{P}$-a.e.\ $\omega$,
\begin{equation}
\liminf_{n\rightarrow\infty}\frac{1}{n}\log P_o^{\omega}(\nu_{n,X}\in O)\geq-\mathfrak{I}_q(\mu).\label{LB}
\end{equation} 
Take the pair \[(\hat{\pi},\phi)=\left(\frac{\mathrm{d}\mu}{\mathrm{d}(\mu)^1},\frac{\mathrm{d}(\mu)^1}{\mathrm{d}\mathbb{P}}\right)\] corresponding to a given $\mu\in M_{1,s}^{\ll}(\Omega\times\mathcal{R})$. Then, $\hat{\pi}(\cdot,z)>0$ $\mathbb{P}$-a.s.\ for each $z\in U$, $\phi\in L^1(\mathbb{P})$, and $\phi\,\mathrm{d}\mathbb{P}$ is a $\hat{\pi}$-invariant probability measure. With this notation, (\ref{LB}) becomes \[\liminf_{n\rightarrow\infty}\frac{1}{n}\log P_o^{\omega}(\nu_{n,X}\in O)\geq-\int_{\Omega}\sum_{z\in\mathcal{R}}\hat{\pi}(\omega,z)\log\frac{\hat{\pi}(\omega,z)}{\pi(0,z)}\phi(\omega)\mathrm{d}\mathbb{P}.\] Recall Definition \ref{ortamkeli} and introduce a new measure $R_o^{\hat{\pi},\omega}$ by setting \[\mathrm{d}R_o^{\hat{\pi},\omega}:=\frac{\one_{\nu_{n,X}\in O}}{P_o^{\hat{\pi},\omega}(\nu_{n,X}\in O)}\,\mathrm{d}P_o^{\hat{\pi},\omega}.\]
\begin{align*}
\text{Then,}&\quad\liminf_{n\rightarrow\infty}\frac{1}{n}\log P_o^\omega(\nu_{n,X}\in O)
=\liminf_{n\rightarrow\infty}\frac{1}{n}\log E_o^{\hat{\pi},\omega}\left[\one_{\nu_{n,X}\in O}\,\frac{\mathrm{d}P_o^\omega}{\mathrm{d}P_o^{\hat{\pi},\omega}}\right]\displaybreak\\
=&\liminf_{n\rightarrow\infty}\frac{1}{n}\left(\log P_o^{\hat{\pi},\omega}(\nu_{n,X}\in O)+\log \int\frac{\mathrm{d}P_o^\omega}{\mathrm{d}P_o^{\hat{\pi},\omega}}\mathrm{d}R_o^{\hat{\pi},\omega}\right)\\
\geq&\liminf_{n\rightarrow\infty}\frac{1}{n}\left(\log P_o^{\hat{\pi},\omega}(\nu_{n,X}\in O)- \int\log\frac{\mathrm{d}P_o^{\hat{\pi},\omega}}{\mathrm{d}P_o^\omega}\mathrm{d}R_o^{\hat{\pi},\omega}\right)\\
=&\liminf_{n\rightarrow\infty}\frac{1}{n}\left(\log P_o^{\hat{\pi},\omega}(\nu_{n,X}\in O)-\frac{1}{P_o^{\hat{\pi},\omega}(\nu_{n,X}\in O)} E_o^{\hat{\pi},\omega}\left[\one_{\nu_{n,X}\in O}\,\log\frac{\mathrm{d}P_o^{\hat{\pi},\omega}}{\mathrm{d}P_o^\omega}\right]\right)
\end{align*} where the fourth line uses Jensen's inequality. It follows from Lemma \ref{Kozlov} that $\lim_{n\rightarrow\infty}P_o^{\hat{\pi},\omega}(\nu_{n,X}\in O)=1$. Therefore,
\begin{align*}
\liminf_{n\rightarrow\infty}\frac{1}{n}\log P_o^\omega(\nu_{n,X}\in O)&\geq-\limsup_{n\rightarrow\infty}\frac{1}{n}E_o^{\hat{\pi},\omega}\left[\one_{\nu_{n,X}\in O}\,\log\frac{\mathrm{d}P_o^{\hat{\pi},\omega}}{\mathrm{d}P_o^\omega}\right]\\
&=-\int_{\Omega}\sum_{z\in\mathcal{R}}\hat{\pi}(\omega,z)\log\frac{\hat{\pi}(\omega,z)}{\pi(0,z)}\phi(\omega)\mathrm{d}\mathbb{P}
\end{align*} again by Lemma \ref{Kozlov} and the $L^1$-ergodic theorem. Theorem \ref{level2LDP} is proved. Finally, note that the convexity of $\mathfrak{I}_q$ follows from an argument similar to the proof of Lemma \ref{camilla}.

\section{LDP for the mean velocity}\label{birboyutadonus}

\subsection{Variational formula for the rate function}

\begin{proof}[Proof of Corollary \ref{level1LDP}]
Recall (\ref{ximu}) and observe that \[\xi_{\nu_{n,X}}=\int\sum_{z\in\mathcal{R}}\mathrm{d}\nu_{n,X}(\omega,z)z=\frac{1}{n}\sum_{k=0}^{n-1}\left(X_{k+1}-X_k\right)=\frac{X_n-X_o}{n}.\] Therefore, as noted in Section \ref{egridogru}, Corollary \ref{level1LDP} follows from Theorem \ref{level2LDP} by the contraction principle (see \cite{DZ}), and the rate function is given by (\ref{level1rate}). 

In order to justify (\ref{level1ratetilde}), let us define $J_q:\mathbb{R}^d\rightarrow\mathbb{R}^+$ by $J_q(\xi)=\inf_{\mu\in A_\xi}\mathfrak{I}_q(\mu)$. We would like to show that $J_q\equiv I_q$. Since $\mathfrak{I}_q$ and $\mathfrak{I}_q^{**}$ are convex, $I_q$ and $J_q$ are convex functions on $\mathbb{R}^d$. Therefore, it suffices to show that $J_q^*\equiv I_q^*$. For any $\eta\in\mathbb{R}^d$, define $f_{\eta}\in C_b(\Omega\times\mathcal{R})$ by $f_{\eta}(\omega,z):=\langle z,\eta\rangle$. Recalling (\ref{ximu}),
\begin{eqnarray*}
I_q^*(\eta)&=&\sup_{\xi}\{\langle\eta,\xi\rangle - \inf_{\mu\in A_\xi}\mathfrak{I}_q^{**}(\mu)\}\\
&=&\sup_{\xi}\sup_{\mu\in A_\xi}\{\langle\eta,\xi_{\mu}\rangle - \mathfrak{I}_q^{**}(\mu)\}\\
&=&\sup_{\mu\in M_1(\Omega\times\mathcal{R})}\{\langle f_{\eta},\mu\rangle - \mathfrak{I}_q^{**}(\mu)\}\\
&=&\mathfrak{I}_q^{***}(f_{\eta})=\Lambda_q(f_{\eta}).
\end{eqnarray*}
Similarly, $J_q^*(\eta)=\mathfrak{I}_q^*(f_{\eta})=\Lambda_q(f_{\eta})$. We are done.
\end{proof}

\subsection{An Ansatz for the unique minimizer}

\begin{proof}[Proof of Lemma \ref{lagrange}]
The rate function given by formula (\ref{level1ratetilde}) is
\begin{equation}
I_q(\xi)=\inf_{\mu\in A_\xi\cap M_{1,s}^{\ll}(\Omega\times\mathcal{R})}\int_{\Omega}\sum_{z\in\mathcal{R}} \mathrm{d}\mu(\omega,z)\log\frac{\mathrm{d}\mu(\omega,z)}{\mathrm{d}(\mu)^1(\omega)\pi(0,z)}.\label{buduranan}
\end{equation} Fix $\xi=(\xi_1,\ldots,\xi_d)\in\mathbb{R}^d$ with $|\xi_1|+\cdots+|\xi_d|\leq B$. (Otherwise, $A_\xi$ is empty.) If there exists $\mu_\xi\in A_\xi\cap M_{1,s}^{\ll}(\Omega\times\mathcal{R})$ such that \[\mathrm{d}\mu_\xi(\omega,z)=\mathrm{d}(\mu_\xi)^1(\omega) \pi(0,z)\mathrm{e}^{\langle\theta,z\rangle+F(\omega,z)+ r}\] for some $\theta\in\mathbb{R}^d$, $F\in\mathcal{K}$ and $r\in\mathbb{R}$, then for any $\nu\in A_\xi\cap M_{1,s}^{\ll}(\Omega\times\mathcal{R})$,
\begin{align*}
\mathfrak{I}_q(\nu)&=\int_{\Omega}\sum_{z\in\mathcal{R}} \mathrm{d}\nu(\omega,z)\log\frac{\mathrm{d}\nu(\omega,z)}{\mathrm{d}(\nu)^1(\omega)\pi(0,z)}\\
&=\int_{\Omega}\sum_{z\in\mathcal{R}} \mathrm{d}\nu(\omega,z)\log\frac{\mathrm{d}\nu(\omega,z)\mathrm{e}^{\langle\theta,z\rangle+F(\omega,z)+r}}{\mathrm{d}(\nu)^1(\omega)\pi(0,z)\mathrm{e}^{\langle\theta,z\rangle+F(\omega,z)+r}}\displaybreak\\
&=\int_{\Omega}\sum_{z\in\mathcal{R}} \mathrm{d}\nu(\omega,z)\left(\langle\theta,z\rangle+F(\omega,z)+r+\log\frac{\mathrm{d}\nu(\omega,z)\;\mathrm{d}(\mu_\xi)^1(\omega)}{\mathrm{d}(\nu)^1(\omega)\;\mathrm{d}\mu_\xi(\omega,z)}\right)\\ &=\langle\theta,\xi\rangle+r+\int_{\Omega}\sum_{z\in\mathcal{R}} \mathrm{d}\nu(\omega,z)F(\omega,z)\\&\ \ \ \ +\int_{\Omega}\sum_{z\in\mathcal{R}} \mathrm{d}\nu(\omega,z)\log\frac{\mathrm{d}\nu(\omega,z)\;\mathrm{d}(\mu_\xi)^1(\omega)}{\mathrm{d}(\nu)^1(\omega)\;\mathrm{d}\mu_\xi(\omega,z)}.
\end{align*}
Under the Markov kernel $\frac{\mathrm{d}\nu}{\mathrm{d}(\nu)^1}$ with invariant measure $(\nu)^1$, $\mathbb{P}$-a.s.\[\lim_{n\rightarrow\infty}\frac{1}{n}\sum_{k=0}^{n-1}F(T_{X_k}\omega,X_{k+1}-X_k)=\int_{\Omega}\sum_{z\in\mathcal{R}} \mathrm{d}\nu(\omega,z)F(\omega,z)\] by Lemma \ref{Kozlov} and the ergodic theorem. But the same limit is $0$ by Lemma \ref{GRR}. Therefore,
\begin{equation}
\mathfrak{I}_q(\nu)=\langle\theta,\xi\rangle+r+\int_{\Omega}\sum_{z\in\mathcal{R}} \mathrm{d}\nu(\omega,z)\log\frac{\mathrm{d}\nu(\omega,z)\;\mathrm{d}(\mu_\xi)^1(\omega)}{\mathrm{d}(\nu)^1(\omega)\;\mathrm{d}\mu_\xi(\omega,z)}.
\label{sifirladik}
\end{equation}
By an application of Jensen's inequality, it is easy to see that the integral on the RHS of (\ref{sifirladik}) is nonnegative. Moreover, this integral is zero if and only if $\frac{\mathrm{d}\nu}{\mathrm{d}(\nu)^1}=\frac{\mathrm{d}\mu_\xi}{\mathrm{d}(\mu_\xi)^1}$ holds $(\nu)^1$-a.s.\ and hence $\mathbb{P}$-a.s.\ by Lemma \ref{Kozlov}. Since $(\mu_\xi)^1$ is the unique invariant measure of $\frac{\mathrm{d}\mu_\xi}{\mathrm{d}(\mu_\xi)^1}$ that is absolutely continuous relative to $\mathbb{P}$ (again by Lemma \ref{Kozlov}), $\mu_\xi$ is the unique minimizer of (\ref{buduranan}).
\end{proof}

\section{The one dimensional case}\label{vandiseksin}

In Subsection \ref{construction}, we prove Theorem \ref{findthemin} by constructing a $\mu_\xi\in M_1(\Omega\times U)$ that fits the Ansatz given in Lemma \ref{lagrange} for $\xi\in (-1,\xi_c')\cup(\xi_c,1)$, where $\xi_c$ and $\xi_c'$ naturally appear. Finally, we prove Theorem \ref{density} in Subsection \ref{densitysection}.

\subsection{Construction of the unique minimizer}\label{construction}

Define $\zeta(r,\omega):=E_o^\omega\left[\mathrm{e}^{rt_1},t_1<\infty\right]$ for any $r\in\mathbb{R}$. Then, $\zeta(r,\omega)=\pi(0,1)\mathrm{e}^r+\pi(0,-1)\mathrm{e}^r\zeta(r,T_{-1}\omega)\zeta(r,\omega)$ if $\zeta(r,\omega)$ is finite.
\begin{equation}
1=\pi(0,1)\mathrm{e}^r\zeta(r,\omega)^{-1}+\pi(0,-1)\mathrm{e}^r\zeta(r,T_{-1}\omega).\label{masterof}
\end{equation} Since $\pi(0,-1)>0$ holds $\mathbb{P}$-a.s., $\{\omega:\zeta(r,\omega)<\infty\}$ is $T$-invariant and therefore its probability under $\mathbb{P}$ is $0$ or $1$. $\zeta(r,\omega)$ is strictly increasing in $r$. There exists $r_c\geq0$ such that $\mathbb{P}$-a.s.\ $\zeta(r,\omega)<\infty$ if $r<r_c$ and $\zeta(r,\omega)=\infty$ if $r>r_c$. By (\ref{masterof}), $1\geq\pi(0,-1)\mathrm{e}^r\zeta(r,T_{-1}\omega)$ and $\log\zeta(r,T_{-1}\omega)\leq-\log\pi(0,-1)-r$. Thus, \begin{equation}\lambda(r):=\mathbb{E}[\log\zeta(r,\cdot)]\leq\int|\log\pi(0,-1)|\mathrm{d}\mathbb{P}-r<\infty\label{turran}\end{equation} for $r<r_c$, and also for $r=r_c$ by the monotone convergence theorem. In particular, $\zeta(r_c,\omega)<\infty$ holds for $\mathbb{P}$-a.e.\ $\omega$. It is easy to see that $r\mapsto\lambda(r)$ is analytic and strictly convex for $r<r_c$. Set $\xi_c:=\lambda'(r_c-)^{-1}$ and note that $$\xi_c^{-1}=\lambda'(r_c-)\geq\lambda'(0-)=\mathbb{E}\left(E_o^\omega[\left.t_1\right|t_1<\infty]\right)>1$$ since the ellipticity condition ensures that the walk is not deterministic. 

For any $\xi\in(\xi_c,1)$, there is a unique $r=r(\xi)<r_c$ such that $\xi^{-1}=\lambda'(r)$. For $r=r(\xi)$, recall (\ref{masterof}) and define an environment kernel $\hat{\pi}$ by
\begin{equation}
\hat{\pi}(\omega,1):=\pi(0,1)\mathrm{e}^r\zeta(r,\omega)^{-1}\quad\mbox{and}\quad\hat{\pi}(\omega,-1):=\pi(0,-1)\mathrm{e}^r\zeta(r,T_{-1}\omega).\label{yyedin}
\end{equation}

\begin{lemma}
$P_o^{\hat{\pi}}(t_1<\infty)=1$.
\end{lemma}
\begin{proof}
It suffices to show that $P_o^{\hat{\pi},\omega}(t_{-1}'<\infty)<1$ holds for $\mathbb{P}$-a.e.\ $\omega$. It follows from (\ref{yyedin}) that
\begin{align*}
P_o^{\hat{\pi},\omega}(t_{-1}'<\infty)&=E_o^\omega[\mathrm{e}^{rt_{-1}'}\zeta(r,T_{-1}\omega),t_{-1}'<\infty]\\
&=E_o^\omega[\mathrm{e}^{rt_{-1}'},t_{-1}'<\infty]E_{-1}^\omega[\mathrm{e}^{rt_o},t_o<\infty]\\
&\leq\mathrm{e}^{2(r-r_c)}E_o^\omega[\mathrm{e}^{r_ct_{-1}'},t_{-1}'<\infty]E_{-1}^\omega[\mathrm{e}^{r_ct_o},t_o<\infty].
\end{align*} On the other hand, for any $n\geq1$,
\begin{align*}
E_o^\omega[\mathrm{e}^{r_ct_n},t_n<\infty]&\geq E_o^\omega[\mathrm{e}^{r_ct_n},t_{-1}'<t_n<\infty]\\
&=E_o^\omega[\mathrm{e}^{r_ct_{-1}'},t_{-1}'<t_n]E_{-1}^\omega[\mathrm{e}^{r_ct_n},t_n<\infty]\\
&=E_o^\omega[\mathrm{e}^{r_ct_{-1}'},t_{-1}'<t_n]E_{-1}^\omega[\mathrm{e}^{r_ct_o},t_o<\infty]E_o^\omega[\mathrm{e}^{r_ct_n},t_n<\infty].
\end{align*} Simplify this to get $1\geq E_o^\omega[\mathrm{e}^{r_ct_{-1}'},t_{-1}'<t_n]E_{-1}^\omega[\mathrm{e}^{r_ct_o},t_o<\infty]$. Taking $n\to\infty$ gives $E_o^\omega[\mathrm{e}^{r_ct_{-1}'},t_{-1}'<\infty]E_{-1}^\omega[\mathrm{e}^{r_ct_o},t_o<\infty]\leq1.$ Since $r-r_c<0$, we conclude that $P_o^{\hat{\pi},\omega}(t_{-1}'<\infty)\leq\mathrm{e}^{2(r-r_c)}<1$.
\end{proof}
\begin{lemma}
$E_o^{\hat{\pi}}[t_1]=\xi^{-1}<\infty$.
\end{lemma}
\begin{proof}
For any $s\in\mathbb{R}$ and $\mathbb{P}$-a.e.\ $\omega$, recall (\ref{yyedin}) and observe that
\begin{align*}E_o^{\hat{\pi},\omega}[\mathrm{e}^{st_1}]=E_o^{\hat{\pi},\omega}[\mathrm{e}^{st_1},t_1<\infty]&=E_o^\omega[\mathrm{e}^{(r+s)t_1}\zeta(r,\omega)^{-1},t_1<\infty]\\&=\zeta(r+s,\omega)\zeta(r,\omega)^{-1}.
\end{align*}
Therefore, $\hat{\lambda}(s):=\mathbb{E}\left(\log E_o^{\hat{\pi},\omega}[\mathrm{e}^{st_1}]\right)=\lambda(r+s)-\lambda(r)$ by (\ref{turran}), and \[E_o^{\hat{\pi}}[t_1]=\left.\frac{\mathrm{d}}{\mathrm{d}s}\right|_{s=0}\!\!\!\!\!\hat{\lambda}(s)=\lambda'(r)=\xi^{-1}.\qedhere\]
\end{proof}
Since $\hat{\pi}(\omega,\pm1)>0$ holds $\mathbb{P}$-a.s., there exists a $\phi\in L^1(\mathbb{P})$ such that $\phi\,\mathrm{d}\mathbb{P}$ is a $\hat{\pi}$-invariant probability measure. (See \cite{alili} or Theorem \ref{density}.) The pair $(\hat{\pi},\phi)$ corresponds to a $\mu_\xi\in M_{1,s}^{\ll}(\Omega\times U)$ with $\mathrm{d}(\mu_\xi)^1=\phi\,\mathrm{d}\mathbb{P}$. By Lemma \ref{Kozlov}, the LLN for the mean velocity of the particle holds under $P_o^{\hat{\pi}}$ and the limiting velocity is (recall (\ref{ximu})) \[\int\sum_{z\in\mathcal{R}}\hat{\pi}(\omega,z)z\mathbb{\phi}(\omega)\;\mathrm{d}\mathbb{P}=\xi_{\mu_\xi}.\] Since $E_o^{\hat{\pi}}[t_1]=\xi^{-1}$, $\xi_{\mu_\xi}=\xi$ and therefore $\mu_\xi\in A_\xi$.

Let us define $F:\Omega\times\{-1,1\}\to\mathbb{R}$ by \[F(\omega,-1):=\log\zeta(r,T_{-1}\omega)-\lambda(r)\quad\mbox{and}\quad F(\omega,1):=-\log\zeta(r,\omega)+\lambda(r).\] Then, recall (\ref{yyedin}) and see that
\begin{equation}
\mathrm{d}\mu_\xi(\omega,z)=\hat{\pi}(\omega,z)\phi(\omega)\mathrm{d}\mathbb{P}(\omega)=\mathrm{d}(\mu_\xi)^1(\omega)\pi(0,z)\mathrm{e}^{-z\lambda(r)+F(\omega,z)+r}\label{veriguut}
\end{equation} for $z\in\{-1,1\}$. In order to conclude that $\mu_\xi$ fits the Ansatz given in Lemma \ref{lagrange}, $F\in\mathcal{K}$ remains to be shown. $F$ clearly satisfies the mean zero and the closed loop conditions in Definition \ref{K}. For $z\in\{-1,1\}$,
\[\pi(0,z)\mathrm{e}^{-z\lambda(r)+F(\omega,z)+r}=\hat{\pi}(\omega,z)\leq1\] gives $F(\omega,z)\leq|\log\pi(0,z)| +z\lambda(r)-r$. Since $-F(\omega,z)=F(T_z\omega,-z)$, we can write $|F(\omega,z)|\leq|\log\pi(0,1)|+|\log\pi(1,0)|+|\lambda(r)|-r$ and see that the moment condition on $F(\cdot,z)$ follows from (\ref{pleasant}).

Recalling (\ref{sifirladik}), $I_q(\xi)=\mathfrak{I}_q(\mu_\xi)=r(\xi)-\xi\lambda(r(\xi))$, which agrees with the formula provided by Comets et al.\ \cite{CGZ}.

By replacing $t_1$ by $t_{-1}'$ in the above construction, we can define $\xi_c'\in(-1,0]$ and obtain the minimizer $\mu_\xi$ when $\xi\in(-1,\xi_c')$. Theorem \ref{findthemin} is proved.

\subsection{Ergodic invariant density of the environment MC}\label{densitysection}

Consider random walk with bounded jumps on $\mathbb{Z}$ in a stationary and ergodic random environment.

\begin{lemma}\label{maxprinciple}
Given an environment kernel $\hat{\pi}$ for which $\hat{\pi}(\omega,1)>0$ holds $\mathbb{P}$-a.s., if a bounded measurable function $u:\Omega\times\mathbb{Z}\to\mathbb{R}$ satisfies \[u(\omega,x)=\sum_{z\in\mathcal{R}}\hat{\pi}(T_x\omega,z)u(\omega,x+z)\] for $\mathbb{P}$-a.e.\ $\omega$ when $|x|$ is large, then $\lim_{x\to-\infty}u(\cdot,x)$ and $\lim_{x\to\infty}u(\cdot,x)$ exist $\mathbb{P}$-a.s.
\end{lemma}

\begin{proof}
Since $\mathbb{P}\left\{\omega:\hat{\pi}(T_z\omega,1)>0\ \forall z\in\mathcal{R}\right\}=1$, $$\mathbb{P}\left\{\omega:\hat{\pi}(T_z\omega,1)\geq\beta\ \forall z\in\mathcal{R}\right\}>0$$ for any small $\beta>0$. The ergodicity of the environment implies that for $\mathbb{P}$-a.e.\ $\omega$, there is a (random) sequence $y_j\to\infty$ such that $\hat{\pi}(T_{y_j+z}\omega,1)\geq\beta$ for each $z\in\mathcal{R}$. Define $W(\omega):=\{y_j-z:j\geq1,0\leq z<B\}$. Since the jumps of the walk under the kernel $\hat{\pi}$ are bounded by $B$, it follows from the maximum principle that \[u(\omega,\infty):=\limsup_{\substack{x\rightarrow\infty\\x\in W(\omega)}}u(\omega,x)=\limsup_{x\to\infty}u(\omega,x).\] So, there exists a sequence $x_k\to\infty$ in $W(\omega)$ such that $u(\omega,x_k)\to u(\omega,\infty)$. For any $\epsilon>0$,
\begin{equation}
u(\omega,x_k+z)-u(\omega,x_k)<\epsilon\label{lizkam}
\end{equation} when $k$ is large and $z\in\mathcal{R}$. It follows by construction that $\hat{\pi}(T_{x_k+z'}\omega,1)\geq\beta$ for each $z'=0,\ldots,B$. Therefore, if $u(\omega,x_k)\geq u(\omega,x_k+1)$, then
\begin{align*}
&\beta [u(\omega,x_k)-u(\omega,x_k+1)]\leq \hat{\pi}(T_{x_k}\omega,1)[u(\omega,x_k)-u(\omega,x_k+1)]\\&=-\sum_{z\neq 1}\hat{\pi}(T_{x_k}\omega,z)[u(\omega,x_k)-u(\omega,x_k+z)]<\epsilon
\end{align*} holds for large $k$, which (in combination with setting $z=1$ in (\ref{lizkam})) implies that $u(\omega,x_k+1)\to u(\omega,\infty)$. Iterating this shows that $u(\omega,x_k+z')\to u(\omega,\infty)$ for each $z'=0,\ldots,B-1$. Again by the maximum principle, $u(\omega,x)\to u(\omega,\infty)$ as $x\to\infty$. The existence of $\lim_{x\to-\infty}u(\omega,x)$ is proved the same way.
\end{proof}

\begin{proof}[Proof of Theorem \ref{density}]

Denoting the walk as usual by $\left(X_k\right)_{k\geq0}$, consider the hitting time $V_o:=\inf\{k\geq0:\,X_k=0\}$ and set $\psi(\omega,x):=P_x^{\hat{\pi},\omega}(V_o<\infty)$ for $x\in\mathbb{Z}$. It follows from these definitions that whenever $x\neq0$, \[\psi(\omega,x)=\sum_{z\in\mathcal{R}}\hat{\pi}(T_x\omega,z)\psi(\omega,x+z)\] holds. It is easy to see that the function $\phi(\omega,x):=E_x^{\hat{\pi},\omega}\left[\sum_{k=0}^\infty\one_{X_k=0}\right]$ satisfies $\phi(\omega,x)=\psi(\omega,x)\phi(\omega,0)$. Hence, \[\phi(\omega)=\lim_{x\rightarrow-\infty}\phi(\omega,x)=\phi(\omega,0)\lim_{x\rightarrow-\infty}\psi(\omega,x)\] exists for $\mathbb{P}$-a.e.\ $\omega$ by Lemma \ref{maxprinciple}. Since the walk is transient to the right and has bounded jumps, the ellipticity condition ensures that $\phi>0$ holds $\mathbb{P}$-a.s. This proves part (a) of the theorem.

Let us now show that $\phi\in L^1(\mathbb{P})$:
\begin{align}
\sum_{y=0}^{N-1}\phi(T_y\omega)&=\sum_{y=0}^{N-1}\lim_{x\rightarrow-\infty}E_x^{\hat{\pi},T_y\omega}\left[\sum_{k=0}^\infty\one_{X_k=0}\right]=\sum_{y=0}^{N-1}\lim_{x\rightarrow-\infty}E_x^{\hat{\pi},\omega}\left[\sum_{k=0}^\infty\one_{X_k=y}\right]\nonumber\\
&=\!\!\lim_{x\rightarrow-\infty}\!\!E_x^{\hat{\pi},\omega}\left[\#\{k\geq0:\,0\leq X_k\leq N-1\}\right]\nonumber\\
&\leq\!\!\lim_{x\rightarrow-\infty}\!\!E_x^{\hat{\pi},\omega}\left[t_N-t_o\right]+\!\!\lim_{x\rightarrow-\infty}\!\!E_x^{\hat{\pi},\omega}\left[\#\{k\geq t_N:\,X_k\leq N-1\}\right].\label{asil}
\end{align} Here, $\#$ denotes the number of elements of a set. In order to control the second term in (\ref{asil}), define a new random time $S:=\inf\{k\geq t_{-1}': X_k\geq0\}$. Since the walk is transient, $P_o^{\hat{\pi},\omega}(t_{-1}'=\infty)>0\ \mathbb{P}$-a.s.\ and $P_o^{\hat{\pi}}(S<\infty\,|\,t_{-1}'<\infty)=1$. Note that if $X_o\geq0$, then $-B\leq X_{t_{-1}'}\leq-1$ and $0\leq X_S\leq B-1$. For any $x$ that satisfies $0\leq x\leq B-1$,
\begin{align*}
&E_x^{\hat{\pi},\omega}\left[\#\{k\geq0:\,X_k\leq-1\}\right]\\
=&E_x^{\hat{\pi},\omega}\left[\#\{k\geq0:\,X_k\leq-1\},t_{-1}'<\infty\right]\\
=&P_x^{\hat{\pi},\omega}(t_{-1}'<\infty)E_x^{\hat{\pi},\omega} E_{X_{t_{-1}'}}^{\hat{\pi},\omega}\left[\#\{k\geq0:\,X_k\leq-1\}\right]\\
=&P_x^{\hat{\pi},\omega}(t_{-1}'<\infty)E_x^{\hat{\pi},\omega}\left[E_{X_{t_{-1}'}}^{\hat{\pi},\omega}\left[t_o\right]+E_{X_S}^{\hat{\pi},\omega}\left[\#\{k\geq0:\,X_k\leq-1\}\right]\right].
\end{align*}
Letting $h_B(\omega):=\max_{0\leq x\leq B-1}E_x^{\hat{\pi},\omega}\left[\#\{k\geq0:\,X_k\leq-1\}\right]$, \[h_B(\omega)\leq\max_{0\leq x\leq B-1}P_x^{\hat{\pi},\omega}(t_{-1}'<\infty)\left(\max_{-B\leq y\leq-1}E_y^{\hat{\pi},\omega}\left[t_o\right]+h_B(\omega)\right).\] Therefore, \[h_B(\omega)\leq\frac{\max_{0\leq x\leq B-1}P_x^{\hat{\pi},\omega}(t_{-1}'<\infty)}{\min_{0\leq x\leq B-1}P_x^{\hat{\pi},\omega}(t_{-1}'=\infty)}\max_{-B\leq y\leq-1}E_y^{\hat{\pi},\omega}\left[t_o\right]<\infty\] holds $\mathbb{P}$-a.s.\ since $E_o^{\hat{\pi}}\left[t_1\right]<\infty$. Because the environment is ergodic under shifts, there is a constant $C$ such that for $\mathbb{P}$-a.e.\ $\omega$, there is a sequence $N_j=N_j(\omega)\to\infty$ for which \[\lim_{x\rightarrow-\infty}E_x^{\hat{\pi},\omega}\left[\#\{k\geq t_{N_j}:\,X_k\leq N_j-1\}\right]\leq h_B(T_{N_j}\omega)\leq C.\] This controls the second term in (\ref{asil}). By the ergodic theorem,
\begin{align*}
&\left\|\phi\right\|_{L^1(\mathbb{P})}=\lim_{N_j\rightarrow\infty}\frac{1}{N_j}\sum_{y=0}^{N_j-1}\phi(T_y\omega)\leq\lim_{N_j\rightarrow\infty}\frac{1}{N_j}\lim_{x\rightarrow-\infty}E_x^{\hat{\pi},\omega}\left[t_{N_j}-t_o\right]\\
&=\lim_{N_j\rightarrow\infty}\frac{1}{N_j}\lim_{x\rightarrow-\infty}\sum_{y=0}^{N_j-1}E_x^{\hat{\pi},\omega}\left[t_{y+1}-t_y\right]\leq\lim_{N_j\rightarrow\infty}\frac{1}{N_j}\sum_{y=0}^{N_j-1}E_o^{\hat{\pi},T_y\omega}\left[t_1\right]=E_o^{\hat{\pi}}\left[t_1\right].
\end{align*} This proves part (b) of the theorem. Finally, note that
\begin{align*}
&\sum_{z\in\mathcal{R}}E_{x+z}^{\hat{\pi},T_{-z}\omega}\left[\sum_{k=0}^\infty\one_{X_k=0}\right]\hat{\pi}(T_{-z}\omega,z)=\sum_{z\in\mathcal{R}}E_x^{\hat{\pi},\omega}\left[\sum_{k=0}^\infty\one_{X_k=-z}\right]\hat{\pi}(T_{-z}\omega,z)\\&=E_x^{\hat{\pi},\omega}\left[\sum_{k=0}^\infty\one_{X_{k+1}=0}\right]=E_x^{\hat{\pi},\omega}\left[\sum_{k=0}^\infty\one_{X_{k}=0}\right]
\end{align*}holds whenever $x\neq0$. Let $x\to-\infty$ to conclude that for $\mathbb{P}$-a.e.\ $\omega$, \[\sum_{z\in\mathcal{R}}\phi(T_{-z}\omega)\hat{\pi}(T_{-z}\omega,z)=\phi(\omega).\]This proves part (c) of the theorem.
\end{proof}

\chapter{Averaged large deviations for RWRE}\label{averagedchapter}

\section{Strict convexity and analyticity}\label{aLDPregularitysection}

Recall the notation and assumptions introduced in Section \ref{franzek}. Our results on averaged large deviations make frequent use of the following
\begin{lemma}[Sznitman]\label{Szestimates}
\begin{itemize}
\item[(a)] $P_o\left(D=\infty\right)>0$.
\item[(b)] If the walk is non-nestling, then $\exists\,c_3>0$ such that $E_o\left[\mathrm{e}^{2c_3\tau_1}\right]<\infty$.
\item[(c)] If the walk is nestling, then $\exists\,c_4>0$ such that $E_o\left[\sup_{1\leq n\leq\tau_1}\mathrm{e}^{c_4\left|X_n\right|}\right]<\infty$. For $d\geq2$, $\tau_1$ has finite $P_o$-moments of arbitrary order.
\end{itemize}
\end{lemma}
\begin{remark}
See Lemma 1.1, Theorem 2.1, Proposition 1.4 and Theorem 3.5 of \cite{SznitmanSlowdown} for the proofs of these statements.
\end{remark}

\subsection{Logarithmic moment generating function}

\begin{lemma}\label{phillysh}
$E_o\left[\left.\mathrm{e}^{\langle\theta,X_{\tau_1}\rangle-\Lambda_a(\theta)\tau_1}\right|D=\infty\right]\leq1$ for every $\theta\in\mathbb{R}^d$.
\end{lemma}
\begin{proof}
For every $m,L\in\mathbb{N}$ and $\epsilon>0$,
\begin{align*}
E_o\left[\mathrm{e}^{\langle\theta,X_{\tau_m}\rangle-\left(\Lambda_a(\theta)+\epsilon\right)\tau_m}\right]&=\sum_{j=L}^\infty\int_{\frac{j}{L}\leq\frac{\tau_m}{m}<\frac{j+1}{L}}\mathrm{e}^{\langle\theta,X_{\tau_m}\rangle-\left(\Lambda_a(\theta)+\epsilon\right)\tau_m}\mathrm{d}P_o\\
&\leq\sum_{j=L}^\infty\int\mathrm{e}^{\langle\theta,X_{\frac{jm}{L}}\rangle+2|\theta|\frac{m}{L}-\left(\Lambda_a(\theta)+\epsilon\right)\frac{jm}{L}}\mathrm{d}P_o\\
&=\mathrm{e}^{2|\theta|\frac{m}{L}}\sum_{j=L}^\infty E_o\left[\mathrm{e}^{\langle\theta,X_{\frac{jm}{L}}\rangle}\right]\mathrm{e}^{-\left(\Lambda_a(\theta)+\epsilon\right)\frac{jm}{L}}\\
&=\mathrm{e}^{2|\theta|\frac{m}{L}}\sum_{j=L}^\infty\mathrm{e}^{o\left(\frac{jm}{L}\right)-\epsilon\frac{jm}{L}}.
\end{align*}Therefore, $$E_o\left[\mathrm{e}^{\langle\theta,X_{\tau_m}\rangle-\left(\Lambda_a(\theta)+\epsilon\right)\tau_m}\right]\leq\mathrm{e}^{2|\theta|\frac{m}{L}}\sum_{j=L}^\infty\mathrm{e}^{-\frac{\epsilon}{2}\frac{jm}{L}}=\frac{\mathrm{e}^{\left(\frac{2|\theta|}{L}-\frac{\epsilon}{2}\right)m}}{1-\mathrm{e}^{-\frac{\epsilon m}{2L}}}$$ holds for large $m$, uniformly in $L$. Taking $L={8|\theta|}/{\epsilon}$ gives $$\lim_{m\to\infty}\frac{1}{m}\log E_o\left[\mathrm{e}^{\langle\theta,X_{\tau_m}\rangle-\left(\Lambda_a(\theta)+\epsilon\right)\tau_m}\right]\leq-\frac{\epsilon}{4}.$$ We also know that $E_o\left[\mathrm{e}^{\langle\theta,X_{\tau_m}\rangle-\left(\Lambda_a(\theta)+\epsilon\right)\tau_m}\right]$ is equal to $$E_o\left[\mathrm{e}^{\langle\theta,X_{\tau_1}\rangle-\left(\Lambda_a(\theta)+\epsilon\right)\tau_1}\right]E_o\left[\left.\mathrm{e}^{\langle\theta,X_{\tau_1}\rangle-\left(\Lambda_a(\theta)+\epsilon\right)\tau_1}\right|D=\infty\right]^{m-1}$$ by the renewal structure. Hence, $\log E_o\left[\left.\mathrm{e}^{\langle\theta,X_{\tau_1}\rangle-\left(\Lambda_a(\theta)+\epsilon\right)\tau_1}\right|D=\infty\right]\leq-\frac{\epsilon}{4}$. The desired result is obtained by taking $\epsilon\to 0$ and applying the monotone convergence theorem.
\end{proof}

With $c_3$ and $c_4$ as in Lemma \ref{Szestimates}, recall the definition of $\mathcal{C}$ in Lemma \ref{berkeleyolursa}:
\begin{itemize}
\item [(a)] If the walk is non-nestling, $\mathcal{C}:=\left\{\theta\in\mathbb{R}^d:|\theta|<c_3\right\}$.
\item [(b)] If the walk is nestling, $\mathcal{C}:=\left\{\theta\in\mathbb{R}^d:|\theta|<c_4\,, \Lambda_a(\theta)>0\right\}$.
\end{itemize}
By Jensen's inequality, 
\begin{align}
\langle\theta,\xi_o\rangle=\lim_{n\to\infty}\frac{1}{n} E_o\left[\langle\theta,X_n\rangle\right]&\leq\lim_{n\to\infty}\frac{1}{n}\log E_o\left[\mathrm{e}^{\langle\theta,X_n\rangle}\right]=\Lambda_a(\theta)\label{noldush}\\&\leq\lim_{n\to\infty}\frac{1}{n}\log E_o\left[\mathrm{e}^{|\theta|n}\right]=|\theta|.\nonumber
\end{align}In the nestling case, $\left\{\theta\in\mathbb{R}^d:|\theta|<c_4,\,\langle\theta,\xi_o\rangle>0\right\}\subset\mathcal{C}$ by (\ref{noldush}). Hence, $\mathcal{C}$ is a non-empty open set both for nestling and non-nestling walks.

\begin{lemma}\label{memenguzel}
If the walk is non-nestling and $\theta\in\mathcal{C}$, then
$$E_o\left[\left.\mathrm{e}^{\langle\theta,X_{\tau_1}\rangle-\Lambda_a(\theta)\tau_1}\right|D=\infty\right]=1.$$
\end{lemma}

\begin{proof}
Given any $\theta\in\mathcal{C}$ and $\epsilon>0$, it follows from Lemma \ref{Szestimates} that whenever $2|\theta|+\epsilon<2c_3$, $$E_o\left[\left.\mathrm{e}^{\langle\theta,X_{\tau_1}\rangle-\left(\Lambda_a(\theta)-\epsilon\right)\tau_1}\right|D=\infty\right]\leq E_o\left[\left.\mathrm{e}^{\left(2|\theta|+\epsilon\right)\tau_1}\right|D=\infty\right]<\infty.$$

For any $r\in\mathbb{R}$ with $|\theta|+|r|<2c_3$ and $E_o\left[\left.\mathrm{e}^{\langle\theta,X_{\tau_1}\rangle-r\tau_1}\right|D=\infty\right]\leq1$,
\begin{align*}
\Lambda_a(\theta)-r&=\lim_{n\to\infty}\frac{1}{n}\log E_o\left[\mathrm{e}^{\langle\theta,X_n\rangle-rn}\right]\leq\lim_{n\to\infty}\frac{1}{n}\log E_o\left[\left.\sup_{1\leq u\leq\tau_1}\mathrm{e}^{\langle\theta,X_u\rangle-ru}\right|D=\infty\right]\\&\leq\lim_{n\to\infty}\frac{1}{n}\log E_o\left[\left.\mathrm{e}^{\left(|\theta|+|r|\right)\tau_1}\right|D=\infty\right]=0
\end{align*}again by Lemma \ref{Szestimates}. Therefore, $1<E_o\left[\left.\mathrm{e}^{\langle\theta,X_{\tau_1}\rangle-\left(\Lambda_a(\theta)-\epsilon\right)\tau_1}\right|D=\infty\right]<\infty$. By the monotone convergence theorem, $E_o\left[\left.\mathrm{e}^{\langle\theta,X_{\tau_1}\rangle-\Lambda_a(\theta)\tau_1}\right|D=\infty\right]\geq1$. Combined with Lemma \ref{phillysh}, this gives the desired result.
\end{proof}

\begin{lemma}\label{gotunguzel}
If the walk is nestling and $\theta\in\mathcal{C}$, then
$$E_o\left[\left.\mathrm{e}^{\langle\theta,X_{\tau_1}\rangle-\Lambda_a(\theta)\tau_1}\right|D=\infty\right]=1.$$
\end{lemma}

\begin{proof}
Given any $\theta\in\mathcal{C}$ and $\epsilon>0$, $$E_o\left[\left.\mathrm{e}^{\langle\theta,X_{\tau_1}\rangle-\left(\Lambda_a(\theta)-\epsilon\right)\tau_1}\right|D=\infty\right]\leq E_o\left[\left.\mathrm{e}^{|\theta|\left|X_{\tau_1}\right|}\right|D=\infty\right]<\infty$$ follows from Lemma \ref{Szestimates} whenever $\Lambda_a(\theta)-\epsilon>0$.

For any $r\geq0$ with $E_o\left[\left.\mathrm{e}^{\langle\theta,X_{\tau_1}\rangle-r\tau_1}\right|D=\infty\right]\leq1$,
\begin{align}
\Lambda_a(\theta)-r&=\lim_{n\to\infty}\frac{1}{n}\log E_o\left[\mathrm{e}^{\langle\theta,X_n\rangle-rn}\right]\leq\lim_{n\to\infty}\frac{1}{n}\log E_o\left[\left.\sup_{1\leq u\leq\tau_1}\mathrm{e}^{\langle\theta,X_u\rangle-ru}\right|D=\infty\right]\nonumber\\&\leq\lim_{n\to\infty}\frac{1}{n}\log E_o\left[\left.\sup_{1\leq u\leq\tau_1}\mathrm{e}^{|\theta|\left|X_{u}\right|}\right|D=\infty\right]=0\label{neguzelbessh}
\end{align}again by Lemma \ref{Szestimates}. Therefore, $1<E_o\left[\left.\mathrm{e}^{\langle\theta,X_{\tau_1}\rangle-\left(\Lambda_a(\theta)-\epsilon\right)\tau_1}\right|D=\infty\right]<\infty$. By the monotone convergence theorem, $E_o\left[\left.\mathrm{e}^{\langle\theta,X_{\tau_1}\rangle-\Lambda_a(\theta)\tau_1}\right|D=\infty\right]\geq1$. Combined with Lemma \ref{phillysh}, this gives the desired result.
\end{proof}

\begin{lemma}\label{geldiiksh}
If the walk is nestling and $|\theta|<c_4$, then:
\begin{itemize}
\item[(a)] $\theta\not\in\mathcal{C}$ if and only if $E_o\left[\left.\mathrm{e}^{\langle\theta,X_{\tau_1}\rangle}\right|D=\infty\right]\leq1$.
\item[(b)] $\theta\in\partial\mathcal{C}$ if and only if $E_o\left[\left.\mathrm{e}^{\langle\theta,X_{\tau_1}\rangle}\right|D=\infty\right]=1$. 
\end{itemize}
\end{lemma}

\begin{proof}
$$0=I_a(0)=\sup_{\theta\in\mathbb{R}^d}\left\{\langle\theta,0\rangle - \Lambda_a(\theta)\right\}=-\inf_{\theta\in\mathbb{R}^d}\Lambda_a(\theta).$$ In other words, $\Lambda_a(\theta)\geq0$ for every $\theta\in\mathbb{R}^d$. If $|\theta|<c_4$ and $\theta\not\in\mathcal{C}$, then $\Lambda_a(\theta)=0$ and $E_o\left[\left.\mathrm{e}^{\langle\theta,X_{\tau_1}\rangle}\right|D=\infty\right]\leq1$ by Lemma \ref{phillysh}. Conversely, if $|\theta|<c_4$ and $E_o\left[\left.\mathrm{e}^{\langle\theta,X_{\tau_1}\rangle}\right|D=\infty\right]\leq1$, then $\Lambda_a(\theta)=0$ follows from (\ref{neguzelbessh}) by setting $r=0$. This proves part (a).

If $|\theta|<c_4$ and $\theta\in\partial\mathcal{C}$, then $\Lambda_a(\theta)=0$. Take $\theta_n\in\mathcal{C}$ such that $\theta_n\to\theta$. It follows from Lemma \ref{gotunguzel} that $E_o\left[\left.\mathrm{e}^{\langle\theta_n,X_{\tau_1}\rangle-\Lambda_a(\theta_n)\tau_1}\right|D=\infty\right]=1$. Since $\Lambda_a$ is continuous at $\theta$ and $\mathrm{e}^{\langle\theta_n,X_{\tau_1}\rangle-\Lambda_a(\theta_n)\tau_1}\leq\mathrm{e}^{c_4\left|X_{\tau_1}\right|}$, Lemma \ref{Szestimates} and the dominated convergence theorem imply that $E_o\left[\left.\mathrm{e}^{\langle\theta,X_{\tau_1}\rangle}\right|D=\infty\right]=1$.

$\Lambda_a$ is a convex function, and therefore $\{\theta\in\mathbb{R}^d:\Lambda_a(\theta)=0\}$ is convex. If $\theta$ is an interior point of this set, then $\theta=t\theta_1+(1-t)\theta_2$ for some $t\in(0,1)$ and $\theta_1,\theta_2\in\mathbb{R}^d$ such that $\theta_1\neq\theta_2$ and $E_o\left[\left.\mathrm{e}^{\langle\theta_i,X_{\tau_1}\rangle}\right|D=\infty\right]\leq1$ for $i=1,2$. By Jensen's inequality, $E_o\left[\left.\mathrm{e}^{\langle\theta,X_{\tau_1}\rangle}\right|D=\infty\right]<1$. The contraposition of this argument concludes the proof of part (b).
\end{proof}

\begin{proof}[Proof of Lemma \ref{berkeleyolursa}]
Consider the function $\psi:\mathbb{R}^d\times\mathbb{R}\to\mathbb{R}$ defined by 
\begin{equation}\label{psish}
\psi(\theta,r):=E_o\left[\left.\mathrm{e}^{\langle\theta,X_{\tau_1}\rangle-r\tau_1}\right|D=\infty\right].
\end{equation} When $\theta\in\mathcal{C}$ and $|r-\Lambda_a(\theta)|$ is small, Lemmas \ref{memenguzel} and \ref{gotunguzel} show that $\psi(\theta,r)<\infty$ and $\psi(\theta,\Lambda_a(\theta))=1$. It is clear that $(\theta,r)\mapsto\psi(\theta,r)$ is analytic at such $(\theta,r)$. Since $\left.\partial_r\psi(\theta,r)\right|_{r=\Lambda_a(\theta)}=-E_o\left[\left.\tau_1\mathrm{e}^{\langle\theta,X_{\tau_1}\rangle-\Lambda_a(\theta)\tau_1}\right|D=\infty\right]\leq-1\neq0$, the implicit function theorem applies and $\theta\mapsto\Lambda_a(\theta)$ is analytic on $\mathcal{C}$.

Differentiating both sides of $\psi(\theta,\Lambda_a(\theta))=1$ with respect to $\theta$ gives
\begin{equation}\label{jacobiansh}
E_o\left[\left.\left(X_{\tau_1}-\nabla\Lambda_a(\theta)\tau_1\right)\mathrm{e}^{\langle\theta,X_{\tau_1}\rangle-\Lambda_a(\theta)\tau_1}\right|D=\infty\right]=0.
\end{equation} Differentiating once again, we see that the Hessian $H_a$ of $\Lambda_a$ satisfies
\begin{equation}\label{hessiansh}
\langle\hat{v},H_a(\theta)\hat{v}\rangle=\frac{E_o\left[\left.\langle X_{\tau_1}-\nabla\Lambda_a(\theta)\tau_1,\hat{v}\rangle^2\,\mathrm{e}^{\langle\theta,X_{\tau_1}\rangle-\Lambda_a(\theta)\tau_1}\right|D=\infty\right]}{E_o\left[\left.\tau_1\mathrm{e}^{\langle\theta,X_{\tau_1}\rangle-\Lambda_a(\theta)\tau_1}\right|D=\infty\right]}>0
\end{equation} for any unit vector $\hat{v}\in\mathbb{R}^d$. Hence, $\Lambda_a$ is strictly convex on $\mathcal{C}$.
\end{proof}

\subsection{Rate function}

\begin{proof}[Proof of Theorem \ref{qual}]
$\Lambda_a$ is analytic on $\mathcal{C}$ by Lemma \ref{berkeleyolursa}. The Hessian of $\Lambda_a$ is positive definite at any $\theta\in\mathcal{C}$ by (\ref{hessiansh}). Therefore, for any $\xi\in\mathcal{A}$, there exists a unique $\theta=\theta(\xi)\in\mathcal{C}$ with $\xi=\nabla\Lambda_a(\theta)$. $\mathcal{A}$ is open since it is the pre-image of $\mathcal{C}$ under the map $\xi\mapsto\theta(\xi)$ which is analytic by the inverse function theorem. Since $$I_a(\xi)=\sup_{\theta'\in\mathbb{R}^d}\left\{\langle\theta',\xi\rangle-\Lambda_a(\theta')\right\}=\langle\theta(\xi),\xi\rangle-\Lambda_a(\theta(\xi)),$$ $I_a$ is analytic at $\xi$. The strict convexity of $I_a$ on $\mathcal{A}$ follows from the differentiability of $\Lambda_a$ on $\mathcal{C}$ by a standard argument. (See \cite{Rockafellar}.)

Note that (\ref{jacobiansh}) gives
\begin{equation}\label{babosh}
\nabla\Lambda_a(\theta)=\frac{E_o\left[\left.X_{\tau_1}\mathrm{e}^{\langle\theta,X_{\tau_1}\rangle-\Lambda_a(\theta)\tau_1}\right|D=\infty\right]}{E_o\left[\left.\tau_1\mathrm{e}^{\langle\theta,X_{\tau_1}\rangle-\Lambda_a(\theta)\tau_1}\right|D=\infty\right]}
\end{equation} for every $\theta\in\mathcal{C}$. If the walk is non-nestling, then $0\in\mathcal{C}$ and $$\xi_o=\frac{E_o\left[\left.X_{\tau_1}\right|D=\infty\right]}{E_o\left[\left.\tau_1\right|D=\infty\right]}=\nabla\Lambda_a(0)\in\mathcal{A}$$ by (\ref{huseysh}). This proves part (a).

The rest of this proof focuses on the nestling case. Recall the definition of $\psi$ in (\ref{psish}). By Lemma \ref{geldiiksh}, $$\{\theta\in\mathbb{R}^d: |\theta|<c_4\}\cap\partial\mathcal{C}=\{\theta\in\mathbb{R}^d: |\theta|<c_4\,,\psi(\theta,0)=1\}.$$ In particular, $0\in\partial\mathcal{C}$. Take any $(\theta_n)_{n\geq1}$ with $\theta_n\in\mathcal{C}$ such that $\theta_n\to 0$. Then, any limit point of $(\nabla\Lambda_a(\theta_n))_{n\geq1}$ belongs to $\partial\mathcal{A}$. When $d=1$, (\ref{babosh}) implies that
\begin{align}
\limsup_{n\to\infty}\nabla\Lambda_a(\theta_n)&=\limsup_{n\to\infty}\frac{E_o\left[\left.X_{\tau_1}\mathrm{e}^{\langle\theta_n,X_{\tau_1}\rangle-\Lambda_a(\theta_n)\tau_1}\right|D=\infty\right]}{E_o\left[\left.\tau_1\mathrm{e}^{\langle\theta_n,X_{\tau_1}\rangle-\Lambda_a(\theta_n)\tau_1}\right|D=\infty\right]}\label{nediyonsh}\\&\leq\frac{E_o\left[\left.X_{\tau_1}\right|D=\infty\right]}{E_o\left[\left.\tau_1\right|D=\infty\right]}=\xi_o\label{nediyonggsh}
\end{align}
where we assume WLOG that $\hat{u}=1$. The numerator in (\ref{nediyonsh}) converges to the numerator in (\ref{nediyonggsh}) by Lemma \ref{Szestimates} and the dominated convergence theorem. The denominator in (\ref{nediyonggsh}) bounds the liminf of the denominator in (\ref{nediyonsh}) by Fatou's lemma. $[0,\xi_o]\cap\mathcal{A}$ is empty since $I_a$ is linear on $[0,\xi_o]$. Therefore, $\liminf_{n\to\infty}\nabla\Lambda_a(\theta_n)\geq\xi_o$. Hence, $\xi_o=\lim_{n\to\infty}\nabla\Lambda_a(\theta_n)\in\partial\mathcal{A}$.

When $d\geq2$, $\tau_1$ has finite $P_o$-moments of arbitrary order and it is easy to see from (\ref{babosh}) that $\xi_o=\lim_{n\to\infty}\nabla\Lambda_a(\theta_n)\in\partial\mathcal{A}$ again by Lemma \ref{Szestimates} and the dominated convergence theorem. The map $\theta\mapsto\psi(\theta,0)$ is analytic for $|\theta|<c_4$, and $$\left.\nabla_\theta\psi(\theta,0)\right|_{\theta=0}=E_o\left[\left.X_{\tau_1}\right|D=\infty\right]=E_o\left[\left.\tau_1\right|D=\infty\right]\xi_o$$ is normal to $\partial\mathcal{C}$ at $0$ by Lemma \ref{geldiiksh}. The RHS of (\ref{hessiansh}) smoothly extends to $\bar{\mathcal{C}}\cap\{\theta\in\mathbb{R}^d:|\theta|<c_4\}$. Refer to the extension again by $H_a$. The unit vector $\eta_o$ normal to $\partial\mathcal{A}$ at $\xi_o$ is $c_5H_a(0)\xi_o$ for some $c_5>0$ by the chain rule, and satisfies $$\langle\eta_o,\xi_o\rangle=c_5\langle\xi_o,H_a(0)\xi_o\rangle>0.\qedhere$$
\end{proof}

\section{Minimizer of Varadhan's variational formula}\label{aLDPminimizersection}

\subsection{Existence of the minimizer}

Varadhan's variational formula for the rate function $I_a$ is
\begin{equation}\label{divaneasiksh}
I_a(\xi)=\inf_{\substack{\mu\in\mathcal{E}\\m(\mu)=\xi}}\mathfrak{I}_a(\mu).
\end{equation} Since ergodic measures on $W_\infty^{\mathrm{tr}}$ have disjoint supports, the formula (\ref{hayirsh}) for $\mathfrak{I}_a$ can be written as
\begin{equation}\label{evetsh}
\mathfrak{I}_a(\mu)=\int_{W_\infty^{\mathrm{tr}}}\left[\sum_{z\in U}\hat{q}(w,z)\log\frac{\hat{q}(w,z)}{q(w,z)}\right]\,\mathrm{d}\mu(w).
\end{equation} where $\hat{q}(\cdot,z)=q_\mu(\cdot,z)$ on the support of $\mu$. Therefore, $\mathfrak{I}_a$ is affine linear on $\mathcal{I}$.

\begin{lemma}
$$I_a(\xi)=\inf_{\substack{\mu\in\mathcal{I}\\m(\mu)=\xi}}\mathfrak{I}_a(\mu).$$
\end{lemma}
\begin{proof}
By the definition of $I_a$ in (\ref{divaneasiksh}), $$I_a(\xi)\geq\inf_{\substack{\mu\in\mathcal{I}\\m(\mu)=\xi}}\mathfrak{I}_a(\mu)$$ is clear. To establish the reverse inequality, take any $\mu\in\mathcal{I}$ with $m(\mu)=\xi$. Since $\mathcal{E}$ is the set of extremal points of $\mathcal{I}$, $\mu$ can be expressed as $$\mu=\int_{\mathcal{E}_o}\alpha\,\mathrm{d}\hat{\mu}(\alpha) + \int_{\mathcal{E}\backslash\mathcal{E}_o}\alpha\,\mathrm{d}\hat{\mu}(\alpha)=\int_{\mathcal{E}_o}\alpha\,\mathrm{d}\hat{\mu}(\alpha) + (1-\hat{\mu}(\mathcal{E}_o))\tilde{\mu}$$ where $\mathcal{E}_o:=\{\alpha\in\mathcal{E}:m(\alpha)\neq0\}$, $\hat{\mu}$ is some measure on $\mathcal{E}$, and $\tilde{\mu}\in\mathcal{I}$ with $m(\tilde{\mu})=0$. Then,
\begin{align}
\mathfrak{I}_a(\mu)&=\int_{\mathcal{E}_o}\mathfrak{I}_a(\alpha)\,\mathrm{d}\hat{\mu}(\alpha) + (1-\hat{\mu}(\mathcal{E}_o))\mathfrak{I}(\tilde{\mu})\label{miyash}\\
&\geq\int_{\mathcal{E}_o}I_a(m(\alpha))\,\mathrm{d}\hat{\mu}(\alpha) + (1-\hat{\mu}(\mathcal{E}_o))I(0)\label{miyaash}\\
&\geq I_a(\xi).\label{miyaaash}
\end{align} The equality in (\ref{miyash}) uses the affine linearity of $\mathfrak{I}$. (\ref{miyaash}) follows from two facts: (i) $\mathfrak{I}_a(\alpha)\geq I_a(m(\alpha))$, (ii) $\mathfrak{I}_a(\tilde{\mu})\geq I_a(0)$. The first fact is immediate from the definition of $I_a$. See Lemma 7.2 of \cite{Raghu} for the proof of the second fact. Finally, the convexity of $I_a$ gives (\ref{miyaaash}).
\end{proof}

\begin{lemma}\label{tugish}
If $I_a(\cdot)$ is strictly convex at $\xi$, then the infimum in (\ref{divaneasiksh}) is attained.
\end{lemma}

\begin{proof}
Let $W_n := \{\left(x_i\right)_{-n\leq i\leq0}:x_{i+1}-x_i\in U,\,x_o=0\}$. The simplest compactification of $W:=\cup_{n}W_n$ is $W_\infty:=\{\left(x_i\right)_{i\leq0}:x_{i+1}-x_i\in U,\,x_o=0\}$ with the product topology. However, the functions $q(\cdot,z)$ (recall (\ref{ozgurevren})) are only defined on $W_{\infty}^{\mathrm{tr}}$, and even when restricted to it they are not continuous since two walks that are identical in the immediate past are close to each other in this topology even if one of them visits $0$ in the remote past and the other one doesn't.

Section 5 of \cite{Raghu} introduces a more convenient compactification $\overline{W}$ of $W$. It is a ramification of $W_\infty$, and the functions $q(\cdot,z)$ continuously extend from $W$ to $\overline{W}$. Denote the $T^*$-invariant probability measures on $\overline{W}$ by $\overline{\mathcal{I}}$, and the extremals of $\overline{\mathcal{I}}$ by $\overline{\mathcal{E}}$. Recall that $\mathcal{E}_o:=\{\alpha\in\mathcal{E}:m(\alpha)\neq0\}$. Then, $\mathcal{E}_o\subset\mathcal{E}\subset\overline{\mathcal{E}}$ and $\mathcal{I}\subset\overline{\mathcal{I}}$. Note that the domain of the formula for $\mathfrak{I}$ given in (\ref{evetsh}) extends to $\overline{\mathcal{I}}$.

Take $\mu_n\in\mathcal{E}$ such that $m(\mu_n)=\xi$ and $\mathfrak{I}_a(\mu_n)\to I_a(\xi)$ as $n\to\infty$. Let $\overline{\mu}\in\overline{\mathcal{I}}$ be a weak limit point of $\mu_n$. Corollary 6.2 of \cite{Raghu} shows that $\overline{\mu}$ has a representation $$\overline{\mu}=\int_{\mathcal{E}_o}\alpha\,\mathrm{d}\hat{\mu}_1(\alpha) + (1-\hat{\mu}_1(\mathcal{E}_o))\overline{\mu}_2$$ where $\hat{\mu}_1$ is some measure on $\mathcal{E}_o$, and $\overline{\mu}_2\in\overline{\mathcal{I}}$ with $m(\overline{\mu}_2)=0$. Then,
\begin{align}
I_a(\xi)=\lim_{n\to\infty}\mathfrak{I}_a(\mu_n)&\geq\mathfrak{I}_a(\overline{\mu})\label{viysh}\\
&=\int_{\mathcal{E}_o}\mathfrak{I}_a(\alpha)\,\mathrm{d}\hat{\mu}_1(\alpha) + (1-\hat{\mu}_1(\mathcal{E}_o))\mathfrak{I}_a(\overline{\mu}_2)\label{viyash}\\
&\geq\int_{\mathcal{E}_o}I_a(m(\alpha))\,\mathrm{d}\hat{\mu}_1(\alpha) + (1-\hat{\mu}_1(\mathcal{E}_o))I_a(0)\label{viyaash}\\
&\geq I_a(\xi).\label{viyaaash}
\end{align} The inequality in (\ref{viysh}) follows from the lower semicontinuity of $\mathfrak{I}_a$, and the equality in (\ref{viyash}) is a consequence of the affine linearity of $\mathfrak{I}_a$. (\ref{viyaash}) relies on the fact that $\mathfrak{I}_a(\overline{\mu}_2)\geq I_a(0)$. See Lemma 7.2 of \cite{Raghu} for the proof. Finally, the convexity of $I_a$ gives (\ref{viyaaash}). Since $I_a(\cdot)$ is assumed to be strictly convex at $\xi$, $\hat{\mu}_1\left(\alpha\in\mathcal{E}_o: m(\alpha)=\xi,\, \mathfrak{I}_a(\alpha)=I_a(\xi)\right)=1$. Hence, we are done.
\end{proof}

\subsection{Formula for the unique minimizer}

Fix any $\xi\in\mathcal{A}$. Recall the definitions of $\bar{\mu}_\xi^\infty$ and $\mu_\xi^\infty$ given in Section \ref{franzek}.

\begin{proposition}\label{bursayok}
$\bar{\mu}_\xi^\infty$ is well defined.
\end{proposition}
\begin{proof}
For every $K\in\mathbb{N}$, take any bounded function $f:U^{\mathbb{N}}\rightarrow\mathbb{R}$ such that $f((z_i)_{i\geq1})$ is independent of $(z_i)_{i>K}$. Then, $f((z_i)_{i\geq1})$ is independent of $(z_i)_{i>K'}$ for every $K'>K$ as well. So, we need to show that (\ref{muyucananan}) does not change if we replace $K$ by $K+1$. But, this is clear because
\begin{align*}
&E_o\left[\left.\sum_{j=0}^{\tau_1 -1}f((Z_{j+i})_{i\geq1})\ \mathrm{e}^{\langle\theta,X_{\tau_{K+1}}\rangle - \Lambda_a(\theta)\tau_{K+1}}\,\right|\,D=\infty\right]\\
=&E_o\left[\left.\sum_{j=0}^{\tau_1 -1}f((Z_{j+i})_{i\geq1})\ \mathrm{e}^{\langle\theta,X_{\tau_{K}}\rangle - \Lambda_a(\theta)\tau_{K}}\left\{\mathrm{e}^{\langle\theta,X_{\tau_{K+1}}-X_{\tau_K}\rangle - \Lambda_a(\theta)(\tau_{K+1}-\tau_K)}\right\}\,\right|\,D=\infty\right]\\
=&E_o\left[\left.\sum_{j=0}^{\tau_1 -1}f((Z_{j+i})_{i\geq1})\ \mathrm{e}^{\langle\theta,X_{\tau_{K}}\rangle - \Lambda_a(\theta)\tau_{K}}\,\right|\,D=\infty\right].
\end{align*}Explanation: In the second line of the above display, the term in $\{\cdot\}$ is independent of the others. The expectation therefore splits, and Lemmas \ref{memenguzel}\,\&\,\ref{gotunguzel} give
\begin{align*}
E_o\left[\left.\mathrm{e}^{\langle\theta,X_{\tau_{K+1}}-X_{\tau_K}\rangle - \Lambda_a(\theta)(\tau_{K+1}-\tau_K)}\,\right|\,D=\infty\right]&=E_o\left[\mathrm{e}^{\langle\theta,X_{\tau_{K+1}}-X_{\tau_K}\rangle - \Lambda_a(\theta)(\tau_{K+1}-\tau_K)}\right]\\&=E_o\left[\left.\mathrm{e}^{\langle\theta,X_{\tau_1}\rangle - \Lambda_a(\theta)\tau_1}\,\right|\,D=\infty\right]\\&=1.\qedhere
\end{align*}
\end{proof}

The following theorem states that the empirical process $$\bar{\nu}_{n,X}^\infty := \frac{1}{n}\sum_{j=0}^{n-1}\one_{\left(Z_{j+i}\right)_{i\geq1}}$$ of the walk under $P_o$ converges to $\bar{\mu}_\xi^\infty$ when the particle is conditioned to have mean velocity $\xi$. (Here, $Z_i = X_i-X_{i-1}$ as usual.)

\begin{theorem}\label{averagedconditioningsh}
For every $K\in\mathbb{N}$, $f:U^{\mathbb{N}}\rightarrow\mathbb{R}$ such that $f((z_i)_{i\geq1})$ is independent of $(z_i)_{i>K}$ and bounded, and $\epsilon>0$,
\[\limsup_{\delta\to0}\limsup_{n\rightarrow\infty}\frac{1}{n}\log P_o\left(\ \left|\int\!\! f\mathrm{d}\bar{\nu}_{n,X}^\infty-\int\!\! f\mathrm{d}\bar{\mu}_\xi^\infty\right|>\epsilon\ \left|\ |\frac{X_n}{n}-\xi|\leq\delta\right.\right)<0.\]
\end{theorem}
\begin{remark}
This result generalizes Theorem \ref{averagedconditioning} to the RWRE setting. The only difference is that the measures involved here are just on particle paths. However, one can easily modify the argument to deal with measures that live both on paths and environments.
\end{remark}
\begin{proof}[Proof in the non-nestling case]
Since $\xi\in\mathcal{A}$, there exists a unique $\theta\in\mathbb{R}^d$ with $|\theta|<c_3$ such that $\xi=\nabla\Lambda_a(\theta)$. Let $g(\cdot):=f(\cdot)-\int\!\! f\mathrm{d}\bar{\mu}_\xi^\infty$. Assume WLOG that $|g|\leq 1$. Then, $\int\!\! f\mathrm{d}\bar{\nu}_{n,X}^\infty-\int\!\! f\mathrm{d}\bar{\mu}_\xi^\infty = \int\!\! g\,\mathrm{d}\bar{\nu}_{n,X}^\infty =: \langle g,\bar{\nu}_{n,X}^\infty\rangle$. For any $s\in\mathbb{R}$,
\begin{align}
&E_o\left[\mathrm{e}^{\langle\theta,X_n\rangle-\Lambda_a(\theta)n+ns\langle g,\bar{\nu}_{n,X}^\infty\rangle}\right]\nonumber\\
&\quad\quad\quad\quad\quad\quad=E_o\left[n<\tau_{K+1},\,\mathrm{e}^{\langle\theta,X_n\rangle-\Lambda_a(\theta)n+ns\langle g,\bar{\nu}_{n,X}^\infty\rangle}\right]\nonumber\\&\quad\quad\quad\quad\quad\quad\quad+\sum_{m=K+1}^nE_o\left[\tau_m\leq n<\tau_{m+1},\,\mathrm{e}^{\langle\theta,X_n\rangle-\Lambda_a(\theta)n+ns\langle g,\bar{\nu}_{n,X}^\infty\rangle}\right]\label{cokazkaldish}.
\end{align}
If $2|\theta|+|s|<2c_3$, then the first term in (\ref{cokazkaldish}) is bounded from above by $E_o\!\left[n<\tau_{K+1},\,\mathrm{e}^{(2|\theta|+|s|)\tau_{K+1}}\right]$ which goes to $0$ as $n\to\infty$ by Lemma \ref{Szestimates} and the monotone convergence theorem. For $j\geq0$, define
\begin{equation}\label{sabrialtintas}
G_j:=\sum_{k=\tau_j}^{\tau_{j+1}-1}g((Z_{k+i})_{i\geq1})
\end{equation} with the convention that $\tau_o=0$. Note that $G_j$ is a function of $Z_{\tau_j+1},\ldots,Z_{\tau_{j+1}+K-1}$. Therefore, $G_j$ and $G_{j+K}$ depend on disjoint sets of steps since $\tau_{j+1}+K-1\leq\tau_{j+K}$. For any $p,q\in\mathbb{R}$ with $1<p<c_3/|\theta|$ and $1/p+1/q=1$, each term of the sum in (\ref{cokazkaldish}) can be bounded using H\"older's inequality:
\begin{align}
&E_o\left[\tau_m\leq n<\tau_{m+1},\,\mathrm{e}^{\langle\theta,X_n\rangle-\Lambda_a(\theta)n+ns\langle g,\bar{\nu}_{n,X}^\infty\rangle}\right]\nonumber\\
&\leq E_o\left[\mathrm{e}^{\langle\theta,X_{\tau_m}-X_{\tau_1}\rangle-\Lambda_a(\theta)(\tau_m-\tau_1)+s\left(G_1+\cdots+G_{m-K-1}\right)+\left(2|\theta|+|s|\right)\left(\tau_1+\tau_{m+1}-\tau_m\right)+|s|\left(\tau_m-\tau_{m-K}\right)}\right]\nonumber\displaybreak\\
&\leq E_o\left[\mathrm{e}^{\left(2|\theta|+|s|\right)\tau_1}\right]E_o\left[\mathrm{e}^{\langle\theta,X_{\tau_m}-X_{\tau_1}\rangle-\Lambda_a(\theta)(\tau_m-\tau_1)+p\left(2|\theta|+|s|\right)\left(\tau_{m+1}-\tau_m\right)+p|s|\left(\tau_m-\tau_{m-K}\right)}\right]^{1/p}\nonumber\\
&\quad\times\prod_{i=1}^K E_o\left[\mathrm{e}^{\langle\theta,X_{\tau_m}-X_{\tau_1}\rangle-\Lambda_a(\theta)(\tau_m-\tau_1)+(Kq)s\left(G_i+G_{i+K}+\cdots+G_{i+[\frac{m-K-i-1}{K}]K}\right)}\right]^{1/(Kq)}\nonumber\\
&\leq E_o\left[\mathrm{e}^{\left(2|\theta|+|s|\right)\tau_1}\right]E_o\left[\left.\mathrm{e}^{p\left(2|\theta|+|s|\right)\tau_1}\,\right|\,D=\infty\right]^{\frac{K+1}{p}}\nonumber\\
&\quad\times E_o\left[\left.\mathrm{e}^{\langle\theta,X_{\tau_K}\rangle-\Lambda_a(\theta)\tau_K+(Kq)sG_o}\,\right|\,D=\infty\right]^{\frac{m-K-1}{Kq}}\label{cumartesish}.
\end{align}
The last inequality follows from the fact that $(G_i,G_{i+K},\ldots)$ is an i.i.d.\ sequence. The terms of the product in (\ref{cumartesish}) are finite by Lemma \ref{Szestimates} if $p(2|\theta|+|s|)<2c_3$ and $2|\theta|+(Kq)|s|<2c_3$.  Putting the pieces together,
\begin{align*}
&\limsup_{n\to\infty}\frac{1}{n}\log E_o\left[\mathrm{e}^{\langle\theta,X_n\rangle-\Lambda_a(\theta)n+ns\langle g,\bar{\nu}_{n,X}^\infty\rangle}\right]\\
&\leq0\vee\limsup_{n\to\infty}\frac{1}{n}\log \sum_{m=K+1}^n E_o\left[\left.\mathrm{e}^{\langle\theta,X_{\tau_K}\rangle-\Lambda_a(\theta)\tau_K+(Kq)sG_o}\,\right|\,D=\infty\right]^{\frac{m-K-1}{Kq}}\\
&\leq0\vee\frac{1}{Kq}\log E_o\left[\left.\mathrm{e}^{\langle\theta,X_{\tau_K}\rangle-\Lambda_a(\theta)\tau_K+(Kq)sG_o}\,\right|\,D=\infty\right].
\end{align*}
Let $h(s):=\frac{1}{Kq}\log E_o\left[\left.\mathrm{e}^{\langle\theta,X_{\tau_K}\rangle-\Lambda_a(\theta)\tau_K+(Kq)sG_o}\,\right|\,D=\infty\right]$. Lemma \ref{memenguzel} implies that $h(0)=0$. The map $s\mapsto h(s)$ is analytic in a neighborhood of $0$, and
\begin{align*}
h'(0)&=E_o\left[\left.G_o\,\mathrm{e}^{\langle\theta,X_{\tau_K}\rangle-\Lambda_a(\theta)\tau_K}\,\right|\,D=\infty\right]\\
&=E_o\left[\left.\sum_{k=0}^{\tau_{1}-1}g((Z_{k+i})_{i\geq1})\,\mathrm{e}^{\langle\theta,X_{\tau_K}\rangle-\Lambda_a(\theta)\tau_K}\,\right|\,D=\infty\right]\\
&=E_o\left[\left.\left(\sum_{k=0}^{\tau_{1}-1}f((Z_{k+i})_{i\geq1})-\tau_1\int\!\! f\mathrm{d}\bar{\mu}_\xi^\infty\right)\mathrm{e}^{\langle\theta,X_{\tau_K}\rangle-\Lambda_a(\theta)\tau_K}\,\right|\,D=\infty\right]=0
\end{align*}
by Definition \ref{definemuyuanan}. We conclude that
\begin{equation}\label{yilish}
\limsup_{n\to\infty}\frac{1}{n}\log E_o\left[\mathrm{e}^{\langle\theta,X_n\rangle-\Lambda_a(\theta)n+ns\langle g,\bar{\nu}_{n,X}^\infty\rangle}\right]\leq o(s).
\end{equation}

Whenever $s>0$ is small enough, a standard change of measure argument and the averaged LDP give
\begin{align*}
&\limsup_{n\rightarrow\infty}\frac{1}{n}\log P_o\left(\ \int\!\! f\mathrm{d}\bar{\nu}_{n,X}^\infty-\int\!\! f\mathrm{d}\bar{\mu}_\xi^\infty>\epsilon\ \left|\ |\frac{X_n}{n}-\xi|\leq\delta\right.\right)\\=&\limsup_{n\rightarrow\infty}\frac{1}{n}\log P_o\left(\langle g,\bar{\nu}_{n,X}^\infty\rangle>\epsilon,|\frac{X_n}{n}-\xi|\leq\delta\right)-\lim_{n\rightarrow\infty}\frac{1}{n}\log P_o\left(|\frac{X_n}{n}-\xi|\leq\delta\right)\\\leq&\limsup_{n\rightarrow\infty}\frac{1}{n}\log E_o\left[\mathrm{e}^{\langle\theta,X_n\rangle},\langle g,\bar{\nu}_{n,X}^\infty\rangle>\epsilon,|\frac{X_n}{n}-\xi|\leq\delta\right]-\langle\theta,\xi\rangle + I_a(\xi) + |\theta|\delta\\\leq&\limsup_{n\rightarrow\infty}\frac{1}{n}\log E_o\left[\mathrm{e}^{\langle\theta,X_n\rangle-\Lambda_a(\theta)n},\langle g,\bar{\nu}_{n,X}^\infty\rangle>\epsilon\right] + |\theta|\delta\\\leq&\limsup_{n\to\infty}\frac{1}{n}\log E_o\left[\mathrm{e}^{\langle\theta,X_n\rangle-\Lambda_a(\theta)n+ns\langle g,\bar{\nu}_{n,X}^\infty\rangle}\right] - s\epsilon + |\theta|\delta\\
\leq&\,o(s) - s\epsilon + |\theta|\delta\\\leq&-s\epsilon/2+ |\theta|\delta\end{align*} for every $\delta>0$. Similarly, $$\limsup_{n\rightarrow\infty}\frac{1}{n}\log P_o\left(\ \int\!\! f\mathrm{d}\bar{\nu}_{n,X}^\infty-\int\!\! f\mathrm{d}\bar{\mu}_\xi^\infty<-\epsilon\ \left|\ |\frac{X_n}{n}-\xi|\leq\delta\right.\right)\leq-s\epsilon/2+ |\theta|\delta.$$By combining these two bounds, we finally deduce that \[\limsup_{\delta\to0}\limsup_{n\rightarrow\infty}\frac{1}{n}\log P_o\left(\ \left|\int\!\! f\mathrm{d}\bar{\nu}_{n,X}^\infty-\int\!\! f\mathrm{d}\bar{\mu}_\xi^\infty\right|>\epsilon\ \left|\ |\frac{X_n}{n}-\xi|\leq\delta\right.\right)\leq-s\epsilon/2.\]
\end{proof}

\begin{proof}[Proof in the nestling case]
Since $\xi\in\mathcal{A}$, there exists a unique $\theta\in\mathbb{R}^d$ with $|\theta|<c_4$ such that $\Lambda_a(\theta)>0$ and $\xi=\nabla\Lambda_a(\theta)$. If $0<s<\Lambda_a(\theta)$, then the first term in (\ref{cokazkaldish}) is bounded by $E_o\left[n<\tau_{K+1},\,\mathrm{e}^{|\theta||X_n|}\right]$ which goes to $0$ as $n\to\infty$ by Lemma \ref{Szestimates} and the monotone convergence theorem.

For any $p,q$ with $1<p<c_4/|\theta|$ and $1/p+1/q=1$, each term of the sum in (\ref{cokazkaldish}) can be bounded using H\"older's inequality when $0<s<{\Lambda_a(\theta)}/{(p\vee Kq)}$:
\begin{align}
&E_o\left[\tau_m\leq n<\tau_{m+1},\,\mathrm{e}^{\langle\theta,X_n\rangle-\Lambda_a(\theta)n+ns\langle g,\bar{\nu}_{n,X}^\infty\rangle}\right]\nonumber\\
&\leq E_o\left[\mathrm{e}^{\langle\theta,X_{\tau_1}\rangle+\langle\theta,X_{\tau_m}-X_{\tau_1}\rangle-\Lambda_a(\theta)(\tau_m-\tau_1)+s\left(G_1+\cdots+G_{m-1}\right)}\!\!\!\!\!\sup_{\tau_m\leq n<\tau_{m+1}}\!\!\!\!\!\mathrm{e}^{|\theta||X_n-X_{\tau_m}|}\right]\nonumber\\
&\leq E_o\left[\mathrm{e}^{\langle\theta,X_{\tau_1}\rangle}\right]E_o\left[\mathrm{e}^{\langle\theta,X_{\tau_m}-X_{\tau_1}\rangle-\Lambda_a(\theta)(\tau_m-\tau_1)+ps\left(G_{m-K}+\cdots+G_{m-1}\right)}\!\!\!\!\!\!\!\sup_{\tau_m\leq n<\tau_{m+1}}\!\!\!\!\!\!\!\mathrm{e}^{p|\theta||X_n-X_{\tau_m}|}\right]^{1/p}\nonumber\\
&\quad\times\prod_{i=1}^K E_o\left[\mathrm{e}^{\langle\theta,X_{\tau_m}-X_{\tau_1}\rangle-\Lambda_a(\theta)(\tau_m-\tau_1)+(Kq)s\left(G_i+G_{i+K}+\cdots+G_{i+[\frac{m-K-i-1}{K}]K}\right)}\right]^{1/(Kq)}\nonumber\\
&\leq E_o\left[\mathrm{e}^{\langle\theta,X_{\tau_1}\rangle}\right]E_o\left[\left.\sup_{\tau_{K}\leq n<\tau_{K+1}}\!\!\!\!\mathrm{e}^{p|\theta||X_n|}\,\right|\,D=\infty\right]^{1/p}\nonumber\\&\quad\times E_o\left[\left.\mathrm{e}^{\langle\theta,X_{\tau_K}\rangle-\Lambda_a(\theta)\tau_K+(Kq)sG_o}\,\right|\,D=\infty\right]^{\frac{m-K-1}{Kq}}.\label{sokush}
\end{align}

The first two terms in (\ref{sokush}) are finite by Lemma \ref{Szestimates}. The last term in (\ref{sokush}) is equal to the last term in (\ref{cumartesish}). The rest of the argument is identical to the one given in the non-nestling case.
\end{proof}

\begin{proof}[Proof of Theorem \ref{sevval}]
Fix $\xi\in\mathcal{A}$. Take any $\alpha\in\mathcal{E}$ with $m(\alpha)=\xi$. Denote by $\bar{\alpha}$ the probability measure $\alpha$ induces on $U^{\mathbb{N}}$ via the map $(x_1, x_2,\ldots)\mapsto(x_1,x_2-x_1,\ldots)$. If $\alpha\neq\mu_\xi^\infty$, then there exist $K\in\mathbb{N}$, $f:U^{\mathbb{N}}\rightarrow\mathbb{R}$ and $\epsilon>0$ such that $f((z_i)_{i\geq1})$ is bounded and independent of $(z_i)_{i>K}$, and $|\langle f,\bar{\alpha} - \bar{\mu}_\xi^\infty\rangle|>\epsilon$. Let $H(n,X)$ denote the number of times $(X_1,\ldots,X_n)$ intersects $(X_i)_{i\leq0}$. Since the walk under $Q_\alpha$ is transient in the $\xi$ direction, there exists a large constant $L$ such that $\lim_{n\to\infty}Q_\alpha(H(n,X)\leq L)\geq1/2$. For notational convenience, let $$A_n^\delta:=\left\{|\langle f,\bar{\nu}_{n,X}^\infty - \bar{\mu}_\xi^\infty\rangle|>\epsilon,\,|\frac{X_n}{n}-\xi|\leq\delta,\,H(n+K,X)\leq L\right\}.$$ Recall Assumption (A2) of Section \ref{franzek}, and use Jensen's inequality to write
\begin{align*}
&P_o\left(|\langle f,\bar{\nu}_{n,X}^\infty - \bar{\mu}_\xi^\infty\rangle|>\epsilon,\,|\frac{X_n}{n}-\xi|\leq\delta\right)\\
%&\quad\quad\quad\geq\sup_{w\in W_\infty^{\mathrm{tr}}}P_o\left(A_n^\delta(f,\xi,\epsilon,w,X)\right)\\
&\geq(c_1)^L\sup_{w\in W_\infty^{\mathrm{tr}}}Q^w\left(A_n^\delta\right)\\
&\geq(c_1)^L\int E^w\left[\one_{A_n^\delta}\right]\,\mathrm{d}\alpha(w)\\
&=(c_1)^L\int E_\alpha^w\left[\one_{A_n^\delta}\,\left.\frac{\mathrm{d}Q^w}{\mathrm{d}Q_\alpha^w}\right|_{\sigma(Z_1,\ldots,Z_n)}\right]\,\mathrm{d}\alpha(w)\\
%&\quad\quad\quad=(c_1)^L\int_{A_n^\delta}\frac{\mathrm{d}Q^w}{\mathrm{d}Q_\alpha^w}(z_1,\ldots,z_n)\mathrm{d}\alpha(w,z)\\
&=(c_1)^LQ_\alpha(A_n^\delta)\frac{1}{Q_\alpha(A_n^\delta)}\int_{A_n^\delta}\exp\left(-\log\frac{\mathrm{d}Q_\alpha^w}{\mathrm{d}Q^w}(z_1,\ldots,z_n)\right)\mathrm{d}Q_\alpha(w,z_1,\ldots,z_n)\\
&\geq(c_1)^LQ_\alpha(A_n^\delta)\exp\left(-\frac{1}{Q_\alpha(A_n^\delta)}\int_{A_n^\delta}\!\!\!\log\frac{\mathrm{d}Q_\alpha^w}{\mathrm{d}Q^w}(z_1,\ldots,z_n)\mathrm{d}Q_\alpha(w,z_1,\ldots,z_n)\right).
\end{align*}Since $m(\alpha)=\xi$ and $|\langle f,\bar{\alpha} - \bar{\mu}_\xi^\infty\rangle|>\epsilon$, the $L^1$-ergodic theorem implies that $\lim_{n\to\infty}Q_\alpha(A_n^\delta)=\lim_{n\to\infty}Q_\alpha(H(n,X)\leq L)\geq1/2$. Therefore,
\begin{align*}
&\liminf_{n\to\infty}\frac{1}{n}\log P_o\left(|\langle f,\bar{\nu}_{n,X}^\infty - \bar{\mu}_{n,X}^\infty\rangle|>\epsilon,\,|\frac{X_n}{n}-\xi|\leq\delta\right)\\
&\quad\geq-\limsup_{n\to\infty}\frac{1}{Q_\alpha(A_n^\delta)}\int_{A_n^\delta}\log\frac{\mathrm{d}Q_\alpha^w}{\mathrm{d}Q^w}(z_1,\ldots,z_n)\mathrm{d}Q_\alpha(w,z_1,\ldots,z_n)\\
&\quad=-\int_{W_\infty^{\mathrm{tr}}}\left[\sum_{z\in U}q_\alpha(w,z)\log\frac{q_\alpha(w,z)}{q(w,z)}\right]\mathrm{d}\alpha(w)=-\mathfrak{I}_a(\alpha)
\end{align*} again by the $L^1$-ergodic theorem. Finally, Theorem \ref{averagedconditioningsh} and the averaged LDP give
\begin{align*}
0>&\limsup_{\delta\to0}\limsup_{n\rightarrow\infty}\frac{1}{n}\log P_o\left(\ \left|\int\!\! f\mathrm{d}\bar{\nu}_{n,X}^\infty-\int\!\! f\mathrm{d}\bar{\mu}_\xi^\infty\right|>\epsilon\ \left|\ |\frac{X_n}{n}-\xi|\leq\delta\right.\right)\\
=&\,I_a(\xi)+\limsup_{\delta\to0}\limsup_{n\rightarrow\infty}\frac{1}{n}\log P_o\left(\ \left|\int\!\! f\mathrm{d}\bar{\nu}_{n,X}^\infty-\int\!\! f\mathrm{d}\bar{\mu}_\xi^\infty\right|>\epsilon,\,|\frac{X_n}{n}-\xi|\leq\delta\right)\\
\geq&\,I_a(\xi)-\mathfrak{I}_a(\alpha).
\end{align*} In words, $\alpha$ is not the minimizer of (\ref{divaneasik}). Theorem \ref{qual} and Lemma \ref{tugish} imply that the infimum in (\ref{divaneasik}) is attained. Hence, $\mu_\xi^\infty$ is the unique minimizer of (\ref{divaneasik}).
\end{proof}
\begin{remark}
The above argument indirectly proves that $\mu_\xi^\infty\in\mathcal{E}$ and $m(\mu_\xi^\infty)=\xi$. These facts are also easy to show directly using Definition \ref{definemuyuanan}. In fact, $\mu_\xi^\infty$ is mixing with rate given by the tail behaviour of $\tau_1$.
\end{remark}

\chapter{Large deviations for space-time RWRE}\label{spacetimechapter}

Section \ref{nilgunisik} introduces the notation for space-time RWRE, and states our results.

\section{Equivalence of quenched and averaged LDPs}\label{Qsection}

%In this section, we obtain the function $u^\theta$ mentioned in the statement of Theorem \ref{doob}, derive some of its properties and define the transformed kernel $\overline{\pi}^\theta$ also mentioned in Theorem \ref{doob}. Finally, having built the necessary machinery, we prove Theorem \ref{AequalsQ} and Theorem \ref{strong}.

\subsection{An $L^2$ estimate}

In this chapter, the following family of functions play a central role:

\begin{definition}
For every $\omega\in\Omega$, $\theta\in\mathbb{R}^d$, $x\in\mathbb{Z}^d$ and $n,N\in\mathbb{Z}$ with $n<N$, define
\begin{equation}\label{babalar}
u_N^\theta(\omega,n,x):=E_{n,x}^\omega\left[\mathrm{e}^{\langle\theta,X_N-X_n\rangle-\Lambda_c(\theta)(N-n)}\right]
\end{equation} where $\Lambda_c$ is given in (\ref{fi}).
\end{definition}

The main estimate that enables us to establish the equivalence of quenched and averaged large deviations is
\begin{lemma}\label{eltu}
If $d\geq3$, then there exists $\bar{\eta}>0$ such that for every $\theta\in\mathbb{R}^d$ with $|\theta|<\bar{\eta}$, $x\in\mathbb{Z}^d$ and $n\in\mathbb{Z}$, $$\sup_{N>n}\left\|u_N^\theta(\cdot,n,x)\right\|_{L^2(\mathbb{P})}<\infty.$$
\end{lemma}

\begin{proof}
It suffices to prove the lemma for $n=0$ and $x=0$.
\begin{align}
&G_N(\theta):=\left\|u_N^\theta(\cdot,0,0)\right\|_{L^2(\mathbb{P})}^2=\mathbb{E}\left(E_{o,o}^\omega\left[\mathrm{e}^{\langle\theta,X_N\rangle-\Lambda_c(\theta)N}\right]^2\right)\label{aboo}\\
%=&\mathbb{E}\left(\sum_{x_1,\ldots,x_n}\prod_{i=0}^{N-1}\pi_{i,i+1}(x_i,x_{i+1})\frac{\mathrm{e}^{\langle\theta,x_{i+1}-x_i\rangle}}{\Lambda_c(\theta)}\sum_{y_1,\ldots,y_n}\prod_{i=0}^{N-1}\pi_{i,i+1}(y_i,y_{i+1})\frac{\mathrm{e}^{\langle\theta,y_{i+1}-y_i\rangle}}{\Lambda_c(\theta)}\right)\\
=&\sum_{x_o=0,x_1,\ldots,x_N\atop y_o=0,y_1,\ldots,y_N} \prod_{i=0}^{N-1}\mathbb{E}\left(\pi_{i,i+1}(x_i,x_{i+1})\pi_{i,i+1}(y_i,y_{i+1})\right)\frac{\mathrm{e}^{\langle\theta,x_{i+1}-x_i\rangle}}{\mathrm{e}^{\Lambda_c(\theta)}}\frac{\mathrm{e}^{\langle\theta,y_{i+1}-y_i\rangle}}{\mathrm{e}^{\Lambda_c(\theta)}}\nonumber\\
=&\sum_{x_o=0,x_1,\ldots,x_N\atop y_o=0,y_1,\ldots,y_N}\prod_{i=0}^{N-1}\frac{\mathbb{E}\left(\pi_{i,i+1}(x_i,x_{i+1})\pi_{i,i+1}(y_i,y_{i+1})\right)}{q(x_{i+1}-x_i)q(y_{i+1}-y_i)}q^\theta(x_{i+1}-x_i)q^\theta(y_{i+1}-y_i)\nonumber
\end{align}
where $q^\theta(z):=q(z)\mathrm{e}^{\langle\theta,z\rangle-\Lambda_c(\theta)}$ for every $z\in U$. For every $x\in\mathbb{Z}^d$, let $\hat{P}_x^\theta$ be the probability measure on paths starting at $x$ and induced by $\left(q^\theta(z)\right)_{z\in U}$. Write $\hat{E}_x^\theta$ to denote expectation with respect to $\hat{P}_x^\theta$.

Note that $\mathbb{E}\left(\pi_{i,i+1}(x_i,x_{i+1})\pi_{i,i+1}(y_i,y_{i+1})\right)=q(x_{i+1}-x_i)q(y_{i+1}-y_i)$ unless $x_i=y_i$. For every $x,y\in U$, set \[V(x,y):=\log\left(\frac{\mathbb{E}\left(\pi_{0,1}(0,x)\pi_{0,1}(0,y)\right)}{q(x)q(y)}\right).\] By uniform ellipticity, $V$ is bounded by some constant $\bar{V}$. With this notation,\[G_N(\theta)=\hat{E}_o^\theta\!\times\!\hat{E}_o^\theta\left[\mathrm{e}^{\sum_{i=0}^{N-1}\one_{X_i=Y_i}V(X_{i+1}-X_i,Y_{i+1}-Y_i)}\right].\] Let $s:=\inf\left\{k\geq0:\ X_k=Y_k\right\}$, $s^+:=\inf\left\{k>0:\ X_k=Y_k\right\}$ and decompose $G_N(\theta)$ with respect to the first steps $X_1$ and $Y_1$:\pagebreak
\begin{align*}
G_N(\theta)=&\sum_{x,y}q^\theta(x)q^\theta(y)\mathrm{e}^{V(x,y)}\sum_{k=0}^{N-2}\hat{P}_x^\theta\!\times\!\hat{P}_y^\theta\left(s=k\right)G_{N-k-1}(\theta)\\&+\sum_{x,y}q^\theta(x)q^\theta(y)\mathrm{e}^{V(x,y)}\hat{P}_x^\theta\!\times\!\hat{P}_y^\theta\left(s\geq N-1\right)\\
=&\sum_{k=0}^{N-2}\left(\sum_{x,y}q^\theta(x)q^\theta(y)\mathrm{e}^{V(x,y)}\hat{P}_x^\theta\!\times\!\hat{P}_y^\theta\left(s=k\right)\right)G_{N-k-1}(\theta)\\&+\sum_{x,y}q^\theta(x)q^\theta(y)\mathrm{e}^{V(x,y)}\hat{P}_x^\theta\!\times\!\hat{P}_y^\theta\left(s\geq N-1\right).
\end{align*}
Simplify the last expression by defining \begin{align*}
B_k(\theta)&:=\sum_{x,y}q^\theta(x)q^\theta(y)\mathrm{e}^{V(x,y)}\hat{P}_x^\theta\!\times\!\hat{P}_y^\theta\left(s=k\right),\\ C_N(\theta)&:=\sum_{x,y}q^\theta(x)q^\theta(y)\mathrm{e}^{V(x,y)}\hat{P}_x^\theta\!\times\!\hat{P}_y^\theta\left(s\geq N-1\right)
\end{align*} and obtain the following equation:
\begin{equation}\label{neis}
G_N(\theta)=\sum_{k=0}^{N-2}B_k(\theta)G_{N-k-1}(\theta)+C_N(\theta).
\end{equation} 
Since $d\geq3$, $(X_i-Y_i)_{i\geq0}$ is a transient random walk under the product measure $\hat{P}_x^\theta\!\times\!\hat{P}_y^\theta$. When $x\neq y$, it has positive probability of never hitting the origin. Therefore,
\[\lim_{N\rightarrow\infty}C_N(\theta)=\inf_{N}C_N(\theta)
=\sum_{x,y}q^\theta(x)q^\theta(y)\mathrm{e}^{V(x,y)}\hat{P}_x^\theta\!\times\!\hat{P}_y^\theta\left(s=\infty\right)>0.\]
By (\ref{aboo}), $G_M(0)=1$ for every $M$. Plugging it in (\ref{neis}), \[1=\sum_{k=0}^{N-2}B_k(0)+C_N(0).\] Taking $N\rightarrow\infty$ gives \begin{equation}\label{fisinial}
\sum_{k=0}^{\infty}B_k(0)<1.
\end{equation} 

We would like to show that \[B(\theta):=\sum_{k=0}^{\infty}B_k(\theta)\] is continuous in $\theta$ at $0$. Since $\theta\mapsto B_k(\theta)$ is continuous for each $k$, it suffices to argue that the tail of this sum is small, uniformly in $\theta$ in a neighborhood of $0$. Indeed,
\begin{align}
\sum_{k=N}^{\infty}B_k(\theta)\leq&\mathrm{e}^{\bar{V}}\sum_{x,y}q^\theta(x)q^\theta(y)\hat{P}_x^\theta\!\times\!\hat{P}_y^\theta\left(N\leq s<\infty\right)\nonumber\\=&\mathrm{e}^{\bar{V}}\hat{P}_o^\theta\!\times\!\hat{P}_o^\theta\left(N+1\leq s^+<\infty\right)\nonumber\\\leq&\mathrm{e}^{\bar{V}}\sum_{k=N+1}^{\infty}\hat{P}_o^\theta\!\times\!\hat{P}_o^\theta\left(X_k=Y_k\right)\label{faikcina}.
\end{align} Since $d\geq3$ and the covariance of $X_1-Y_1$ under $\hat{P}_o^\theta\!\times\!\hat{P}_o^\theta$ is a nonsingular matrix whose entries are continuous in $\theta$, the local CLT implies that the sum in (\ref{faikcina}) is the tail of a series which converges uniformly in $\theta$ in a neighborhood of $0$.

Now that we know $\theta\mapsto B(\theta)$ is continuous at $0$, recall (\ref{fisinial}) and see that there exists $\bar{\eta}>0$ such that for every $\theta\in\mathbb{R}^d$ with $|\theta|<\bar{\eta}$, $B(\theta)<1$. Letting $C(\theta):=\sup_{M}C_M(\theta)$, turn to (\ref{neis}) and conclude that \[\sup_{M\leq N}G_M(\theta)\leq \frac{C(\theta)}{1-B(\theta)}<\infty.\]Taking $N\rightarrow\infty$ gives the desired result.
\end{proof}

\subsection{Proof of Conjecture \ref{conjecture}}

From now on, consider $d\geq3$ and $\theta$ as in Lemma \ref{eltu}. For every $x\in\mathbb{Z}^d$ and $n,N\in\mathbb{Z}$ with $n<N$, recall (\ref{babalar}) and observe that $\mathbb{P}$-a.s.
\begin{align*}
u_N^\theta(\omega,n,x)&=E_{n,x}^\omega\left[\mathrm{e}^{\langle\theta,X_N-X_n\rangle-\Lambda_c(\theta)(N-n)}\right]\\
&=\sum_{y}\pi_{n,n+1}(x,y)\mathrm{e}^{\langle\theta,y-x\rangle}E_{n+1,y}^\omega\left[\mathrm{e}^{\langle\theta,X_N-X_{n+1}\rangle-\Lambda_c(\theta)(N-n)}\right]\\
&=\sum_{y}\pi_{n,n+1}(x,y)\mathrm{e}^{\langle\theta,y-x\rangle-\Lambda_c(\theta)}u_N^\theta(\omega,n+1,y).
\end{align*}
$\left(u_N^\theta(\cdot,n,x)\right)_{N>n}$ is a nonnegative martingale and $\mathbb{P}$-a.s.\ converges to some limit $u^\theta(\cdot,n,x)$ which satisfies
\begin{equation}\label{denk}
u^\theta(\omega,n,x)=\sum_{y}\pi_{n,n+1}(x,y)\mathrm{e}^{\langle\theta,y-x\rangle-\Lambda_c(\theta)}u^\theta(\omega,n+1,y).
\end{equation} By Lemma \ref{eltu}, $\left(u_N^\theta(\cdot,n,x)\right)_{N>n}$ is uniformly bounded in $L^2(\mathbb{P})$, and therefore the convergence takes place also in $L^2(\mathbb{P})$.

For every $x\in\mathbb{Z}^d$ and $n,N\in\mathbb{Z}$ with $n<N$, clearly $\left\|u_N^\theta(\cdot,n,x)\right\|_{L^1(\mathbb{P})}=1$. Since $\left(u_N^\theta(\cdot,n,x)\right)_{N>n}$ converges to $u^\theta(\cdot,n,x)$ in $L^2(\mathbb{P})$, $\left\|u^\theta(\cdot,n,x)\right\|_{L^1(\mathbb{P})}=1$ and $u^\theta(\cdot,n,x)\in L^2(\mathbb{P})$.
\[u_N^\theta(T_{n,x}\omega,0,0)=\frac{E_{o,o}^{T_{n,x}\omega}\left[\mathrm{e}^{\langle\theta,X_N-X_o\rangle}\right]}{\mathrm{e}^{\Lambda_c(\theta)N}}=\frac{E_{n,x}^\omega\left[\mathrm{e}^{\langle\theta,X_{N+n}-X_n\rangle}\right]}{\mathrm{e}^{\Lambda_c(\theta)N}}=u_{N+n}^\theta(\omega,n,x).\] holds $\mathbb{P}$-a.s. Taking $N\to\infty$,
\begin{align}
&u^\theta(T_{n,x}\omega,0,0)=u^\theta(\omega,n,x).\label{oldu}\\
&u^\theta(\omega):=u^\theta(\omega,0,0)\label{yeniu}
\end{align}
abbreviates the notation. Since $u_N^\theta(\cdot,0,0)$ is $\mathcal{B}_0^+$-measurable, it follows that $u^\theta$ is $\mathcal{B}_0^+$-measurable as well.

Using (\ref{oldu}) and (\ref{yeniu}), put (\ref{denk}) in the following form: $\mathbb{P}$-a.s.
\begin{equation}\label{bariz}
u^\theta(\omega)=\sum_{z\in U}\overline{\pi}(\omega,T_{1,z}\omega)\mathrm{e}^{\langle\theta,z\rangle-\Lambda_c(\theta)}u^\theta(T_{1,z}\omega).
\end{equation}

Finally, let us prove that $u^\theta>0$ holds $\mathbb{P}$-a.s. We already know that $u^\theta\geq0$ holds $\mathbb{P}$-a.s. Clearly, (\ref{bariz}) implies that $\left\{\omega:\ u^\theta(\omega)=0\right\}$ is invariant under $T_{1,z}$ for every $z\in U$. Since $\mathbb{P}$ is ergodic under shifts, $\mathbb{P}(u^\theta(\omega)=0)$ is either $0$ or $1$. But we know that $\left\|u^\theta(\cdot,n,x)\right\|_{L^1(\mathbb{P})}=1$, and therefore $\mathbb{P}(u^\theta(\omega)=0)=0$.

Define a new transition kernel $\overline{\pi}^\theta$ on $\Omega$ by a Doob $h$-transform: For every $z\in U$, $\mathbb{P}$-a.s.
\begin{equation}\label{doobcan}
\overline{\pi}^\theta(\omega,T_{1,z}\omega):= \overline{\pi}(\omega,T_{1,z}\omega)\frac{u^\theta(T_{1,z}\omega)}{u^\theta(\omega)}\mathrm{e}^{\langle\theta,z\rangle-\Lambda_c(\theta)}.
\end{equation} $\overline{\pi}^\theta$ induces a probability measure $P_{k,x}^{\theta,\omega}$ on particle paths starting at position $x$ at time $k$. Write $E_{k,x}^{\theta,\omega}$ to denote expectation under this measure.

\begin{proof}[Proof of Theorem \ref{AequalsQ}]

For $d\geq3$ and $\bar{\eta}$ as in Lemma \ref{eltu}, recall (\ref{babalar}) and observe that if $|\theta|<\bar{\eta}$, then
\begin{equation}
\lim_{n\rightarrow\infty}\frac{1}{n}\log E_{o,o}^\omega\left[\mathrm{e}^{\langle\theta,X_n\rangle}\right]=\Lambda_c(\theta)+\lim_{n\rightarrow\infty}\frac{1}{n}\log u_n^\theta(\omega)=\Lambda_c(\theta)\label{higherz}
\end{equation} because $\lim_{n\to\infty} u_n^\theta(\omega) = u^\theta(\omega)>0$ holds $\mathbb{P}$-a.s. Since  $\Lambda_c$ is strictly convex and $\xi_o=\nabla\Lambda_c(0)$, the set $\left\{\nabla\Lambda_c(\theta):\ |\theta|<\bar{\eta}\right\}$ is open and contains $\xi_o$. Hence, there exists $\eta>0$ such that for every $\xi\in\mathcal{D}$ with $|\xi-\xi_o|<\eta$, there is a unique $\theta$ satisfying $|\theta|<\bar{\eta}$ and $\xi=\nabla\Lambda_c(\theta)$. Because $\theta\mapsto\Lambda_c(\theta)$ is analytic, (\ref{higherz}) and the G\"{a}rtner-Ellis theorem (see \cite{DZ}, page 44) immediately imply the desired result.
\end{proof}

\section{Conditioning on the mean velocity}\label{Asection}

\subsection{Environment MC under the averaged measure}

%This is a special case of Section falan filan for RWRE diyebiliriz. The only difference is that now the functions $f$ depend on the environment. (Remark falan filan'in adini ver.) 

%For the full proofs, see Yilmaz de.

%Ondan sonra, well defined teoreminin ispatina direkt olarak this is a special case de. Bir iki yorumda bulun.

%Sonra, stationary teoremini ver.

%En son da, averaged conditioning teoremi icin de direkt olarak special case de ve bir iki yorumda bulun.

%anan.

%For any $n\in\mathbb{N}$ and $\theta\in\mathbb{R}^d$, since the environment is i.i.d.,
%\begin{equation}
%E_{o,o}\left[\mathrm{e}^{\langle\theta,X_n\rangle}\right]=\mathrm{e}^{n\Lambda_c(\theta)}\label{budur}
%\end{equation} with the notation in (\ref{fi}). By Cram\'{e}r's theorem, the LDP for the mean velocity of the particle holds under $P_{o,o}$ with the rate function $I_c$ given by
%\begin{equation}\label{Ia}
%I_c(\xi)=\sup_{\theta'}\left\{\langle\theta',\xi\rangle-\Lambda_c(\theta')\right\} = \langle\theta,\xi\rangle-\Lambda_c(\theta)
%\end{equation} where $\theta$ is the unique solution of $\xi=\nabla\Lambda_c(\theta)$. Due to our nearest-neighbor and ellipticity assumptions, \[\mathcal{D}=\left\{\xi\in\mathbb{R}^d\ :I_c(\xi)<\infty\right\}=\left\{(\xi^1,\ldots,\xi^d)\in\mathbb{R}^d\ :|\xi^1|+\cdots+|\xi^d|\leq1\right\}.\]

Fix any $\xi\in\mathcal{D}^o$. Recall Definition \ref{definemu}.
\begin{proposition}\label{welldefined}
$\bar{\mu}_\xi^\infty$ is well defined.
\end{proposition}

\begin{proof}
%For any $N,M,K\in\mathbb{N}$, take any $f$ as in Definition \ref{definemu}.
Let $L:=N+M+K+1$.
%Since $f(\cdot,(z_i)_{i\geq1})$ is independent of $(z_i)_{i>K}$ and $\mathcal{B}_{-N}^+\cap\mathcal{B}_M^-$-measurable for each $(z_i)_{i\geq1}$, $f(\cdot,(z_i)_{i\geq1})$ is independent of $(z_i)_{i>K+1}$ and $\mathcal{B}_{-(N+1)}^+\cap\mathcal{B}_{M+1}^-$-measurable for each $(z_i)_{i\geq1}$ as well. 
We need to show that (\ref{mucan}) does not change if we replace $N$ by $N+1$, $M$ by $M+1$, or $K$ by $K+1$.

Let us start with the argument for $N$.
\begin{align}
&E_{o,o}\left[\mathrm{e}^{\langle\theta,X_{L+1}\rangle-(L+1)\Lambda_c(\theta)}f(T_{N+1,X_{N+1}}\omega,(Z_{N+1+i})_{i\geq1})\right]\nonumber\\
&=\sum_{x}\mathbb{E}\left(E_{o,o}^\omega\left[\mathrm{e}^{\langle\theta,X_1\rangle-\Lambda_c(\theta)}, X_1=x\right]\right.\label{bagimsiz}\\&\ \ \ \ \ \ \ \times\left.E_{1,x}^\omega\left[\mathrm{e}^{\langle\theta,X_{L+1}-X_1\rangle-L\Lambda_c(\theta)}f(T_{N+1,X_{N+1}}\omega,(Z_{N+1+i})_{i\geq1})\right]\right)\nonumber\\
&=\sum_{x}E_{o,o}\left[\mathrm{e}^{\langle\theta,X_1\rangle-\Lambda_c(\theta)}, X_1=x\right]\label{gozum}\\&\ \ \ \ \times E_{1,x}\left[\mathrm{e}^{\langle\theta,X_{L+1}-X_1\rangle-L\Lambda_c(\theta)}f(T_{N+1,X_{N+1}}\omega,(Z_{N+1+i})_{i\geq1})\right]\nonumber\\
&=\sum_{x}E_{o,o}\left[\mathrm{e}^{\langle\theta,X_1\rangle-\Lambda_c(\theta)}, X_1=x\right]\label{kaydir}\\&\ \ \ \ \times E_{o,o}\left[\mathrm{e}^{\langle\theta,X_L\rangle-L\Lambda_c(\theta)}f(T_{N,X_N}\omega,(Z_{N+i})_{i\geq1})\right]\nonumber\\
&=\,E_{o,o}\left[\mathrm{e}^{\langle\theta,X_L\rangle-L\Lambda_c(\theta)}f(T_{N,X_N}\omega,(Z_{N+i})_{i\geq1})\right]\nonumber
\end{align}
holds. Note that each term of the sum in (\ref{bagimsiz}) is the $\mathbb{P}$-expectation of two random variables; the first one is $\mathcal{B}_0^-$-measurable and the second one is $\mathcal{B}_1^+$-measurable. Use this independence to obtain (\ref{gozum}). (\ref{kaydir}) follows from the stationarity of $\mathbb{P}$ under shifts. Hence, (\ref{mucan}) does not change if $N$ is replaced by $N+1$.

Similarly, if $M$ is replaced by $M+1$ in (\ref{mucan}),
\begin{align}
&E_{o,o}\left[\mathrm{e}^{\langle\theta,X_{L+1}\rangle-(L+1)\Lambda_c(\theta)}f(T_{N,X_N}\omega,(Z_{N+i})_{i\geq1})\right]\nonumber\\
&=\sum_{x}\mathbb{E}\left(E_{o,o}^\omega\left[\mathrm{e}^{\langle\theta,X_L\rangle-L\Lambda_c(\theta)}f(T_{N,X_N}\omega,(Z_{N+i})_{i\geq1}), X_L=x\right]\right.\label{baris}\\
&\ \ \ \ \ \ \ \times\left.E_{L,x}^\omega\left[\mathrm{e}^{\langle\theta,X_{L+1}-X_L\rangle-\Lambda_c(\theta)}\right]\right)\nonumber\\
&=\sum_{x}E_{o,o}\left[\mathrm{e}^{\langle\theta,X_L\rangle-L\Lambda_c(\theta)}f(T_{N,X_N}\omega,(Z_{N+i})_{i\geq1}), X_L=x\right]\nonumber\\
&\ \ \ \ \ \ \ \times E_{L,x}\left[\mathrm{e}^{\langle\theta,X_{L+1}-X_L\rangle-\Lambda_c(\theta)}\right]\nonumber\\
&=E_{o,o}\left[\mathrm{e}^{\langle\theta,X_L\rangle-L\Lambda_c(\theta)}f(T_{N,X_N}\omega,(Z_{N+i})_{i\geq1})\right]\nonumber
\end{align}
where each term of the sum in (\ref{baris}) is the $\mathbb{P}$-expectation of two random variables; the first one is $\mathcal{B}_{L-1}^-$-measurable and the second one is $\mathcal{B}_L^+$-measurable.

The argument for $K$ is the same as the one for $M$. 
\end{proof}

\begin{proposition}\label{stationary}
$\bar{\mu}_\xi^\infty$ induces a stationary process $\mu_\xi^\infty$ with values in $\Omega$.
\end{proposition}

\begin{proof}
Define $\bar{S}: \Omega\times U^{\mathbb{N}}\rightarrow\Omega\times U^{\mathbb{N}}$ by $\bar{S}: \left(\omega,(z_i)_{i\geq1}\right)\mapsto \left(T_{1,z_1}\omega,(z_i)_{i\geq2}\right)$ and the projection map $\Psi:\Omega\times U^{\mathbb{N}}\rightarrow\Omega$ by $\Psi:\left(\omega,(z_i)_{i\geq1}\right)\mapsto\omega.$ Let us show that $\bar{\mu}_\xi^\infty$ is invariant under $\bar{S}$. For every $N,M$ and $K\in\mathbb{N}$, and any $f$ as in Definition \ref{definemu}, $f\circ\bar{S}\left(\omega,(z_i)_{i\geq1}\right)=f\left(T_{1,z_1}\omega,(z_i)_{i\geq2}\right)$ is $\mathcal{B}_{-(N-1)}^+\cap\mathcal{B}_{M+1}^-$-measurable and independent of $(z_i)_{i>K+1}$. By definition,
\begin{align*}
&\int\!\! f\!\circ\!\bar{S}\mathrm{d}\bar{\mu}_\xi^\infty=E_{o,o}\!\left[\mathrm{e}^{\langle\theta,X_{N+M+K+2}\rangle - (N+M+K+2)\Lambda_c(\theta)}f\!\circ\!\bar{S}(T_{N-1,X_{N-1}}\omega,(Z_{(N-1)+i})_{i\geq1})\right]\\
&=E_{o,o}\!\left[\mathrm{e}^{\langle\theta,X_{N+M+K+2}\rangle - (N+M+K+2)\Lambda_c(\theta)}f(T_{N,X_N}\omega,(Z_{N+i})_{i\geq1})\right]=\int\!\!f\mathrm{d}\bar{\mu}_\xi^\infty.
\end{align*}

Therefore, under $\bar{\mu}_\xi^\infty$, $\left(\Psi\circ\bar{S}^k(\cdot)\right)_{k\geq0}$ extends to a stationary process taking values in $\Omega$, whose distribution is denoted by $\mu_\xi^\infty$.
\end{proof}

\begin{proof}[Proof of Theorem \ref{averagedconditioning}]

As noted in Remark \ref{fehimbey}, $(n,X_n)_{n\geq0}$ can be viewed as RWRE on $\mathbb{Z}^{d+1}$. It is clearly non-nestling in the ``time" direction, and the associated regeneration times satisfy $\tau_m=m$.
Therefore, Theorem \ref{averagedconditioning} is almost a special case of Theorem \ref{averagedconditioningsh}. But, there is a slight difference:

Recall (\ref{sabrialtintas}). In Theorem \ref{averagedconditioningsh}, we consider $f:U^{\mathbb{N}}\rightarrow\mathbb{R}$ such that $f((z_i)_{i\geq1})$ is independent of $(z_i)_{i>K}$, and therefore $(G_j,G_{j+K},\ldots)$ is an i.i.d.\ sequence under $P_o$. In Theorem \ref{averagedconditioning}, we instead consider $f:\Omega\times U^{\mathbb{N}}\rightarrow\mathbb{R}$ such that $f(\cdot,(z_i)_{i\geq1})$ is independent of $(z_i)_{i>K}$ and $\mathcal{B}_{-N}^+\cap\mathcal{B}_M^-$-measurable for each $(z_i)_{i\geq1}$, and this time $(\tilde{G}_j,\tilde{G}_{j+L},\ldots)$ is an i.i.d.\ sequence under $P_{o,o}$, where $\tilde{G}_j:=g(T_{j,X_j}\omega,(Z_{j+i})_{i\geq1})$ and $L:=N+M+K+1$. The first part of the proof of Theorem \ref{averagedconditioningsh} carries over with this minor modification and
\begin{equation}\label{yaniyorumbee}
\limsup_{n\rightarrow\infty}\frac{1}{n}\log E_{o,o}\left[\mathrm{e}^{\langle\theta,X_n\rangle-n\Lambda_c(\theta)},\ \left|\int\!\! f\mathrm{d}\bar{\nu}_{n,X}^\infty-\int\!\! f\mathrm{d}\bar{\mu}_\xi^\infty\right|>\epsilon\right]=:\gamma<0.
\end{equation} The desired result is obtained by a standard change of measure argument given in the last part of the proof of Theorem \ref{averagedconditioningsh}. See \cite{YilmazSpaceTime} for the complete proof.
\end{proof}

\subsection{Environment MC under the quenched measure}

\begin{proof}[Proof of Theorem \ref{strong}]
Let $\theta$ be the unique solution of $\xi=\nabla\Lambda_c(\theta)$. Fix $\alpha>0$. Recall (\ref{yaniyorumbee}). For every $n\in\mathbb{N}$, the event $B_n'\subset\Omega$ is defined by
\begin{align*}
B_n':&=\left\{\omega:\ E_{o,o}^\omega\left[\mathrm{e}^{\langle\theta,X_n\rangle-n\Lambda_c(\theta)},\ |\langle f,\bar{\nu}_{n,X}^\infty - \bar{\mu}_\xi^\infty\rangle|>\epsilon\right]>\mathrm{e}^{n(\gamma+\alpha)}\right\}.\\
\mathbb{P}\left(B_n'\right)&\leq\int_{B_n'}E_{o,o}^\omega\left[\mathrm{e}^{\langle\theta,X_n\rangle-n\Lambda_c(\theta)},\ |\langle f,\bar{\nu}_{n,X}^\infty - \bar{\mu}_\xi^\infty\rangle|>\epsilon\right]\mathrm{e}^{-n(\gamma+\alpha)}\mathrm{d}\mathbb{P}\\
&\leq E_{o,o}\left[\mathrm{e}^{\langle\theta,X_n\rangle-n\Lambda_c(\theta)},\ |\langle f,\bar{\nu}_{n,X}^\infty - \bar{\mu}_\xi^\infty\rangle|>\epsilon\right]\mathrm{e}^{-n(\gamma+\alpha)}.
\end{align*}
Therefore, $\limsup_{n\rightarrow\infty}\frac{1}{n}\log\mathbb{P}\left(B_n'\right)\leq -\alpha$, and in particular $\sum_{n=1}^{\infty}\mathbb{P}\left(B_n'\right)<\infty$. By the Borel-Cantelli lemma, $\mathbb{P}\left(B_n'\ \text{i.o.}\right)=0$. In other words, $\mathbb{P}$-a.s. \[E_{o,o}^\omega\left[\mathrm{e}^{\langle\theta,X_n\rangle-n\Lambda_c(\theta)},\ |\langle f,\bar{\nu}_{n,X}^\infty - \bar{\mu}_\xi^\infty\rangle|>\epsilon\right]\leq\mathrm{e}^{n(\gamma+\alpha)}\] for sufficiently large $n$. Thus, \[\limsup_{n\rightarrow\infty}\frac{1}{n}\log E_{o,o}^\omega\left[\mathrm{e}^{\langle\theta,X_n\rangle-n\Lambda_c(\theta)},\ |\langle f,\bar{\nu}_{n,X}^\infty - \bar{\mu}_\xi^\infty\rangle|>\epsilon\right]\leq\gamma+\alpha.\] Since $\alpha>0$ is arbitrary, \[\limsup_{n\rightarrow\infty}\frac{1}{n}\log E_{o,o}^\omega\left[\mathrm{e}^{\langle\theta,X_n\rangle-n\Lambda_c(\theta)},\ |\langle f,\bar{\nu}_{n,X}^\infty - \bar{\mu}_\xi^\infty\rangle|>\epsilon\right]\leq\gamma.\]

Let us now finish the proof of the theorem:
\begin{align*}
&\limsup_{n\rightarrow\infty}\frac{1}{n}\log P_{o,o}^\omega\left(|\langle f,\bar{\nu}_{n,X}^\infty - \bar{\mu}_\xi^\infty\rangle|>\epsilon\,\left|\,|\frac{X_n}{n}-\xi|\leq\delta\right.\right)\\
\leq&\limsup_{n\rightarrow\infty}\frac{1}{n}\log P_{o,o}^\omega\left(|\langle f,\bar{\nu}_{n,X}^\infty - \bar{\mu}_\xi^\infty\rangle|>\epsilon\,,\,|\frac{X_n}{n}-\xi|\leq\delta\right)\\
&\quad - \liminf_{n\rightarrow\infty}\frac{1}{n}\log P_{o,o}^\omega\left(|\frac{X_n}{n}-\xi|\leq\delta\right)\\
\leq&\limsup_{n\rightarrow\infty}\frac{1}{n}\log E_{o,o}^\omega\left[\mathrm{e}^{\langle\theta,X_n\rangle},|\langle f,\bar{\nu}_{n,X}^\infty - \bar{\mu}_\xi^\infty\rangle|>\epsilon\,,\,|\frac{X_n}{n}-\xi|\leq\delta\right]\\
&\quad - \langle\theta,\xi\rangle + I_c(\xi) + |\theta|\delta\\
\leq&\limsup_{n\rightarrow\infty}\frac{1}{n}\log E_{o,o}^\omega\left[\mathrm{e}^{\langle\theta,X_n\rangle - n\Lambda_c(\theta)}\,,\,|\langle f,\bar{\nu}_{n,X}^\infty - \bar{\mu}_\xi^\infty\rangle|>\epsilon\right] + |\theta|\delta\leq\gamma + |\theta|\delta<0
\end{align*}
when $\delta>0$ is sufficiently small. The above estimate uses the fact that the quenched LDP holds in a neighborhood of $\xi$ with rate $I_c(\xi)=\langle\theta,\xi\rangle-\Lambda_c(\theta)$ at $\xi$, which is true by hypothesis.
\end{proof}

For $d\geq3$ and $|\xi-\xi_o|<\eta$ with $\eta$ as in Theorem \ref{AequalsQ}, we can put $\bar{\mu}_\xi^\infty$ in a nicer form. For every $N,M,K\in\mathbb{N}$ and any $f$ as in Definition \ref{definemu}, set $L:=N+M+K+1$.
\begin{align}
\int\!\!f\mathrm{d}\bar{\mu}_\xi^\infty&=E_{o,o}\left[\mathrm{e}^{\langle\theta,X_{L}\rangle - L\Lambda_c(\theta)}f(T_{N,X_N}\omega,(Z_{N+i})_{i\geq1})\right]\nonumber\\
&=\sum_{x}\mathbb{E}\!\left(E_{o,o}^\omega\left[\mathrm{e}^{\langle\theta,X_{L}\rangle - L\Lambda_c(\theta)}f(T_{N,X_N}\omega,(Z_{N+i})_{i\geq1}),X_L=x\right]\right)\!\mathbb{E}\!\left(u^\theta(T_{L,x}\omega)\right)\nonumber\\
&=\sum_{x}\mathbb{E}\left(E_{o,o}^\omega\left[\mathrm{e}^{\langle\theta,X_{L}\rangle - L\Lambda_c(\theta)}u^\theta(T_{L,x}\omega)f(T_{N,X_N}\omega,(Z_{N+i})_{i\geq1}),X_L=x\right]\right)\nonumber\\
&=\mathbb{E}\left(u^\theta(\omega)E_{o,o}^\omega\left[\mathrm{e}^{\langle\theta,X_{L}\rangle - L\Lambda_c(\theta)}\frac{u^\theta(T_{L,X_L}\omega)}{u^\theta(\omega)}f(T_{N,X_N}\omega,(Z_{N+i})_{i\geq1})\right]\right)\nonumber\\
&=\mathbb{E}\left(u^\theta(\omega)E_{o,o}^{\theta,\omega}\left[f(T_{N,X_N}\omega,(Z_{N+i})_{i\geq1})\right]\right)\label{genc}
\end{align} holds since $\mathbb{E}\left(u^\theta(T_{L,x}\cdot)\right)=1$ and $u^\theta(T_{L,x}\cdot)$ is $\mathcal{B}_L^+$-measurable. Note that (\ref{genc}) is independent of $M$ and $K$. This immediately implies that the marginal $\mu_\xi^1$ of $\mu_\xi^\infty$ is absolutely continuous relative to $\mathbb{P}$ on every $\mathcal{B}_{-N}^+$. Here is how: Fix $N\in\mathbb{N}$. For any $M\in\mathbb{N}$ and any bounded $\mathcal{B}_{-N}^+\cap\mathcal{B}_{M}^-$-measurable $h:\Omega\rightarrow\mathbb{R}$,
\begin{align*}
\int h\mathrm{d}\mu_\xi^1&=\mathbb{E}\left(u^\theta(\omega)E_{o,o}^{\theta,\omega}\left[h(T_{N,X_N}\omega)\right]\right)\\
&\leq\left\|u^\theta\right\|_{L^2(\mathbb{P})}\left\|E_{o,o}^{\theta,\omega}\left[h(T_{N,X_N}\omega)\right]\right\|_{L^2(\mathbb{P})}\\
&\leq\left\|u^\theta\right\|_{L^2(\mathbb{P})}\left\|\sum_{|x|\leq N}h(T_{N,x}\omega)\right\|_{L^2(\mathbb{P})}\\
&\leq (2N+1)^d\left\|u^\theta\right\|_{L^2(\mathbb{P})}\left\|h\right\|_{L^2(\mathbb{P})}.
\end{align*} Since such functions are dense in $L^2(\Omega,\mathcal{B}_{-N}^+,\mathbb{P})$, it follows by the Riesz representation theorem that
\begin{equation}\label{riesz}
\left.\frac{\mathrm{d}\mu_{\xi}^1}{\mathrm{d}\mathbb{P}}\right|_{\mathcal{B}_{-N}^+}\in L^2(\mathbb{P}).
\end{equation}

\begin{proof}[Proof of Theorem \ref{doob}]

The function $u^\theta$ is defined in (\ref{yeniu}). Recall (\ref{bariz}) and (\ref{doobcan}). For every $N,K\in\mathbb{N}$, take any bounded $f:\Omega^{K+1}\rightarrow\mathbb{R}$ and $g:\Omega^{\mathbb{N}}\rightarrow\mathbb{R}$ such that \[f\left(\omega,T_{1,z_1}\omega,\ldots,T_{K,z_1+\cdots+z_K}\omega\right)g\left(T_{K,z_1+\cdots+z_K}\omega,T_{K+1,z_1+\cdots+z_{K+1}}\omega,\ldots\right)\] is $\mathcal{B}_{-N}^+$-measurable for any $(z_i)_{i\geq 1}$. Then,
\begin{align*}
&\int\!\! f(\omega_1,\ldots,\omega_{K+1})g(\omega_{K+1},\omega_{K+2},\ldots)\,\mathrm{d}\mu_\xi^\infty(\omega_1,\omega_2,\ldots)\\
=&\int\!\! f\left(\omega,T_{1,z_1}\omega,\ldots,T_{K,z_1+\cdots+z_K}\omega\right)g\left(T_{K,z_1+\cdots+z_K}\omega,T_{K+1,z_1+\cdots+z_{K+1}}\omega,\ldots\right)\mathrm{d}\bar{\mu}_\xi^\infty\\
=&\mathbb{E}\left(u^\theta(\omega)E_{o,o}^{\theta,\omega}\left[f(T_{N,X_N}\omega,\ldots,T_{N+K,X_{N+K}}\omega)g(T_{N+K,X_{N+K}}\omega,\ldots)\right]\right)\\
=&\mathbb{E}\left(u^\theta(\omega)E_{o,o}^{\theta,\omega}\left[f(T_{N,X_N}\omega,\ldots,T_{N+K,X_{N+K}}\omega)E_{N+K,X_{N+K}}^{\theta,\omega}\left[g(T_{N+K,X_{N+K}}\omega,\ldots)\right]\right]\right)\\
=&\int\!\! f(\omega_1,\ldots,\omega_{K+1})E_{o,o}^{\theta,\omega_{K+1}}\left[g(\omega_{K+1},T_{1,X_1}\omega_{K+1},\ldots)\right]\mathrm{d}\mu_\xi^\infty
\end{align*} by (\ref{genc}) and the Markov property. This proves that $\mu_\xi^\infty$ is indeed a Markov process with state space $\Omega$ and transition kernel $\overline{\pi}^\theta$.

$\mu_\xi^\infty$ is a stationary process by Proposition \ref{stationary}. Hence, its marginal $\mu_\xi^1$ is an invariant measure for $\overline{\pi}^\theta$. Since $\mu_\xi^1$ is absolutely continuous relative to $\mathbb{P}$ on every $\mathcal{B}_{-N}^+$ by (\ref{riesz}), it follows that $\mu_\xi^1$ is the unique invariant measure for $\overline{\pi}^\theta$ with that absolute continuity property (see \cite{Firas}).

\end{proof}

\backmatter

\end{document}